%% file: permutrees.tex
\documentclass[a4paper]{amsart}

\usepackage{enumerate, amsmath, amsfonts, amssymb, amsthm, wasysym, graphics, graphicx, xcolor, url, hyperref, hypcap, a4wide, stmaryrd, shuffle, xargs, multicol, overpic, pdflscape, multirow, hvfloat, minibox, accents, etoolbox, dsfont}
\hypersetup{colorlinks=true, citecolor=darkblue, linkcolor=darkblue, urlcolor=darkblue}
\usepackage[bottom]{footmisc}
\usepackage{tikz}\usetikzlibrary{trees,snakes,shapes,arrows,matrix,calc}
\graphicspath{{figures/}{figures/nodes/}}
\makeatletter
\def\input@path{{figures/}}
\makeatother

\title{Permutrees}

\author{Vincent Pilaud}
\author{Viviane Pons}

\address[VPi]{CNRS \& LIX, \'Ecole Polytechnique, Palaiseau}
\email{vincent.pilaud@lix.polytechnique.fr}
\urladdr{http://www.lix.polytechnique.fr/~pilaud/}

\address[VPo]{LRI, Univ. Paris-Sud - CNRS - Centrale Supelec - Univ. Paris-Saclay}
\email{viviane.pons@lri.fr}
\urladdr{https://www.lri.fr/~pons/}

\thanks{VPi~was partially supported by the French ANR grant SC3A~(15\,CE40\,0004\,01).}

\newtheorem{theorem}{Theorem}
\newtheorem{corollary}[theorem]{Corollary}
\newtheorem{proposition}[theorem]{Proposition}
\newtheorem{lemma}[theorem]{Lemma}
\newtheorem{definition}[theorem]{Definition}

\theoremstyle{definition}
\newtheorem{example}[theorem]{Example}
\newtheorem{remark}[theorem]{Remark}

\newcommand{\R}{\mathbb{R}} 
\newcommand{\N}{\mathbb{N}} 
\newcommand{\Z}{\mathbb{Z}} 
\newcommand{\HH}{\mathbb{H}} 
\newcommand{\fS}{\mathfrak{S}} 
\newcommand{\fP}{\mathfrak{P}} 
\renewcommand{\b}[1]{\mathbf{#1}} 

\newcommand{\set}[2]{\left\{ #1 \;\middle|\; #2 \right\}} 
\newcommand{\bigset}[2]{\big\{ #1 \;|\; #2 \big\}} 
\newcommand{\biggset}[2]{\bigg\{ #1 \;\bigg|\; #2 \bigg\}} 
\newcommand{\ssm}{\smallsetminus} 
\newcommand{\dotprod}[2]{\langle \, #1 \; | \; #2 \, \rangle} 
\newcommand{\one}{{1\!\!1}} 
\newcommand{\eqdef}{\mbox{\,\raisebox{0.2ex}{\scriptsize\ensuremath{\mathrm:}}\ensuremath{=}\,}} 
\newcommand{\polar}{^\diamond} 
\newcommand{\blackPolar}{^\blacklozenge} 
\renewcommand{\implies}{\Rightarrow} 
\newcommand{\sep}{|} 
\newcommand{\freeSep}{\freeGap\hspace{-.92mm}|} 
\newcommand{\Id}{\mathrm{Id}} 
\newcommand{\sign}{\mathrm{sign}} 

\newcommandx{\graphG}[1][1=G]{\mathrm{#1}} 
\newcommandx{\tree}[1][1=T]{\mathrm{#1}} 
\newcommand{\ground}{\mathrm{V}} 
\newcommand{\decoration}{\delta} 
\newcommand{\Decorations}{\{\noneCirc{}, \downCirc{}, \upCirc{}, \upDownCirc{}\}} 
\newcommand{\less}{\preccurlyeq} 
\newcommand{\pdecoration}{\decoration_p} 
\newcommand{\vdecoration}{\decoration_v} 
\newcommand{\generatingTree}{\mathcal{T}} 
\newcommand{\generatingTreeSchroder}{\mathcal{S}} 
\newcommand{\level}{m} 
\newcommand{\up}[1]{\overline{#1}}
\newcommand{\upr}[1]{{\red \up{#1}}}
\newcommand{\upb}[1]{{\blue \up{#1}}}
\newcommand{\down}[1]{\underline{#1}}

\newcommand{\downb}[1]{{\blue \down{#1}}}
\newcommand{\updown}[1]{\up{\down{#1}}}
\newcommand{\updownr}[1]{{\red \updown{#1}}}

\newcommand{\linearExtensions}{\mathcal{L}} 
\newcommand{\circDecoration}{{\text{\tiny$\bigcirc$}}} 
\newcommand{\squareDecoration}{{\text{\raisebox{-.05cm}{$\square$}}}} 
\newcommand{\edgecut}[2]{\left( #1 \;\middle\|\; #2 \right)} 
\newcommand{\cuts}{\mathsf{EC}} 
\newcommand{\decreasingCuts}{\mathsf{DEC}} 
\newcommand{\dash}{\text{-}}
\newcommand{\Permutrees}{\mathcal{PT}} 
\newcommand{\SchrPermutrees}{\mathrm{SchrPT}} 
\newcommand{\factCatalan}[1]{\text{\bf C}(#1)} 
\newcommand{\factSchroder}[1]{\text{\bf S}(#1)} 

\newcommand{\arcDiagrams}{\mathcal{A}} 
\newcommand{\asc}{\mathsf{asc}} 
\newcommand{\desc}{\mathsf{desc}} 

\newcommandx{\Perm}[1][1=n]{\mathds{P}\mathrm{erm}(#1)} 
\newcommandx{\Asso}[1][1=n]{\mathds{A}\mathrm{sso}(#1)} 
\newcommandx{\Para}[1][1=n]{\mathds{P}\mathrm{ara}(#1)} 
\newcommandx{\Zono}[1][1=\decoration]{\mathds{Z}\mathrm{ono}(#1)} 
\newcommandx{\Permutreehedron}[1][1=\decoration]{\mathds{PT}(#1)} 
\newcommand{\Cone}{\mathrm{C}} 
\newcommand{\fan}{\mathcal{F}} 
\newcommandx{\Fan}[1][1=\decoration]{\fan(#1)} 
\newcommand{\Hyp}{\b{H}^=} 
\newcommand{\HS}{\b{H}^\ge} 
\newcommand{\face}{\mathsf{F}} 

\DeclareMathOperator{\conv}{conv} 
\DeclareMathOperator{\cone}{cone} 
\DeclareMathOperator{\coinv}{coinv} 
\newcommand{\meet}{\wedge} 
\newcommand{\join}{\vee} 
\newcommand{\projDown}{\pi_{\!\downarrow\!}} 
\newcommand{\projUp}{\pi^{\!\uparrow\!}} 
\newcommand{\verticalSymmetry}[1]{#1^{\;\mathbin{\tikz [thin, baseline=-0.2em] \draw [<->] (0em,-.3em) -- (0em,.3em);}}\,} 
\newcommand{\horizontalSymmetry}[1]{#1^{\mathbin{\tikz [baseline=-0.2em] \draw [<->] (-.3em,0em) -- (.3em,0em);}}} 
\newcommand{\horizontalVerticalSymmetry}[1]{#1^{\mathbin{\tikz [baseline=-0.2em] {\draw [<->] (0em,-.3em) -- (0em,.3em); \draw [<->] (-.4em,0em) -- (.4em,0em);}}}} 

\newcommand{\FQSym}{\mathsf{FQSym}} 
\newcommand{\PermutreeAlgebra}{\mathsf{PT}} 
\newcommand{\SchrPermutreeAlgebra}{\mathsf{SchrPT}} 
\newcommand{\OrdPart}{\mathsf{OrdPart}} 
\newcommand{\sfx}{\mathsf{x}} 
\newcommand{\sfy}{\mathsf{y}} 
\newcommand{\sfz}{\mathsf{z}} 
\newcommand{\integerPointTransform}{\mathbb{Z}} 

\newcommand{\product}{\cdot} 
\newcommand{\coproduct}{\triangle} 
\newcommand{\shiftedShuffle}{\,\bar\shuffle\,} 
\newcommand{\convolution}{\star} 
\newcommand{\underprod}[2]{{#1}\backslash{#2}} 
\newcommand{\overprod}[2]{{#1}\slash{#2}} 

\newcommand{\F}{\mathbb{F}} 
\newcommand{\G}{\mathbb{G}} 
\newcommand{\HFQSym}{\mathbb{H}} 
\newcommand{\EFQSym}{\mathbb{E}} 
\newcommand{\PPT}{\mathbb{P}} 
\newcommand{\QPT}{\mathbb{Q}} 
\newcommand{\HPT}{\mathbb{H}} 
\newcommand{\EPT}{\mathbb{E}} 

\newcommand{\permutreeCorresp}{\Theta} 
\newcommand{\SchroderPermutreeCorresp}{\Theta^\star} 
\newcommandx{\surjection}[2][1=\decoration, 2=\decoration']{\Psi_{#1}^{#2}} 
\newcommandx{\surjectionSchroder}[2][1=\decoration, 2=\decoration']{{\Psi^\star}_{#1}^{#2}} 
\newcommand{\PSymbol}{\mathbf{P}} 
\newcommand{\QSymbol}{\mathbf{Q}} 
\newcommand{\SchroderPSymbol}{\mathbf{P}^\star} 

\newcommand{\fref}[1]{Figure~\ref{#1}} 
\newcommand{\ie}{\textit{i.e.}~} 
\newcommand{\eg}{\textit{e.g.}~} 
\definecolor{darkblue}{rgb}{0,0,0.7} 
\newcommand{\darkblue}{\color{darkblue}} 
\newcommand{\red}{\color{red}} 
\newcommand{\blue}{\color{blue}} 
\newcommand{\defn}[1]{\emph{\darkblue #1}} 
\usepackage{todonotes}

\newcommand{\para}[1]{\medskip\noindent\framebox{\textsc{#1}}} 

\newcommand{\includeSymbol}[1]{\ensuremath{%
	\mathchoice
		{\raisebox{-.7mm}{\includegraphics[height=2.2ex]{#1}}}	
		{\raisebox{-.7mm}{\includegraphics[height=2.2ex]{#1}}}
		{\raisebox{-.6mm}{\includegraphics[height=1.6ex]{#1}}}
		{\raisebox{-.5mm}{\includegraphics[height=1ex]{#1}}}
}}
\robustify{\includeSymbol}
\newcommand{\noneCirc}{\includeSymbol{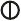}}
\newcommand{\upCirc}{\includeSymbol{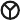}}
\newcommand{\downCirc}{\includeSymbol{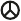}}
\newcommand{\upDownCirc}{\includeSymbol{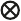}}
\newcommand{\noneSquare}{\includeSymbol{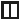}}
\newcommand{\upSquare}{\includeSymbol{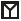}}
\newcommand{\downSquare}{\includeSymbol{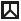}}
\newcommand{\upDownSquare}{\includeSymbol{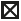}}

\usepackage{tikz}
\usetikzlibrary{positioning,arrows,matrix,chains,shadows,shapes}
\tikzstyle{Point} = [fill, radius=0.08]
\newcommand{\bannedGap}{}
\newcommand{\freeGap}{\textcolor{blue}{{\scalebox{.5}{$\bullet$}}}}

\makeatletter
\def\l@section{\@tocline{1}{3pt}{0pc}{}{}}
\makeatother
\let\oldtocpart=\tocpart
\renewcommand{\tocpart}[2]{\hspace{0em}\bf\large\oldtocpart{#1}{#2}}
\let\oldtocsection=\tocsection
\renewcommand{\tocsection}[2]{\hspace{0em}\bf\oldtocsection{#1}{#2}}

\begin{document}

\begin{abstract}
We introduce permutrees, a unified model for permutations, binary trees, Cambrian trees and binary sequences. On the combinatorial side, we study the rotation lattices on permutrees and their lattice homomorphisms, unifying the weak order, Tamari, Cambrian and boolean lattices and the classical maps between them. On the geometric side, we provide both the vertex and facet descriptions of a polytope realizing the rotation lattice, specializing to the permutahedron, the associahedra, and certain graphical zonotopes. On the algebraic side, we construct a Hopf algebra on permutrees containing the known Hopf algebraic structures on permutations, binary trees, Cambrian trees, and binary sequences.
\end{abstract}

\vspace*{-1cm}
\maketitle

\vspace{-.5cm}
\tableofcontents
\vspace{-.5cm}


\section{Introduction}
\label{sec:intro}

Binary words, binary trees, and permutations are three combinatorial families that share a common pattern linking their combinatorics to geometry and algebra. For example, the family of binary words of size $n$ is naturally endowed with the boolean lattice structure. Its Hasse diagram corresponds to the skeleton of an $n$-dimensional hypercube. Finally, all binary words index the basis of a Hopf algebra~\cite{GelfandKrobLascouxLeclercRetakhThibon} whose product is encoded in the boolean lattice (the product of two basis elements is given by a sum over an interval in the boolean lattice).
We find a similar scheme for the other families which we summarize in \fref{fig:objectstable}. 
Recently, the family of Cambrian trees was also shown to share a similar pattern: Cambrian trees generalize the notion of binary trees, they are naturally endowed with N.~Reading's (type~$A$) Cambrian lattice structure~\cite{Reading-CambrianLattices}, they correspond to the vertices of C.~Hohlweg and C.~Lange's associahedra~\cite{HohlwegLange}, and they index the basis of G.~Chatel and V.~Pilaud's Cambrian algebra~\cite{ChatelPilaud}.

All these families are related through deep structural properties from all three aspects: combinatorics, geometry, and algebra. As lattices, the boolean lattice on binary sequences is both a sub- and a quotient lattice of the Tamari lattice on binary trees, itself a sub- and quotient lattice of the weak order on permutations~\cite{BjornerWachs, Reading-CambrianLattices}. As polytopes, the cube contains J.-L.~Loday's associahedron~\cite{Loday} which in turn contains the permutahedron~\cite[Lecture~0]{Ziegler}. And as algebras, the descent Hopf algebra on binary sequences of~\cite{GelfandKrobLascouxLeclercRetakhThibon} is a sub- and quotient Hopf algebra of J.-L.~Loday and M.~Ronco's Hopf algebra on binary trees~\cite{LodayRonco, HivertNovelliThibon-algebraBinarySearchTrees}, itself a sub- and quotient algebra of C.~Malvenuto and C.~Reutenauer's Hopf algebra on permutations~\cite{MalvenutoReutenauer, DuchampHivertThibon}. More generally, these exemples led to the development of necessary and sufficient conditions to obtain combinatorial Hopf algebras from lattice congruences of the weak order~\cite{Reading-HopfAlgebras} and from rewriting rules on monoids~\cite{HivertNovelliThibon-algebraBinarySearchTrees, Priez}.

All these different aspects have been deeply studied but until now, these families have been considered as \defn{different} kind of objects. In this paper, we unify all these objects under a \defn{unique combinatorial definition} containing structural, geometric, and algebraic information. A given family type (binary words, binary trees, Cambrian trees, permutations) can then be encoded by a special \defn{decoration} on $[n]$. Our definition also allows for \defn{interpolations} between the known families: we obtain combinatorial objects that are structurally ``half'' binary trees and ``half'' permutations. This leads in particular to new lattice structures (see Figures \ref{fig:permutreeLattices} and~\ref{fig:fibersPermutreeCongruences}), to new polytopes (see Figures~\ref{fig:permutreehedron} and~\ref{fig:permutreehedra}) and to new combinatorial Hopf algebras (see Section~\ref{sec:permutreeHopfAlgebra}).

We call our new objects \defn{permutrees}. They are labeled and oriented trees where each vertex can have one or two parents and one or two children, and with local rules around each vertex similar to the classical rule for binary trees (see Defintion~\ref{def:permutree} for a precise statement). We will explore in particular the following features of permutrees:
\begin{description}
\item[Combinatorics] We describe a natural insertion map from (decorated) permutations to permutrees similar to the binary tree insertion. The fibers of this map define a lattice congruence of the weak order. Therefore, there is an homomorphism from the weak order on permutations to the rotation lattice on permutrees. It specializes to the classical weak order on permutations~\cite{LascouxSchutzenberger}, the Tamari order on binary trees~\cite{TamariFestschrift}, the Cambrian lattice on Cambrian trees~\cite{Reading-CambrianLattices, ChatelPilaud}, and the boolean lattice on binary sequences.
\\[-.3cm]
\item[Geometry] We provide the vertex and facet description of the permutreehedron, a polytope whose graph is the Hasse diagram of the rotation lattice on permutrees. The permutreehedron is obtained by deleting facets from the classical permutahedron. It specializes to the classical permutahedron~\cite[Lecture~0]{Ziegler}, to J.-L.~Loday's and C.~Hohlweg and C.~Lange's associahedra~\cite{Loday, HohlwegLange}, and to the parallelepiped generated by~$\set{\b{e}_{i+1}-\b{e}_i}{i \in [n-1]}$.
\\[-.3cm]
\item[Algebra] We construct a Hopf algebra on permutrees and describe the product and coproduct in this algebra and its dual in terms of cut and paste operations on permutrees. It contains as subalgebras C.~Malvenuto and C.~Reutenauer's algebra on permutations~\cite{MalvenutoReutenauer, DuchampHivertThibon}, J.-L.~Loday and M.~Ronco's algebra on binary trees~\cite{LodayRonco, HivertNovelliThibon-algebraBinarySearchTrees}, G.~Chatel and V.~Pilaud's algebra on Cambrian trees~\cite{ChatelPilaud}, and I.~Gelfand, D.~Krob, A.~Lascoux, B.~Leclerc, V.~S. Retakh, and J.-Y.~Thibon's algebra on binary sequences~\cite{GelfandKrobLascouxLeclercRetakhThibon}.
\end{description}

All our constructions and proofs are generalizations of previous work, in particular~\cite{ChatelPilaud} from which we borrow the general structure of the paper. Nevertheless, we believe that our main contribution is the very \defn{unified definition} of permutrees which leads to natural constructions, simple proofs, and new objects in algebra and geometry. 

\begin{figure}
\begin{tabular}{l|c|c|c}
	& permutations & binary trees & binary sequences \\
	\hline
	& Weak order & Tamari lattice & boolean lattice \\
	\rotatebox{90}{\hspace*{.1cm} Combinatorics} &  \includegraphics[scale=.3]{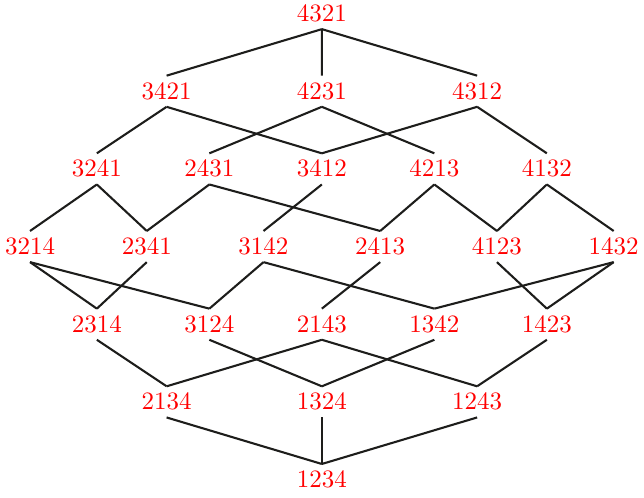} & \includegraphics[scale=.23]{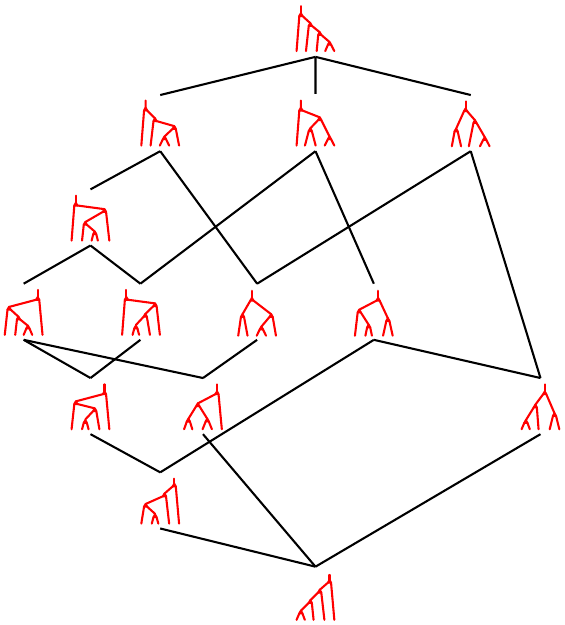} & \includegraphics[scale=.35]{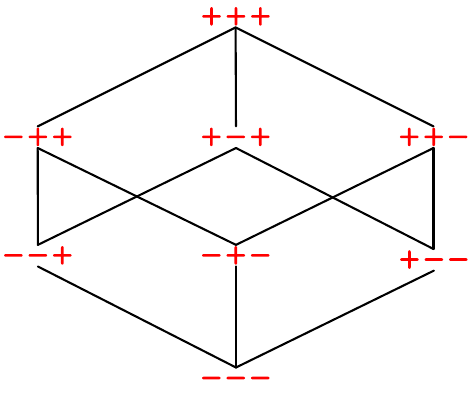} \\
	& \cite{LascouxSchutzenberger} & \cite{TamariFestschrift} & \\
	\hline
	& Permutahedron & Associahedron & Cube \\
	\rotatebox{90}{\hspace*{.5cm} Geometry} & \includegraphics[scale=.4]{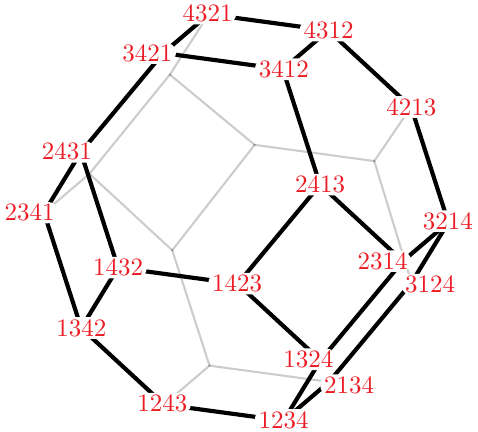} & \includegraphics[scale=.3]{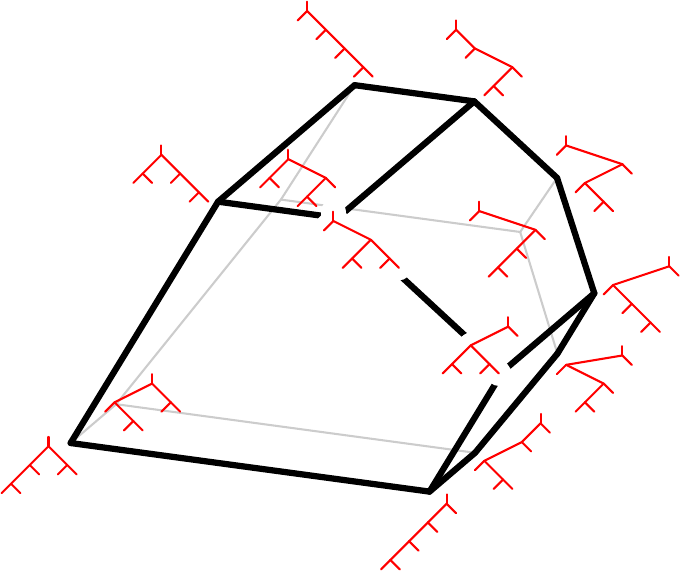} & \includegraphics[scale=.35]{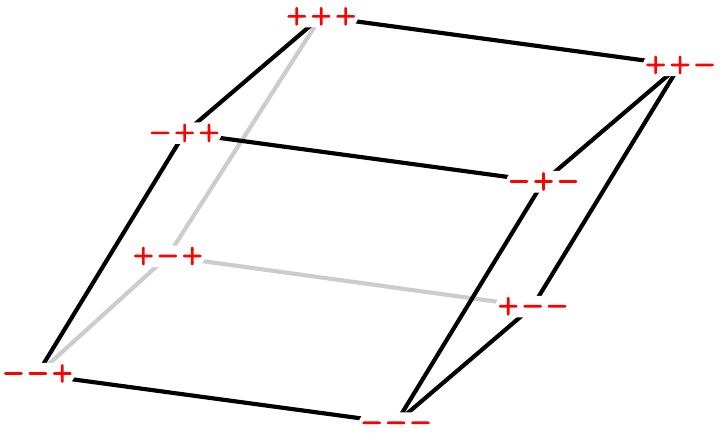} \\
	& \cite[Lecture~0]{Ziegler} & \cite{Loday, HohlwegLange} &  \\
	\hline
	\rotatebox{90}{\hspace*{-.8cm} Algebra \;} & Malvenuto-Reutenauer & Loday-Ronco  & Descent Hopf algebra \\
    & algebra~\cite{MalvenutoReutenauer, DuchampHivertThibon} & algebra~\cite{LodayRonco, HivertNovelliThibon-algebraBinarySearchTrees} & \cite{GelfandKrobLascouxLeclercRetakhThibon} \\
    & & & \\
\hline
\end{tabular}
\vspace{-.2cm}
\caption{Summary of lattice structures, polytopes, and Hopf algebras on 3 families of combinatorial objects}
\label{fig:objectstable}
\end{figure}


\newpage
\section{Permutrees}
\label{sec:permutrees}


\subsection{Permutrees and leveled permutrees}
\label{subsec:permutrees}

\enlargethispage{.5cm}
This paper focuses on the following family of trees. 

\begin{definition}
\label{def:permutree}
A \defn{permutree} is a directed tree~$\tree$ with vertex set~$\ground$ endowed with a bijective vertex labeling $p : \ground \to [n]$ such that for each vertex~$v \in \ground$,
\begin{enumerate}[(i)]
\item $v$ has one or two parents (outgoing neighbors), and one or two children (incoming neighbors);
\item if $v$ has two parents (resp.~children), then all labels in the left ancestor (resp.~descendant) subtree of~$v$ are smaller than~$p(v)$ while all labels in the right ancestor (resp.~descendant) subtree of~$v$ are larger than~$p(v)$.
\end{enumerate}
The \defn{decoration} of a permutree is the $n$-tuple~$\decoration(\tree) \in \Decorations^n$ defined by
\[
\decoration(\tree)_{p(v)} = 
\begin{cases}
\noneCirc{} & \text{if $v$ has one parent and one child,} \\
\downCirc{} & \text{if $v$ has one parent and two children,} \\
\upCirc{} & \text{if $v$ has two parents and one child,} \\
\upDownCirc{} & \text{if $v$ has two parents and two children,}
\end{cases}
\]
for all vertex~$v \in \ground$. Equivalently, one can record the \defn{up} labels~$\decoration^\vee(\tree) = \set{i \in [n]}{\decoration(\tree)_i = \upCirc{} \text{ or } \upDownCirc{} }$ of the vertices with two parents and the \defn{down} labels~$\decoration_\wedge(\tree) = \set{i \in [n]}{\decoration(\tree)_i = \downCirc{} \text{ or } \upDownCirc{} }$ of the vertices with two children. If~$\decoration(\tree) = \decoration$, we say that $\tree$ is a \defn{$\decoration$-permutree}. 

We denote by~$\Permutrees(\decoration)$ the set of $\decoration$-permutrees, by~$\Permutrees(n) = \bigsqcup_{\decoration \in \Decorations^n} \Permutrees(\decoration)$ the set of all permutrees on~$n$ vertices, and by~$\Permutrees \eqdef \bigsqcup_{n \in \N} \Permutrees(n)$ the set of all permutrees.
\end{definition}

\begin{definition}
\label{def:increasingTree}
An \defn{increasing tree} is a directed tree~$\tree$ with vertex set~$\ground$ endowed with a bijective vertex labeling~$q : \ground \to [n]$ such that~$v \to w$ in~$\tree$ implies~$q(v) < q(w)$.
\end{definition}

\begin{definition}
\label{def:labeledPermutree}
A \defn{leveled permutree} is a directed tree~$\tree$ with vertex set~$\ground$ endowed with two bijective vertex labelings~${p,q : V \to [n]}$ which respectively define a permutree and an increasing tree. In other words, a leveled permutree is a permutree endowed with a linear extension of its transitive closure (given by the inverse of the labeling~$q$).
\end{definition}

\begin{figure}
  \centerline{\includegraphics{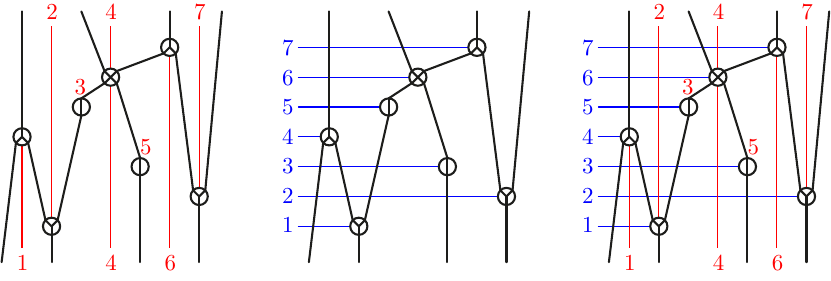}}
  \caption{A permutree (left), an increasing tree (middle), and a leveled permutree (right). The decoration is \downCirc{}\upCirc{}\noneCirc{}\upDownCirc{}\noneCirc{}\downCirc{}\upCirc{}.}
  \label{fig:leveledPermutree}
\end{figure}

\fref{fig:leveledPermutree} provides examples of a permutree (left), an increasing tree (middle), and a leveled permutree (right). We use the following conventions in all figures of this paper:
\begin{enumerate}[(i)]
\item All edges are oriented bottom-up --- we can thus omit the edge orientation;
\item For a permutree, the vertices appear from left to right in the order given by the labeling~$p$ --- we can thus omit the vertex labeling;
\item For an increasing tree, the vertices appear from bottom to top in the order given by the labeling~$q$ --- we can thus omit the vertex labeling;
\item In particular, for a leveled permutree with vertex labelings~$p,q : \ground \to [n]$ as in Definition~\ref{def:labeledPermutree}, each vertex~$v$ appears at position~$\big( p(v),q(v) \big)$ --- we can thus omit both labelings;
\item In all our trees, we decorate the vertices with \noneCirc{}, \downCirc{}, \upCirc{}, or \upDownCirc{} depending on their number of parents and children, following the natural visual convention of Definition~\ref{def:permutree};
\item In a permutree, we often draw a vertical red wall below \downCirc{} and \upDownCirc{} vertices and above \upCirc{} and \upDownCirc{} vertices to mark the separation between the left and right descendant or ancestor subtrees of these vertices.
\end{enumerate}

\begin{example}
\label{exm:permutreesSpecificDecorations}

\begin{figure}[b]
  \centerline{\includegraphics[scale=.8]{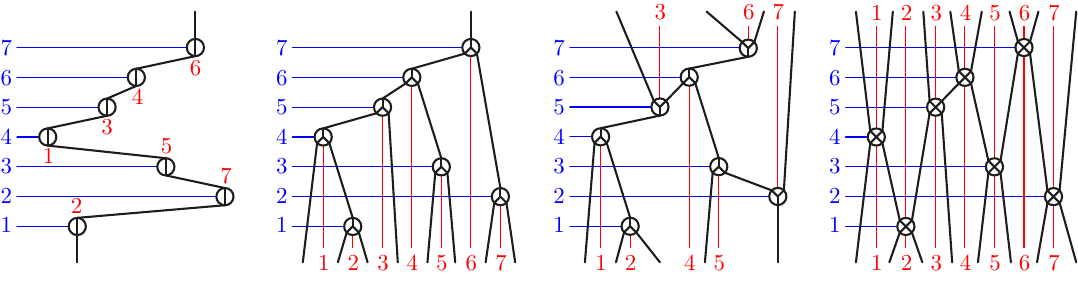}}
  \caption{Leveled permutrees corresponding to a permutation (left), a leveled binary tree (middle left), a leveled Cambrian tree (middle right), and a leveled binary sequence (right).}
  \label{fig:permutationsBinaryTreesCambrianTreesBinarySequences}
\end{figure}

For specific decorations, permutrees specialize to classical combinatorial families:
\begin{enumerate}[(i)]
\item Permutrees with decoration~$\noneCirc{}^n$ are in bijection with permutations of~$[n]$. Indeed, such a permutree is just a path of vertices labeled by~$[n]$.
\item Permutrees with decoration~$\downCirc{}^n$ are in bijection with rooted planar binary trees on $n$ vertices. Indeed, such a permutree has a structure of rooted planar binary tree, and its labeling~${p : \ground \to [n]}$ is just the inorder labeling which can be recovered from the binary tree (inductively label each vertex after its left subtree and before its right subtree).
\item Permutrees with decoration in~$\{\downCirc{}, \upCirc{}\}^n$ are precisely the Cambrian trees of~\cite{ChatelPilaud}.
\item Permutrees with decoration~$\upDownCirc{}^n$ are in bijection with binary sequences with $n-1$ letters. Indeed, such a tree can be transformed to a sequence of letters~$u$ or~$d$ whose $i$th letter records whether the vertex~$i$ is a child of the vertex~$i+1$ or the opposite.
\item Permutrees with decoration in~$\{\noneCirc{}, \upDownCirc{}\}^n$ are in bijection with acyclic orientations of the graph with vertices~$[n]$ and an edge between any two positions separated only by~\noneCirc{}'s in the decoration.
\end{enumerate}
\fref{fig:permutationsBinaryTreesCambrianTreesBinarySequences} illustrates these families represented as permutrees. The goal of this paper is to propose a uniform treatment of all these families and of the lattices, morphisms, polytopes, and Hopf algebras associated to them.
\end{example}

\begin{remark}
\label{rem:symmetrees}
There are two natural operations on permutrees, that we call \defn{symmetrees}, given by horizontal and vertical reflections. Denote by~$\horizontalSymmetry{\tree}$ (resp.~$\verticalSymmetry{\tree}$) the permutree obtained from~$\tree$ by a horizontal (resp.~vertical) reflection. Its decoration is then~$\decoration(\horizontalSymmetry{\tree}) = \horizontalSymmetry{\decoration(\tree)}$ (resp.~${\decoration(\verticalSymmetry{\tree}) = \verticalSymmetry{\decoration(\tree)}}$), where~$\horizontalSymmetry{\decoration}$ (resp.~$\verticalSymmetry{\decoration}$) denotes the decoration obtained from~$\decoration$ by a mirror image (resp.~by interverting~\downCirc{} and~\upCirc{} decorations).
\end{remark}

\begin{remark}
\label{rem:bijectionsPermutrees}
Observe that the decorations of the leftmost and rightmost vertices of a permutree are not really relevant. Namely, consider two decorations~$\decoration, \decoration' \in \Decorations^n$ such that~$\decoration'$ is obtained from~$\decoration$ by forcing~$\decoration'_1 = \decoration'_n = \noneCirc{}$. Then there is a bijection from $\decoration$-permutrees to $\decoration'$-permutrees which consists in deleting the left (resp.~right) incoming and outgoing edges~---~if any~---~of the leftmost (resp.~rightmost) vertex of a $\decoration$-permutree. See \fref{fig:bijectionsPermutrees}\,(left).

When~$\decoration = \decoration'\upDownCirc{}\,\decoration''$, there is a bijection between~$\Permutrees(\decoration)$ and~$\Permutrees(\decoration'\noneCirc{}) \times \Permutrees(\noneCirc{}\decoration'')$. In one direction, we send a $\decoration$-permutree~$\tree$ to the pair of permutrees~$(\tree', \tree'')$ where~$\tree'$ (resp.~$\tree''$) is the $(\decoration'\noneCirc{})$-permutree (resp.~the $(\noneCirc{}\,\decoration'')$-permutree) on the left (resp.~right) of the~$(|\decoration'|+1)$th vertex of~$\tree$. In the other direction, we send a pair of permutrees~$(\tree', \tree'')$ to the $\decoration$-permutree~$\tree$ obtained by merging the rightmost vertex of~$\tree'$ with the leftmost vertex of~$\tree''$. See \fref{fig:bijectionsPermutrees}\,(right).

\begin{figure}[h]
  \centerline{\includegraphics[scale=.8]{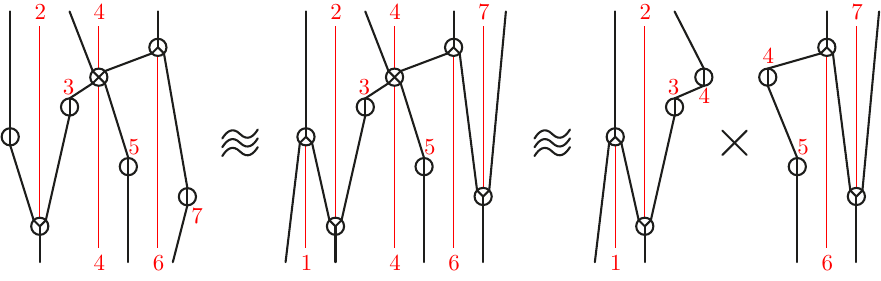}}
  \caption{Some bijections between permutrees: the leftmost and rightmost decorations do not matter (left), and a decoration~$\upDownCirc{}$ yields a product (right).}
  \label{fig:bijectionsPermutrees}
\end{figure}
\end{remark}

\begin{remark}
\label{rem:234angulations}
Permutrees can as well be seen as dual trees of certain $\{2,3,4\}$-angulations. We keep the presentation informal as we only need the intuition of the construction in this paper.

For a given decoration~$\decoration \in \Decorations^n$, we construct a collection~$\b{P}_\decoration$ of points in the plane as follows. We first fix~$\b{p}_0 = (0,0)$ and~$\b{p}_{n+1} = (n+1,0)$ and denote by~$C$ the circle with diameter~$\b{p}_0\b{p}_{n+1}$. Then for each~$i \in [n]$, we a place at abscissa~$i$ a point~$\b{p}_i^0$ on~$\b{p}_0\b{p}_{n+1}$ if~$\decoration_i = \noneCirc{}$, a point~$\b{p}_i^-$ on the circle~$C$ and below~$\b{p}_0\b{p}_{n+1}$ if~$\decoration_i \in \{\downCirc{}, \upDownCirc{}\}$, and a point~$\b{p}_i^+$ on the circle~$C$ and above~$\b{p}_0\b{p}_{n+1}$ if~$\decoration_i \in \{\upCirc{}, \upDownCirc{}\}$. Note that when~$\decoration_i = \upDownCirc{}$, we have both~$\b{p}_i^-$ and~$\b{p}_i^+$ at abscissa~$i$. \fref{fig:234angulation} shows the point set~$\b{P}_\decoration$ for~$\decoration = \downCirc{}\upCirc{}\noneCirc{}\upDownCirc{}\noneCirc{}\downCirc{}\upCirc{}$.

An \defn{arc} in~$\b{P}_\decoration$ is an abscissa monotone curve connecting two external points of~$\b{P}_\decoration$, not passing through any other point of~$\b{P}_\decoration$, and not crossing the vertical line at abscissa~$i$ if~$\decoration_i = \upDownCirc{}$. Arcs are considered up to isotopy in~$\R^2 \ssm \b{P}_\decoration$. In particular, we can assume that the arcs joining two consecutive points on the boundary of the convex hull of~$\b{P}_\decoration$ are straight. We call \mbox{\defn{$\{2,3,4\}$-angulation}} of~$\b{P}_\decoration$ a maximal set of non-crossing arcs in~$\b{P}_\decoration$. An example is given in \fref{fig:234angulation}. As observed in this picture, one can check that a $\{2,3,4\}$-angulation decomposes the convex hull of~$\b{P}_\decoration$ into diangles, triangles and quadrangles. In fact, for each~$j \in [n]$, there is one diangle around~$\b{p}_j^0$ if~$\decoration_j = \noneCirc{}$, one triangle~$\{\b{p}_i^\pm, \b{p}_j^\pm, \b{p}_k^\pm\}$ with~$i < j < k$ if~$\decoration_j \in \{\downCirc{}, \upCirc{}\}$, and one quadrangle~$\{\b{p}_i^\pm, \b{p}_j^-, \b{p}_j^+, \b{p}_k^\pm\}$ with~$i < j < k$ if~$\decoration_j = \upDownCirc{}$.
 
We associate to a $\{2,3,4\}$-angulation of~$\b{P}_\decoration$ its dual permutree with
\begin{itemize}
\item a vertex in each $\{2,3,4\}$-angle: the $j$th vertex is a vertex~\noneCirc{} in the diangle enclosing~$\b{p}_j^0$ if~$\decoration_j = \noneCirc{}$, a vertex~\upCirc{} (resp.~\upCirc{}) in the triangle~$\{\b{p}_i^\pm, \b{p}_j^\pm, \b{p}_k^\pm\}$ with~$i < j < k$ if~$\decoration_j = \downCirc{}$ (resp.~$\upCirc{}$), and a vertex~\upDownCirc{} in the quadrangle~$\{\b{p}_i^\pm, \b{p}_j^-, \b{p}_j^+, \b{p}_k^\pm\}$ with~${i < j < k}$ if~$\decoration_j = \upDownCirc{}$,
\item an edge for each arc: for each arc~$\alpha$, there is an edge from the $\{2,3,4\}$-angle adjacent to~$\alpha$ and below~$\alpha$ to the $\{2,3,4\}$-angle adjacent to~$\alpha$ and above~$\alpha$.
\end{itemize}

\begin{figure}
  \centerline{\includegraphics[scale=1]{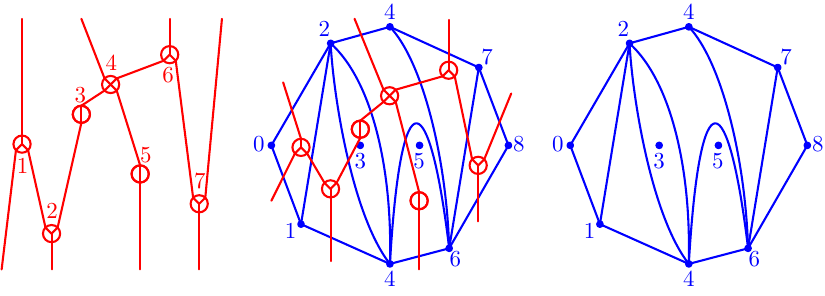}}
  \caption{Permutrees (left) and $\{2,3,4\}$-angulations (right) are dual to each other.}
  \label{fig:234angulation}
\end{figure}
\end{remark}


\subsection{Permutree correspondence}
\label{subsec:correspondence}

In this section, we present a correspondence between the permutations of~$\fS_n$ and the leveled $\decoration$-permutrees for any given decoration~$\decoration \in \Decorations$. This correspondence defines a surjection from the permutations of~$\fS_n$ to the $\decoration$-permutrees by forgetting the increasing labeling. This surjection will be a special case of the surjections described in Section~\ref{subsec:decorationRefinements}. We follow here the presentation of the Cambrian correspondence in~\cite{ChatelPilaud}.

We represent graphically a permutation~$\tau \in \fS_n$ by the $(n \times n)$-table, with rows labeled by positions from bottom to top and columns labeled by values from left to right, and with a dot at row~$i$ and column~$\tau(i)$ for all~$i \in [n]$. (This unusual choice of orientation is necessary to fit later with the existing constructions of~\cite{LodayRonco, HivertNovelliThibon-algebraBinarySearchTrees, ChatelPilaud}.)

A \defn{decorated permutation} is a permutation table where each dot is decorated by~\noneCirc{},~\downCirc{},~\upCirc{},~or~\upDownCirc{}. See the top left corner of \fref{fig:insertionAlgorithm}. We could equivalently think of a permutation where the positions or the values receive a decoration, but it will be useful later to switch the decoration from positions to values. The \defn{p-decoration} (resp.~\defn{v-decoration}) of a decorated permutation~$\tau$ is the sequence~$\pdecoration(\tau)$ (resp.~$\vdecoration(\tau)$) of decorations of~$\tau$ ordered by positions from bottom to top (resp.~by values from left to right). For a permutation~$\tau \in \fS_n$ and a decoration~$\decoration \in \Decorations^n$, we denote by~$\tau_\decoration$ (resp.~by~$\tau^\decoration$) the permutation~$\tau$ with p-decoration~$\pdecoration(\tau_\decoration) = \decoration$ (resp.~with \mbox{v-decoration}~$\vdecoration(\tau^\decoration) = \decoration$). We let~$\fS_\decoration \eqdef \set{\tau_\decoration}{\tau \in \fS_n}$ and~$\fS^\decoration \eqdef \set{\tau^\decoration}{\tau \in \fS_n}$. Finally, we let
\[
\fS_{\Decorations} \eqdef \bigsqcup_{\substack{n \in \N \\ \decoration \in \Decorations^n}} \fS_\decoration = \bigsqcup_{\substack{n \in \N \\ \decoration \in \Decorations^n}} \fS^\decoration
\]
denote the set of all decorated permutations.

In concrete examples, we underline the down positions/values (those decorated by~\downCirc{} or~\upDownCirc{}) while we overline the up positions/values (those decorated by~\upCirc{} or~\upDownCirc{}): for example, $\up{2}\up{7}5\down{1}3\updown{4}\down{6}$ is the decorated permutation represented on the top left corner of \fref{fig:insertionAlgorithm}, where~${\tau = [2,7,5,1,3,4,6]}$, ${\pdecoration = \upCirc{}\upCirc{}\noneCirc{}\downCirc{}\noneCirc{}\upDownCirc{}\downCirc{}}$ and~${\vdecoration = \downCirc{}\upCirc{}\noneCirc{}\upDownCirc{}\noneCirc{}\downCirc{}\upCirc{}}$.

The insertion algorithm transforms a decorated permutation~$\tau$ to a leveled permutree~$\permutreeCorresp(\tau)$. As a preprocessing, we represent the table of~$\tau$ (with decorated dots in positions~$(\tau(i),i)$ for~$i \in [n]$) and draw a vertical red wall below the down vertices and above the up vertices. These walls separate the table into regions. Note that the number of children (resp. parents) expected at each vertex is the number of regions visible below (resp. above) this vertex. We then sweep the table from bottom to top (thus reading the permutation~$\tau$ from left to right) as follows. The procedure starts with an incoming strand in between any two consecutive down values. At each step, we sweep the next vertex and proceed to the following operations depending on its decoration:
\begin{enumerate}[(i)]
\item a vertex decorated by \noneCirc{} or \upCirc{} catches the only incoming strands it sees, while a vertex decorated by \downCirc{} or \upDownCirc{} connects the two incoming strands just to its left and to its~right,
\item a vertex decorated by \noneCirc{} or \downCirc{} creates a unique outgoing strand, while a vertex decorated by \upCirc{} or \upDownCirc{} creates two outgoing strands just to its left and to its right.
\end{enumerate}
The procedure finishes with an outgoing strand in between any two consecutive up values. See \fref{fig:insertionAlgorithm}.

\begin{figure}
  \centerline{\includegraphics[scale=.8]{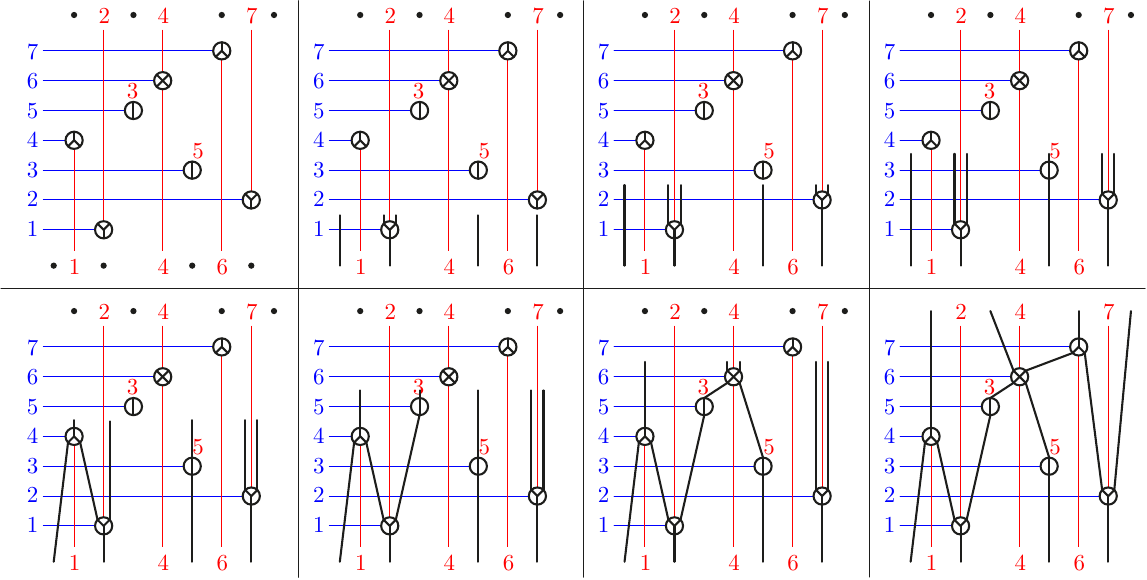}}
  \caption{The insertion algorithm on the decorated permutation~$\up{2}\up{7}5\down{1}3\updown{4}\down{6}$.}
  \label{fig:insertionAlgorithm}
\end{figure}

\begin{proposition}
\label{prop:permutreeCorrespondence}
The map~$\permutreeCorresp$ is a bijection from $\decoration$-decorated permutations to leveled $\decoration$-permutrees.
\end{proposition}

\begin{proof}
First, we need to prove that the map~$\permutreeCorresp$ is well-defined. This relies on the following invariant of the sweeping algorithm: along the sweeping line, there is precisely one strand in each of the intervals separated by the walls. Indeed, this invariant  holds when we start the procedure and is preserved when we sweep any kind of vertex. Therefore, the sweeping algorithm creates a graph whose vertices are the decorated dots of the permutation table together with the initial and final positions of the strands and where no edge crosses a red wall. It follows that this graph is a tree (a cycle would force an edge to cross a red wall), and it is a leveled permutree (the walls separate left and right ancestor or descendant subtrees). To prove that~$\permutreeCorresp$ is bijective, we already observed that a leveled permutree~$\tree$ is a permutree endowed with a linear extension~$\tau$. We can consider that~$\tau$ is decorated by the decorations of the vertices of~$\tree$. Finally, one checks easily that when inserting the decorated permutation~$\tau$, the resulting leveled permutree is~$\permutreeCorresp(\tau) = \tree$.
\end{proof}

For a decorated permutation~$\tau$, we denote by~$\PSymbol(\tau)$ the permutree obtained by forgetting the increasing labeling in~$\permutreeCorresp(\tau)$ and by~$\QSymbol(\tau)$ the increasing tree obtained by forgetting the permutree labeling in~$\permutreeCorresp(\tau)$. These trees should be thought of as the insertion and recording trees of the permutree correspondence, in analogy to the insertion and recording tableaux in the Robinson-Schensted correspondence~\cite{Schensted}. The same analogy was already done for the sylvester correspondence~\cite{HivertNovelliThibon-algebraBinarySearchTrees} and the Cambrian correspondence~\cite{ChatelPilaud}. The following statement was observed along the previous proof.

\begin{proposition}
\label{prop:fibersPSymbol}
The decorated permutations~$\tau \in \fS^\decoration$ such that~$\PSymbol(\tau) = \tree$ are precisely the linear extensions of (the transitive closure of) the permutree~$\tree$.
\end{proposition}

\begin{example}
\label{exm:insertionSpecificDecorations}
Following Example~\ref{exm:permutreesSpecificDecorations}, the $\decoration$-permutree~$\PSymbol(\tau)$ is:
\begin{enumerate}[(i)]
\item a path with vertices labeled by~$\tau$ when~$\decoration = \noneCirc{}^n$,
\item the binary tree obtained by successive insertions (in a binary search tree) of the values of~$\tau$ read from right to left when~$\decoration = \downCirc{}^n$,
\item the Cambrian tree obtained by the insertion algorithm described in~\cite{ChatelPilaud} when~$\decoration \in \{\downCirc{}, \upCirc{}\}^n$,
\item a permutree recording the recoils of~$\tau$ when~$\decoration = \upDownCirc{}^n$. Namely, for~${i \in [n-1]}$, the vertex~$i$ is below the vertex~$i+1$ in~$\PSymbol(\tau)$ if $\tau^{-1}(i) < \tau^{-1}(i+1)$, and above otherwise.
\end{enumerate}
For example, the leveled permutrees of \fref{fig:permutationsBinaryTreesCambrianTreesBinarySequences} were all obtained by inserting the permutation~$2751346$ with different decorations.
\end{example}


\subsection{Permutree congruence}
\label{subsec:congruence}

In this section, we characterize the decorated permutations which give the same $\PSymbol$-symbol in terms of a congruence relation defined by a rewriting rule. We note that this is just a straightforward extension of the definitions of the sylvester congruence by F.~Hivert, J.-C.~Novelli and J.-Y.~Thibon~\cite{HivertNovelliThibon-algebraBinarySearchTrees} and of the Cambrian congruence of N.~Reading~\cite{Reading-CambrianLattices}.

\begin{definition}
\label{def:permutreeCongruence}
For a decoration~$\decoration \in \Decorations^n$, the \defn{$\decoration$-permutree congruence} is the equivalence relation on~$\fS^\decoration$ defined as the transitive closure of the rewriting rules
\[
\begin{array}{ll}
UacVbW \equiv_\decoration UcaVbW & \text{if } a < b < c \text{ and } \decoration_b = \downCirc{} \text{ or } \upDownCirc{}, \\
UbVacW \equiv_\decoration UbVcaW & \text{if } a < b < c \text{ and } \decoration_b = \upCirc{} \text{ or } \upDownCirc{},
\end{array}
\]
where~$a,b,c$ are elements of~$[n]$ while~$U,V,W$ are words on~$[n]$. Note that the decorations of~$a$ and~$c$ do not matter, only that of~$b$ which we call the \defn{witness} of the rewriting rule.

The \defn{permutree congruence} is the equivalence relation on all decorated permutations~$\fS_{\Decorations}$ obtained as the union of all $\decoration$-permutree congruences:
\[
\equiv \; \eqdef \bigsqcup_{\substack{n \in \N \\ \decoration \in \Decorations^n}} \!\! \equiv_\decoration.
\]
\end{definition}

\begin{proposition}
\label{prop:permutreeCongruenceClass}
Two decorated permutations~$\tau, \tau' \in \fS_{\Decorations}$ are permutree congruent if and only if they have the same $\PSymbol$-symbol:
\[
\tau \equiv \tau' \iff \PSymbol(\tau) = \PSymbol(\tau').
\]
\end{proposition}

\begin{proof}
It boils down to observe that two consecutive vertices~$a,c$ in a linear extension~$\tau$ of a \mbox{$\decoration$-permu}\-tree~$\tree$ can be switched while preserving a linear extension~$\tau'$ of~$\tree$ precisely when they belong to distinct ancestor or descendant subtrees of a vertex~$b$ of~$\tree$. It follows that the vertices~$a,c$ lie on either sides of~$b$ so that we have~$a < b < c$. If~$\decoration_b = \downCirc{}$ or~$\upDownCirc{}$ and~$a,c$ appear before~$b$ in~$\tau$, then they belong to distinct descendant subtrees of~$b$ and~$\tau = UacVbW$ can be switched to~$\tau' = UcaVbW$. If~$\decoration_b = \upCirc{}$ or~$\upDownCirc{}$ and~$a,c$ appear after~$b$ in~$\tau$, then they belong to distinct ancestor subtrees of~$b$ and~$\tau = UbVacW$ can be switched to~$\tau' = UbVcaW$.
\end{proof}

Recall that the (right) \defn{weak order} on~$\fS_n$ is defined as the inclusion order of (right) inversions, where a (right) \defn{inversion} of~$\tau \in \fS_n$ is a pair of values~$i < j$ such that~${\tau^{-1}(i) > \tau^{-1}(j)}$. This order is a lattice with minimal element~$[1, 2, \dots, n-1, n]$ and maximal element~$[n, n-1, \dots, 2, 1]$.

A \defn{lattice congruence} of a lattice~$(L, \le, \meet, \join)$ is an equivalence relation~$\equiv$ on~$L$ which respects the meet and join operations: $x \equiv x'$ and~$y \equiv y'$ implies $x \meet y \equiv x' \meet y'$ and~$x \join y \equiv x' \join y'$ for all~$x, x', y, y' \in L$. For finite lattices, it is equivalent to require that equivalence classes of~$\equiv$ are intervals of~$L$ and that the maps~$\projDown$ and~$\projUp$ respectively sending an element to the bottom and top elements of its equivalence class are order preserving.

The \defn{quotient} of~$L$ modulo the congruence~$\equiv$ is the lattice~$L/{\equiv}$ whose elements are the equivalence classes of~$L$ under~$\equiv$, and where for any two classes~$X, Y \in L/{\equiv}$, the order is given by~$X \le Y$ if and only if there exist representatives~$x \in X$ and~$y \in Y$ such that~$x \le y$, and the meet (resp.~join) is given by~$X \meet Y = x \meet y$ (resp.~$X \join Y = x \join y$) for any representatives~$x \in X$ and~$y \in Y$.

N.~Reading deeply studied the lattice congruences of the weak order, see in particular~\cite{Reading-latticeCongruences, Reading-arcDiagrams}. Using his technology, in particular that of~\cite{Reading-arcDiagrams}, we will prove in the next section that our permutree congruences are as well lattice congruences of the weak order.

\begin{proposition}
\label{prop:weakOrderCongruence}
For any decoration~$\decoration \in \Decorations^n$, the $\decoration$-permutree congruence~$\equiv_\decoration$ is a lattice congruence of the weak order on~$\fS^\decoration$.
\end{proposition}

\begin{corollary}
\label{coro:patternAvoidingPermutations}
The $\decoration$-permutree congruence classes are intervals of the weak order on~$\fS_n$. In particular, the following sets are in bijection:
\begin{enumerate}[(i)]
\item permutrees with decoration~$\decoration$,
\item $\decoration$-permutree congruence classes,
\item permutations of~$\fS_n$ avoiding the patterns~$ac \dash b$ with~$\decoration_b \in \{\downCirc{}, \upDownCirc{}\}$ and~$b \dash ac$ with~$\decoration_b \in \{\upCirc{}, \upDownCirc{}\}$,
\item permutations of~$\fS_n$ avoiding the patterns~$ca \dash b$ with~$\decoration_b \in \{\downCirc{}, \upDownCirc{}\}$ and~$b \dash ca$ with~$\decoration_b \in \{\upCirc{}, \upDownCirc{}\}$.
\end{enumerate}
\end{corollary}

\begin{example}
\label{exm:congruencesSpecificDecorations}
Following Example~\ref{exm:permutreesSpecificDecorations}, the $\decoration$-permutree congruence~$\equiv$ is:
\begin{enumerate}[(i)]
\item the trivial congruence when~$\decoration = \noneCirc{}^n$,
\item the sylvester congruence~\cite{HivertNovelliThibon-algebraBinarySearchTrees} when~$\decoration = \downCirc{}^n$,
\item the Cambrian congruence~\cite{Reading-latticeCongruences, Reading-CambrianLattices, ChatelPilaud} when~$\decoration \in \{\downCirc{}, \upCirc{}\}^n$,
\item the hypoplactic congruence~\cite{KrobThibon-NCSF4, Novelli-hypoplactic} ($\sigma \equiv \tau$ if and only if $\sigma$ and~$\tau$ have the same descent sets) when~$\decoration = \upDownCirc{}^n$.
\end{enumerate}
\end{example}

\begin{remark}
In~\cite{Reading-latticeCongruences}, N.~Reading defines a notion of homogeneous congruences of the weak order. For example, the parabolic congruences cover all homogeneous degree $1$ congruences. It turns out that the permutree congruences cover all homogeneous degree $2$ congruences. In other words, the permutree congruences are precisely the lattice congruences obtained by contracting a subset of side edges of the bottom hexagonal faces of the weak order together with all edges forced by these contractions. For example, as illustrated in \fref{fig:fibersPermutreeCongruences}, for each~$\delta \in \Decorations^4$, the $\delta$-permutree congruence is the finest congruence which contracts the edges
\begin{gather*}
[1324, 3124] \text{ if } \delta_2 \in \{\downCirc, \upDownCirc\}, \qquad\qquad [2134, 2314] \text{ if } \delta_2 \in \{\upCirc, \upDownCirc\}, \\
[1243, 1423] \text{ if } \delta_3 \in \{\downCirc, \upDownCirc\}, \qquad\qquad [1324, 1342] \text{ if } \delta_3 \in \{\upCirc, \upDownCirc\}.
\end{gather*}
\end{remark}


\subsection{Arc diagrams}
\label{subsec:arcDiagrams}

We now interpret the permutree congruence in terms of the arc diagrams of N.~Reading~\cite{Reading-arcDiagrams}. We first briefly recall some definitions adapted to suit better our purposes (in contrast to the presentation of~\cite{Reading-arcDiagrams}, our arc diagrams are horizontal to fit our conventions).

Consider the $n$ points~$\{\b{q}_1, \dots, \b{q}_n\}$ where~$\b{q}_i = (i,0)$. An \defn{arc diagram} is a set of abscissa monotone curves (called \defn{arcs}) joining two points~$\b{q}_i$ and~$\b{q}_j$, not passing through any other point~$\b{q}_k$, and such that:
\begin{itemize}
\item no two arcs intersect except possibly at their endpoints,
\item no two arcs share the same left endpoint or the same right endpoint (but the right endpoint of an arc may be the left endpoint of another arc).
\end{itemize}
Two arcs are equivalent if they have the same endpoints~$\b{q}_i, \b{q}_j$ and pass above or below the same points~$\b{q}_k$ for~$i < k < j$, and two arc diagrams are  equivalent if their arcs are pairwise equivalent. In other words, arc diagrams are considered up to isotopy. Denote by~$\arcDiagrams_n$ the set of arc diagrams on~$n$ points.

There are two similar maps~$\asc$ and~$\desc$ from~$\fS_n$ to~$\arcDiagrams_n$: given a permutation~$\tau \in \fS_n$, draw the table with a dot at row~$i$ and column~$\tau(i)$ for each~$i \in [n]$, trace the segments joining two consecutive dots corresponding to ascents (resp.~descents) of~$\tau$, let the points and segments fall down to the horizontal line, allowing the segments connecting ascents (resp.~descents) to curve but not to pass through any dot, and call the resulting arc diagram~$\asc(\tau)$ (resp.~$\desc(\tau)$). See \fref{fig:arcDiagrams}\,(left) for an illustration. It is proved in~\cite{Reading-arcDiagrams} that~$\asc$ and~$\desc$ define bijections from~$\fS_n$ to~$\arcDiagrams_n$.

Consider now a decoration~$\decoration \in \Decorations^n$. Draw a vertical wall below each point~$\b{q}_i$ with~$\decoration_i \in \{\downCirc{}, \upDownCirc\}$ and above each point~$\b{q}_i$ with~$\decoration_i \in \{\upCirc{}, \upDownCirc\}$. We denote by~$U_\decoration$ the set of arcs which do not cross any of these walls. Note that this is very similar to~\cite[Example~4.9]{Reading-arcDiagrams} where each point~$\b{q}_i$ is incident to precisely one wall. The following lemma is immediate.

\begin{lemma}
\label{lem:arcDiagrams}
For any~$\tau \in \fS_n$, the arc diagram~$\asc(\tau)$ (resp.~$\desc(\tau)$) uses only arcs in~$U_\decoration$ if and only if~$\tau$ avoids the patterns~$ac \dash b$ with~$\decoration_b \in \{\downCirc{}, \upDownCirc{}\}$ and~$b \dash ac$ with~$\decoration_b \in \{\upCirc{}, \upDownCirc{}\}$ (resp.~the patterns~$ca \dash b$ with~$\decoration_b \in \{\downCirc{}, \upDownCirc{}\}$ and~$b \dash ca$ with~$\decoration_b \in \{\upCirc{}, \upDownCirc{}\}$).
\end{lemma}

Consider now two arcs~$\alpha$ with endpoints~$\b{q}_i,\b{q}_j$, and~$\beta$ with endpoints~$\b{q}_k,\b{q}_\ell$. Then~$\alpha$ is a \defn{subarc} of~$\beta$ if~$k \le i \le j \le \ell$ and~$\alpha$ and~$\beta$ pass above or below the same points~$\b{q}_m$ for~$i < m < j$. Consider now a subset~$U$ of all possible arcs which is closed by subarcs. Denote by~$\arcDiagrams_n(U)$ the set of arc diagrams consisting only of arcs of~$U$. It is then proved in~\cite{Reading-arcDiagrams} that~$\asc^{-1}(\arcDiagrams_n(U))$ (resp.~$\desc^{-1}(\arcDiagrams_n(U))$) is the set of bottom (resp.~top) elements of the classes of a lattice congruence~$\equiv_U$ of the weak order. We therefore obtain the proof of Proposition~\ref{prop:weakOrderCongruence}.

\begin{proof}[Proof of Proposition~\ref{prop:weakOrderCongruence}]
For any decoration~$\decoration \in \Decorations^n$, the set~$U_\decoration$ of arcs not crossing any wall is clearly closed by subarcs. It follows that~$\asc^{-1}(\arcDiagrams_n(U_\decoration))$ (resp.~$\desc^{-1}(\arcDiagrams_n(U_\decoration))$) is the set of bottom (resp.~top) elements of the classes of a lattice congruence~$\equiv_{U_\decoration}$ of the weak order. But Lemma~\ref{lem:arcDiagrams} ensures that~$\asc^{-1}(\arcDiagrams_n(U_\decoration))$ (resp.~$\desc^{-1}(\arcDiagrams_n(U_\decoration))$) are precisely the bottom (resp.~top) elements of the classes of the permutree congruence~$\equiv_\decoration$. The two congruences~$\equiv_{U_\decoration}$ and~$\equiv_\decoration$ thus coincide which proves that~$\equiv_\decoration$ is a lattice congruence of the weak order.
\end{proof}

We conclude this section with a brief comparison between permutrees and arc diagrams. Consider a permutree~$\tree$, delete all its leaves, and let its vertices fall down to the horizontal axis, allowing the edges to curve but not to pass through any vertex. The resulting set of oriented arcs can be decomposed into the set~$\asc(\tree)$ of increasing arcs oriented from~$i$ to~$j$ with~$i < j$ and the set~$\desc(\tree)$ of decreasing arcs oriented from~$j$ to~$i$ with~$i < j$. The following observation, left to the reader, is illustrated in \fref{fig:arcDiagrams}\,(right).

\begin{proposition}
The set~$\asc(T)$ is the arc diagram~$\asc(\tau)$ of the maximal linear extension~$\tau$ of~$\tree$ while the set~$\desc(\tree)$ is the arc diagram~$\desc(\sigma)$ of the minimal linear extension~$\sigma$ of~$\tree$.
\end{proposition}

\begin{figure}
  \centerline{\includegraphics[scale=1]{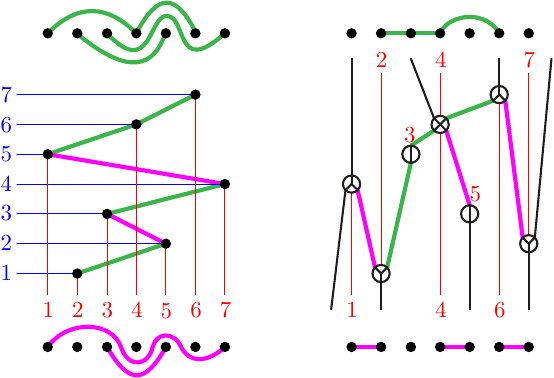}}
  \caption{The arc diagrams~$\asc(\tau)$ (green, up) and~$\desc(\tau)$ (pink, down) associated to the permutation~$\tau = 2537146$ (left) and the arc diagrams~$\asc(\tree)$ (green, up) and $\desc(\tree)$ (pink, down) associated to a permutree~$\tree$ (right).}
  \label{fig:arcDiagrams}
\end{figure}


\subsection{Numerology}
\label{subsec:numerology}

\enlargethispage{.5cm}
In this section, we discuss enumerative properties of permutrees. We call \defn{factorial-Catalan number} the number~$\factCatalan{\decoration}$ of $\decoration$-permutrees for~$\decoration \in \Decorations^n$. The values of~$\factCatalan{\decoration}$ for~$|\decoration| \in \{3,4,5,6\}$ are reported in Table~\ref{table:factorialCatalan}. To evaluate these numbers, we proceed in two steps: we first show that the number of $\decoration$-permutrees only depends on the positions of the \noneCirc{} and \upDownCirc{} in~$\decoration$, and then give summation formulas for factorial-Catalan numbers~$\factCatalan{\decoration}$ for $\decoration \in \{\noneCirc{}, \downCirc{}, \upDownCirc{}\}^n$.

\begin{table}
    \centerline{
    $\begin{array}[t]{c|ccc}
     			& {\cdot}{\cdot}{\cdot} & {\cdot}\noneCirc{}{\cdot} \\
    \hline
    {\cdot}{\cdot}{\cdot} 		& 5 & 6 \\
    {\cdot}\upDownCirc{}{\cdot} & 4 & .
    \end{array}$
    \qquad
    $\begin{array}[t]{c|ccc}
     			& {\cdot}{\cdot}{\cdot}{\cdot} & {\cdot}\noneCirc{}{\cdot}{\cdot} & {\cdot}\noneCirc{}\noneCirc{}{\cdot} \\
    \hline
    {\cdot}{\cdot}{\cdot}{\cdot} 				& 14 & 18 & 24 \\
    {\cdot}{\cdot}\upDownCirc{}{\cdot} 			& 10 & 12 & .  \\
    {\cdot}\upDownCirc{}\upDownCirc{}{\cdot}	& 8  & .  & .
    \end{array}$
    \qquad
    $\begin{array}[t]{c|cccccc}
     			& {\cdot}{\cdot}{\cdot}{\cdot}{\cdot} & {\cdot}\noneCirc{}{\cdot}{\cdot}{\cdot} & {\cdot}{\cdot}\noneCirc{}{\cdot}{\cdot} & {\cdot}\noneCirc{}{\cdot}\noneCirc{}{\cdot} & {\cdot}\noneCirc{}\noneCirc{}{\cdot}{\cdot} & {\cdot}\noneCirc{}\noneCirc{}\noneCirc{}{\cdot} \\
    \hline
    {\cdot}{\cdot}{\cdot}{\cdot}{\cdot} 					& 42 & 56 & 60 & 76 & 84 & 120 \\
    {\cdot}{\cdot}{\cdot}\upDownCirc{}{\cdot} 				& 28 & 36 & 36 & .  & 48 & .   \\
    {\cdot}{\cdot}\upDownCirc{}{\cdot}{\cdot} 				& 25 & 30 & 36 & .  & .  & .   \\
    {\cdot}{\cdot}\upDownCirc{}\upDownCirc{}{\cdot}			& 20 & 24 & .  & .  & .  & .   \\
    {\cdot}\upDownCirc{}{\cdot}\upDownCirc{}{\cdot}			& 20 & 24 & .  & .  & .  & .   \\
    {\cdot}\upDownCirc{}\upDownCirc{}\upDownCirc{}{\cdot}	& 16 & .  & .  & .  & .  & .
    \end{array}$}
    
    \vspace{.3cm}
    \centerline{$\begin{array}{c|cccccccccccccccc}
     			& {\cdot}{\cdot}{\cdot}{\cdot}{\cdot}{\cdot} & {\cdot}\noneCirc{}{\cdot}{\cdot}{\cdot}{\cdot} & {\cdot}{\cdot}\noneCirc{}{\cdot}{\cdot}{\cdot} & {\cdot}{\cdot}{\cdot}\noneCirc{}{\cdot}{\cdot} & {\cdot}\noneCirc{}{\cdot}{\cdot}\noneCirc{}{\cdot} & {\cdot}\noneCirc{}{\cdot}\noneCirc{}{\cdot}{\cdot} & {\cdot}\noneCirc{}\noneCirc{}{\cdot}{\cdot}{\cdot} & {\cdot}{\cdot}\noneCirc{}\noneCirc{}{\cdot}{\cdot} & {\cdot}\noneCirc{}\noneCirc{}{\cdot}\noneCirc{}{\cdot} & {\cdot}\noneCirc{}\noneCirc{}\noneCirc{}{\cdot}{\cdot} & {\cdot}\noneCirc{}\noneCirc{}\noneCirc{}\noneCirc{}{\cdot} \\
    \hline
    {\cdot}{\cdot}{\cdot}{\cdot}{\cdot}{\cdot}							& 132 & 180 & 200 & 200 & 248 & 280 & 288 & 324 & 408 & 480 & 720 \\
    {\cdot}{\cdot}{\cdot}{\cdot}\upDownCirc{}{\cdot}					& 84  & 112 & 120 & 112	& .   & 152 & 168 & 168 & .   & 240 & .   \\
    {\cdot}{\cdot}{\cdot}\upDownCirc{}{\cdot}{\cdot}					& 70  & 90  & 90  & .   & 108 & .   & 120 & .   & 144 & .   & .   \\
    {\cdot}{\cdot}\upDownCirc{}{\cdot}{\cdot}{\cdot}					& 70  & 84  & .   & 90  & 108 & 108 & .   & .   & .   & .   & .   \\
    {\cdot}{\cdot}{\cdot}\upDownCirc{}\upDownCirc{}{\cdot}				& 56  & 72  & 72  & .   & .   & .   & 96  & .   & .   & .   & .   \\
    {\cdot}{\cdot}\upDownCirc{}\upDownCirc{}{\cdot}{\cdot}				& 56  & 60  & .   & .   & 72  & .   & .   & .   & .   & .   & .   \\
    {\cdot}\upDownCirc{}{\cdot}{\cdot}\upDownCirc{}{\cdot}				& 50  & .   & 72  & 72  & .   & .   & .   & 96  & .   & .   & .   \\
    {\cdot}{\cdot}\upDownCirc{}{\cdot}\upDownCirc{}{\cdot}				& 50  & 60  & .   & 60  & .   & 72  & .   & .   & .   & .   & .   \\
    {\cdot}{\cdot}\upDownCirc{}\upDownCirc{}\upDownCirc{}{\cdot}		& 40  & 48  & .   & .   & .   & .   & .   & .   & .   & .   & .   \\
    {\cdot}\upDownCirc{}{\cdot}\upDownCirc{}\upDownCirc{}{\cdot}		& 40  & .   & 48  & .   & .   & .   & .   & .   & .   & .   & .   \\
    {\cdot}\upDownCirc{}\upDownCirc{}\upDownCirc{}\upDownCirc{}{\cdot}	& 32  & .   & .   & .   & .   & .   & .   & .   & .   & .   & .
    \end{array}$}
  \caption{All factorial-Catalan numbers~$\factCatalan{\decoration}$ for~$|\decoration| \in \{3,4,5,6\}$. Each row (resp.~column) corresponds to all decorations with a fixed subset of positions marked with~\upDownCirc{} (resp.~\noneCirc{}). For example, we read in row~${\cdot}{\cdot}\upDownCirc{}{\cdot}{\cdot}{\cdot}$ and column~${\cdot}\noneCirc{}{\cdot}{\cdot}\noneCirc{}{\cdot}$ of the bottom table that~$\factCatalan{\decoration_1\noneCirc{}\upDownCirc{}\decoration_4\noneCirc{}\decoration_6} = 108$ for any~${\decoration_1, \decoration_4, \decoration_6 \in \{\upCirc{}, \downCirc{}\}}$. Dots inside the table correspond to overlapping sets of positions of~\noneCirc{} and~\upDownCirc{}. These tables give all factorial-Catalan numbers up to the mirror symmetry of Remark~\ref{rem:symmetrees}.}
  \label{table:factorialCatalan}
\end{table}


\subsubsection{Only \noneCirc{} and \upDownCirc{} matter}
According to Corollary~\ref{coro:patternAvoidingPermutations}, $\decoration$-permutrees are in bijection with permutations of~$\fS_n$ avoiding the patterns~$ac \dash b$ with~$\decoration_b \in \{\downCirc{}, \upDownCirc{}\}$ and~$b \dash ac$ with~$\decoration_b \in \{\upCirc{}, \upDownCirc{}\}$. We construct a generating tree~$\generatingTree_{\decoration}$ for these permutations. This tree has~$n$ levels, and the nodes at level~$\level$ are labeled by the permutations of~$[\level]$ whose values are decorated by the restriction of~$\decoration$ to~$[\level]$ and avoiding the two patterns~$ac \dash b$ with~$\decoration_b \in \{\downCirc{}, \upDownCirc{}\}$ and~$b \dash ac$ with~$\decoration_b \in \{\upCirc{}, \upDownCirc{}\}$. The parent of a permutation in~$\generatingTree_\decoration$ is obtained by deleting its maximal value. See ~\fref{fig:generatingTree} for examples of such generating trees. 

\begin{figure}[t]
  \centerline{\input{generatingTree}}
  \caption{The generating trees~$\generatingTree_\decoration$ for the decorations~$\decoration = \downCirc{}\upCirc{}\noneCirc{}\upDownCirc{}$ (top) and ${\decoration = \downCirc{}\noneCirc{}\upDownCirc{}\upCirc{}}$ (bottom). Free gaps are marked with blue dots.}
  \label{fig:generatingTree}
\end{figure}

\begin{proposition}
\label{prop:generatingTree}
For any decorations~$\decoration, \decoration' \in \Decorations^n$ such that~$\decoration^{-1}(\noneCirc{}) = \decoration'\,\!^{-1}(\noneCirc{})$ and ${\decoration^{-1}(\upDownCirc{}) = \decoration'\,\!^{-1}(\upDownCirc{})}$, the generating trees~$\generatingTree_\decoration$ and~$\generatingTree_{\decoration'}$ are isomorphic.
\end{proposition}

For the proof, we consider the possible positions of~$\level+1$ in the children of a permutation~$\tau$ at level~$\level$ in~$\generatingTree_\decoration$. Index by~$\{0, \dots, \level\}$ from left to right the gaps before the first letter, between two consecutive letters, and after the last letter of~$\tau$. We call \defn{free gaps} the gaps in~$\{0, \dots, \level\}$ where placing~$\level+1$ does not create a pattern~$ac \dash b$ with~$\decoration_b \in \{\downCirc{}, \upDownCirc{}\}$ and~$b \dash ac$ with~$\decoration_b \in \{\upCirc{}, \upDownCirc{}\}$. They are marked with a blue point~$\freeGap$ in \fref{fig:generatingTree}.

\begin{lemma}
\label{lem:GeneratingTree}
Any permutation at level~$\level$ with~$g$ free gaps has $g$ children in~$\generatingTree_\decoration$, whose numbers of free gaps 
\begin{itemize}
\item all equal $g+1$ when~$\decoration_{\level+1} = \noneCirc{}$,
\item range from~$2$ to~$g+1$ when~$\decoration_{\level+1} = \downCirc{}$ or \upCirc{},
\item all equal $2$ when~$\decoration_{\level+1} = \upDownCirc{}$.
\end{itemize}
\end{lemma}

\begin{proof}
Let~$\tau$ be a permutation at level~$\level$ in~$\generatingTree_\decoration$ with $g$ free gaps. Let~$\sigma$ be the child of~$\tau$ in~$\generatingTree_\decoration$ obtained by inserting~$\level+1$ at a free gap~$j \in \{0, \dots, \level\}$. Then the free gaps of~$\sigma$ are~$0$, $j+1$ together with
\begin{itemize}
\item all free gaps of~$\tau$ if~$\decoration_{\level+1} = \noneCirc{}$,
\item the free gaps of~$\tau$ after~$j$ if~$\decoration_{\level+1} = \downCirc{}$,
\item the free gaps of~$\tau$ before~$j+1$ if~$\decoration_{\level+1} = \upCirc{}$,
\item no other free gaps if~$\decoration_{\level+1} = \upDownCirc{}$. \qedhere
\end{itemize}
\end{proof}

\begin{proof}[Proof of Proposition~\ref{prop:generatingTree}]
Order the children of a node of~$\generatingTree_\decoration$ from left to right by increasing number of free gaps as in \fref{fig:generatingTree}. Lemma~\ref{lem:GeneratingTree} shows that the shape of the resulting tree only depends on the positions of~\noneCirc{} and~\downCirc{} in~$\decoration$. It ensures that the trees~$\generatingTree_\decoration$ and~$\generatingTree_{\decoration'}$ are isomorphic and provides an explicit bijection between the~$\decoration$-permutrees and~$\decoration'$-permutrees when~$\decoration^{-1}(\noneCirc{}) = \decoration'\,\!^{-1}(\noneCirc{})$ and ${\decoration^{-1}(\upDownCirc{}) = \decoration'\,\!^{-1}(\upDownCirc{})}$.
\end{proof}

\enlargethispage{.8cm}
Proposition~\ref{prop:generatingTree} immediately implies the following equi-enumeration result.

\begin{corollary}
\label{coro:equienumerated}
The factorial-Catalan number~$\factCatalan{\decoration}$ only depends on the positions of the symbols~\noneCirc{} and~\upDownCirc{}~in~$\decoration$.
\end{corollary}

Using more carefully the description of the generating tree in Lemma~\ref{lem:GeneratingTree}, we obtain the following recursive formulas for the factorial-Catalan numbers.

\begin{corollary}
\label{coro:inductionFactCatalan}
Let~$\decoration \in \Decorations^n$ and~$\decoration'$ be obtained by deleting the last letter~$\decoration_n$ of~$\decoration$. The number~$\factCatalan{\decoration,g}$ of permutations avoiding~$ac \dash b$ with~$\decoration_b \in \{\downCirc{}, \upDownCirc{}\}$ and~$b \dash ac$ with~$\decoration_b \in \{\upCirc{}, \upDownCirc{}\}$ and with $g$ free gaps satisfies the following recurrence relations:
\[
\factCatalan{\decoration,g} = 
\begin{cases}
\one_{g > 2} \cdot (g-1) \cdot \factCatalan{\decoration',g-1} & \text{if } \decoration_n = \noneCirc{}, \\[.2cm]
\displaystyle \one_{g \ge 2} \cdot \sum_{g' \ge g-1} \factCatalan{\decoration',g'} & \text{if } \decoration_n = \upCirc{} \text{ or } \downCirc{}, \\[.6cm]
\displaystyle \one_{g = 2} \cdot \sum_{g' \ge 2} g' \cdot \factCatalan{\decoration',g'} & \text{if } \decoration_n = \upDownCirc{},
\end{cases}
\]
where~$\one_X$ is $1$ if~$X$ is satisfied and~$0$ otherwise.
\end{corollary}

The last corollary can be used to compute the factorial-Catalan number~$\factCatalan{\decoration} = \sum_{g \ge 2} \factCatalan{\decoration,g}$ inductively from~$\factCatalan{\decoration,2} = 1$ for any~$\decoration$ of size~$1$. We will see however different formulas in the remainder of this section.


\subsubsection{Summation formulas}

By Corollary~\ref{coro:equienumerated}, it is enough to understand the factorial-Catalan number~$\factCatalan{\decoration}$ when~$\decoration \in \{\noneCirc{}, \downCirc{}, \upDownCirc{}\}^n$. Following Remark~\ref{rem:bijectionsPermutrees}, we first observe that we can also get rid of the~\upDownCirc{} symbols.

\begin{lemma}
\label{lem:productFactorialCatalanNumbers}
Assume that~$\decoration = \decoration'\upDownCirc{}\,\decoration''$, then $\factCatalan{\decoration} = \factCatalan{\decoration'\noneCirc{}} \cdot \factCatalan{\noneCirc{}\,\decoration''}$.
\end{lemma}

We can therefore focus on the factorial-Catalan number~$\factCatalan{\decoration}$ when~$\decoration \in \{\noneCirc{}, \downCirc{}\}^n$. Note that when~$\decoration \in \{\noneCirc{}, \downCirc{}\}^n$, all $\decoration$-permutrees have a single outgoing strand, and are therefore rooted. This enables us to derive recursive formulas for factorial-Catalan numbers.

\begin{proposition}
\label{prop:recursiveFormulaOneNode}
For any decoration~$\decoration \in \{\noneCirc{}, \downCirc{}\}^n$, the factorial-Catalan number~$\factCatalan{\decoration}$ satisfies the following recurrence relation
\[
\factCatalan{\decoration} = \sum_{i \in \decoration^{-1}(\noneCirc{})} \factCatalan{\decoration_{|[n] \ssm i}} + \sum_{i \in \decoration^{-1}(\downCirc{})} \factCatalan{\decoration_{|[1, \dots, i-1]}} \cdot \factCatalan{\decoration_{|[i+1, \dots, n]}}
\]
\end{proposition}

\begin{proof}
We group the $\decoration$-permutrees according to their root. The formula thus follows from the obvious bijection between $\decoration$-permutree~$\tree$ with root~$i$ and
\begin{itemize}
\item $\decoration_{|[n] \ssm i}$-permutrees if~$\decoration_i = \noneCirc{}$,
\item pairs of $\decoration_{|[1, \dots, i-1]}$- and $\decoration_{|[i+1, \dots, n]}$-permutrees if~$\decoration_i = \downCirc{}$. \qedhere
\end{itemize}
\end{proof}

\begin{proposition}
\label{prop:recursiveFormulaSeveralNodes}
For any decoration~$\decoration \in \{\noneCirc{}, \downCirc{}\}^n$, the factorial-Catalan number~$\factCatalan{\decoration}$ satisfies the following recurrence relation
\[
\factCatalan{\decoration} = \sum_{\substack{i \in \decoration^{-1}(\downCirc{}) \\ J \subseteq \decoration^{-1}(\noneCirc{})}} \factCatalan{\decoration_{|[1, \dots, i-1] \ssm J}} \cdot \factCatalan{\decoration_{|[i+1, \dots, n] \ssm J}} \cdot |J|!
\]
\end{proposition}

\begin{proof}
Similar as the previous proof except that we group the $\decoration$-permutrees according to their topmost~\downCirc{} vertex. Details are left to the reader.
\end{proof}

We conclude by the following statement which sums up the results of this numerology section.

\begin{corollary}
For any decoration~$\decoration$, the factorial-Catalan number~$\factCatalan{\decoration}$ is given by the recurrence formula
\[
\factCatalan{\decoration} = \prod_{k \in [m]} \sum_{\substack{i \in [b_{k-1}, b_k] \cap \decoration^{-1}(\downCirc{}) \\ J \subseteq [b_{k-1}, b_k] \cap \decoration^{-1}(\noneCirc{})}} \factCatalan{\decoration_{|[b_{k-1}, \dots, i-1] \ssm J}} \cdot \factCatalan{\decoration_{|[i+1, \dots, b_k] \ssm J}} \cdot |J|!
\]
where~$\{b_0 < b_1 < \dots < b_m\} = \{0,n\} \cup \decoration^{-1}(\upDownCirc{})$.
\end{corollary}

\begin{example}
\label{exm:numerologySpecificDecorations}
Following Example~\ref{exm:permutreesSpecificDecorations}, the previous statements specialize to classical formulas for
\begin{enumerate}[(i)]
\item the factorial~$n! = \factCatalan{\noneCirc{}^n}$ \href{https://oeis.org/A000142}{\cite[A000142]{OEIS}},
\item the Catalan number~$\frac{1}{n+1}\binom{2n}{n} = \factCatalan{\downCirc{}^n} = \factCatalan{\decoration}$ for~$\decoration \in \{\downCirc{}, \upCirc{}\}^n$ \href{https://oeis.org/A000108}{\cite[A000108]{OEIS}}, and
\item the power~$2^{n-1} = \factCatalan{\upDownCirc{}^n}$ \href{https://oeis.org/A000079}{\cite[A000079]{OEIS}}.
\end{enumerate}
Other relevant subfamilies of decorations will naturally appear in Section~\ref{sec:permutreeHopfAlgebra}. Namely, we are particularly interested in the subfamilies of decorations given by all words in~$D^n$ on a given subset~$D$ of~$\Decorations$. See Table~\ref{table:factorialCatalanFamilies}.

\begin{table}
    \centerline{
    $\begin{array}[t]{c|r@{\quad}r@{\quad}r@{\quad}r@{\quad}r@{\quad}r@{\quad}r@{\quad}r|c}
    D & \multicolumn{8}{|c|}{\sum_{\decoration \in D^n} \factCatalan{\decoration} \quad \text{for } n \in [10]} & \text{reference} \\[.1cm]
    \hline
    \Decorations & 4 & 32 & 320 & 3584 & 43264 & 553472 & 7441920 & 104740864 & - \\
    \hline
    \{\noneCirc{}, \upCirc{}, \downCirc{}\} & 3 & 18 & 144 & 1368 & 14688 & 173664 & 2226528 & 30647808 & - \\    
    \{\upCirc{}, \downCirc{}, \upDownCirc{}\} & 3 & 18 & 126 & 936 & 7164 & 55800 & 439560 & 3489696 & - \\    
    \{\noneCirc{}, \upCirc{}, \upDownCirc{}\} \text{ or } \{\noneCirc{}, \downCirc{}, \upDownCirc{}\} & 3 & 18 & 135 & 1134 & 10287 & 99306 & 1014039 & 10933542 & - \\    
    \hline
    \{\noneCirc{}, \upCirc{}\} \text{ or } \{\noneCirc{}, \downCirc{}\} & 2 & 8 & 44 & 296 & 2312 & 20384 & 199376 & 2138336 & \text{\href{https://oeis.org/A077607}{\cite[A077607]{OEIS}}} \\
    \{\noneCirc{}, \upDownCirc{}\} & 2 & 8 & 40 & 224 & 1360 & 8864 & 61984 & 467072 & - \\
    \{\upCirc{}, \upDownCirc{}\} \text{ or } \{\downCirc{}, \upDownCirc{}\} & 2 & 8 & 36 & 168 & 796 & 3800 & 18216 & 87536 & \text{\href{https://oeis.org/A084868}{\cite[A084868]{OEIS}}} \\
    \{\upCirc{}, \downCirc{}\} & 2 & 8 & 40 & 224 & 1344 & 8448 & 54912 & 366080 & \text{\href{https://oeis.org/A052701}{\cite[A052701]{OEIS}}} \\
    \hline
    \{\noneCirc{}\} & 1 & 2 & 6 & 24 & 120 & 720 & 5040 & 40320 & \text{\href{https://oeis.org/A000142}{\cite[A000142]{OEIS}}} \\
    \{\upCirc{}\} \text{ or } \{\downCirc{}\} & 1 & 2 & 5 & 14 & 42 & 132 & 429 & 1430 & \text{\href{https://oeis.org/A000108}{\cite[A000108]{OEIS}}} \\
    \{\upDownCirc{}\} & 1 & 2 & 4 & 8 & 16 & 32 & 64 & 128 & \text{\href{https://oeis.org/A000079}{\cite[A000079]{OEIS}}} \\
    \end{array}$}
  \caption{Sums of the factorial-Catalan numbers over subfamilies of decorations given by all words on a given subset of~$\Decorations$.}
  \label{table:factorialCatalanFamilies}
\end{table}
\end{example}

\begin{remark}
In general, the factorial-Catalan number is always bounded by~$2^{n-1} \le \factCatalan{\decoration} \le n!$. Moreover, we even have~$\factCatalan{\decoration} \le \frac{1}{n+1}\binom{2n}{n}$ if~$\decoration$ contains no~$\noneCirc{}$, while $\factCatalan{\decoration} \ge \frac{1}{n+1}\binom{2n}{n}$ if~$\decoration$ contains no~$\upDownCirc{}$. This follows from Corollary~\ref{coro:inductionFactCatalan} but will even be easier to see from decoration refinements in the next section.
\end{remark}

We refer to Table~\ref{table:factorialCatalan} for the values of all factorial-Catalan numbers~$\factCatalan{\decoration}$ for all decorations~$\decoration$ such that~$|\decoration| \in \{3,4,5,6\}$.


\subsection{Rotations and permutree lattices}
\label{subsec:rotations}

We now extend the rotation from binary trees to all permutrees. This local operation only exchanges the orientation of an edge and rearranges the endpoints of two other edges.

\begin{definition}
\label{def:rotation}
Let~$i \to j$ be an edge in a $\decoration$-permutree~$\tree$, with~$i < j$. Let~$D$ denote the only (resp.~the right) descendant subtree of vertex~$i$ if~$\decoration_i \in \{\noneCirc{}, \upCirc{}\}$ (resp.~if~$\decoration_i \in \{\downCirc{}, \upDownCirc{}\}$) and let~$U$ denote the only (resp.~the left) ancestor subtree of vertex~$j$ if~$\decoration_j \in \{\noneCirc{}, \downCirc{}\}$ (resp.~if~$\decoration_j \in \{\upCirc{}, \upDownCirc{}\}$). Let~$\tree'$ be the oriented tree obtained from~$\tree$ just reversing the orientation of~$i \to j$ and attaching the subtree~$U$ to~$i$ and the subtree~$D$ to~$j$. The transformation from~$\tree$ to~$\tree'$ is called \defn{rotation} of the edge~$i \to j$. See \fref{fig:rotation}.

\begin{figure}[h]
  \centerline{\includegraphics[scale=1]{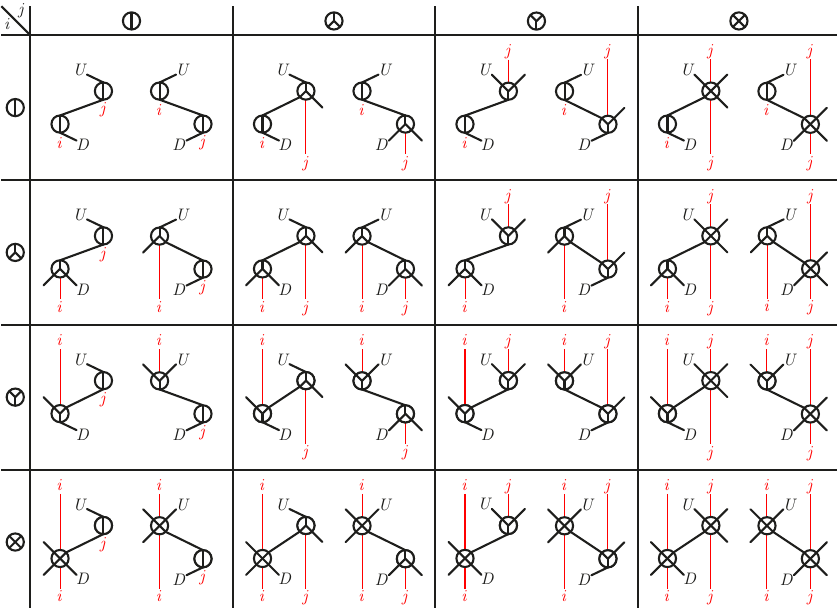}}
  \caption{Rotations in permutrees: in each box, the tree~$\tree$ (left) is transformed into the tree~$\tree'$ (right) by rotation of the edge~$i \to j$. The $16$ boxes correspond to the possible decorations of~$i$~and~$j$.}
  \label{fig:rotation}
\end{figure}
\end{definition}

The following statement shows that the rotation of the edge~$i \to j$ is the only operation which exchanges the orientation of this edge while preserving all other edge cuts. An \defn{edge cut} in a permutree~$\tree$ is the ordered partition~$\edgecut{I}{J}$ of the vertices of~$\tree$ into the set~$I$ of vertices in the source set and the set~$J = [n] \ssm I$ of vertices in the target set of an oriented edge of~$\tree$.

\begin{proposition}
\label{prop:rotation}
The result~$\tree'$ of the rotation of an edge~$i \to j$ in a $\decoration$-permutree~$\tree$ is a \mbox{$\decoration$-permutree}. Moreover, $\tree'$ is the unique $\decoration$-permutree with the same edge cuts as~$\tree$, except the cut defined by the edge~$i \to j$.
\end{proposition}

\begin{proof}
We first observe that~$\tree'$ is still a tree, with an orientation of its edges and a bijective labeling of its vertices by~$[n]$. To check that~$\tree'$ is still a $\decoration$-permutree, we need to check the local conditions of Definition~\ref{def:permutree} around each vertex of~$\tree'$. Since we did not perturb the labels nor the edges incident to the vertices of~$\tree$ distinct from~$i$ and~$j$, it suffices to check the local conditions around~$i$ and~$j$. By symmetry, we only give the arguments around~$i$. We distinguish two cases:
\begin{itemize}
\item If $i$ is an down vertex in~$\tree$ (decorated by~\downCirc{} or~\upDownCirc{}), then~$D$ is the right descendant of~$i$ in~$\tree$, so that all labels in~$D$ are larger than~$i$. Moreover, all labels in the right ancestor and descendant of~$j$ (if any) are larger than~$j$, which is in turn larger than~$i$. It follows that all labels in the right descendant of~$i$ in~$\tree'$ are larger than~$i$. Finally, the left descendant of~$i$ in~$\tree'$ is the left descendant of~$i$ in~$\tree$, so all its labels are still smaller than~$i$.
\item If~$i$ is an up vertex in~$\tree$ (decorated by~\upCirc{} or~\upDownCirc{}), then~$U$ belongs to the right ancestor of~$i$ in~$\tree$, so that all labels in~$U$ are larger than~$i$. It follows that all labels in the right ancestor of~$i$ in~$\tree'$ are larger than~$i$. Finally, the left ancestor of~$i$ in~$\tree'$ is the left ancestor of~$i$ in~$\tree$, so all its labels are still smaller than~$i$.
\end{itemize}
This closes the proof that~$\tree'$ is a $\decoration$-permutree.
By construction, the $\decoration$-permutree~$\tree'$ clearly has the same edge cuts as~$\tree$, except the cut corresponding to the edge~$i \to j$. Any $\decoration$-permutree with this property is obtained from~$\tree$ by reversing the edge~$i \to j$ to an edge~$i \leftarrow j$ and rearranging the neighbors of~$i$ and~$j$. But it is clear that the rearrangement given in Definition~\ref{def:rotation} is the only one which preserves the local conditions of Definition~\ref{def:permutree} around the vertices~$i$ and~$j$.
\end{proof}

\begin{remark}
Following Remark~\ref{rem:234angulations}, a rotation on permutrees is dual to a flip on $\{2,3,4\}$-angulations of~$\b{P}_\decoration$. Namely, after deletion of an internal arc from a $\{2,3,4\}$-angulation, there is a unique distinct internal arc that can be inserted to complete to a new~$\{2,3,4\}$-angulation. See \fref{fig:flip}.

\begin{figure}[h]
  \centerline{\includegraphics{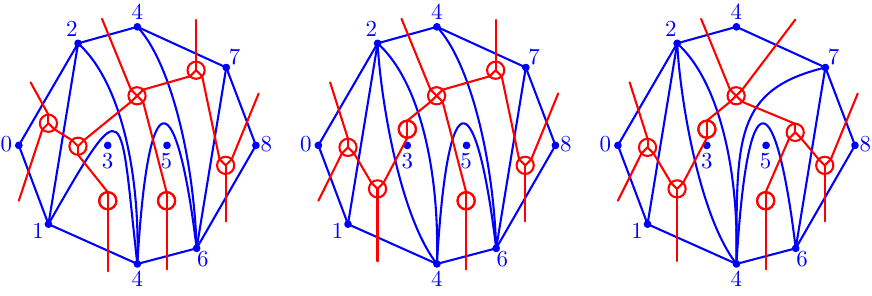}}
  \caption{A sequence of two rotations on permutrees, and the corresponding flips on $\{2,3,4\}$-angulations.}
  \label{fig:flip}
\end{figure}
\end{remark}

Define the \defn{increasing rotation graph} on~$\Permutrees(\decoration)$ to be the graph whose vertices are the $\decoration$-permutrees and whose arcs are increasing rotations~$\tree \to \tree'$, \ie where the edge~$i \to j$ in~$\tree$ is reversed to the edge~$i \leftarrow j$ in~$\tree'$ for~$i < j$. See \fref{fig:permutreeLattices}. The following statement, adapted from N.~Reading's work~\cite{Reading-CambrianLattices}, asserts that this graph is acyclic, that its transitive closure defines a lattice, and that this lattice is closely related to the weak order. See \fref{fig:permutreeLattices}.

\hvFloat[floatPos=p, capWidth=h, capPos=r, capAngle=90, objectAngle=90, capVPos=c, objectPos=c]{figure}
{\begin{minipage}{25cm}\vspace*{-1.2cm}\centerline{\includegraphics[scale=.9]{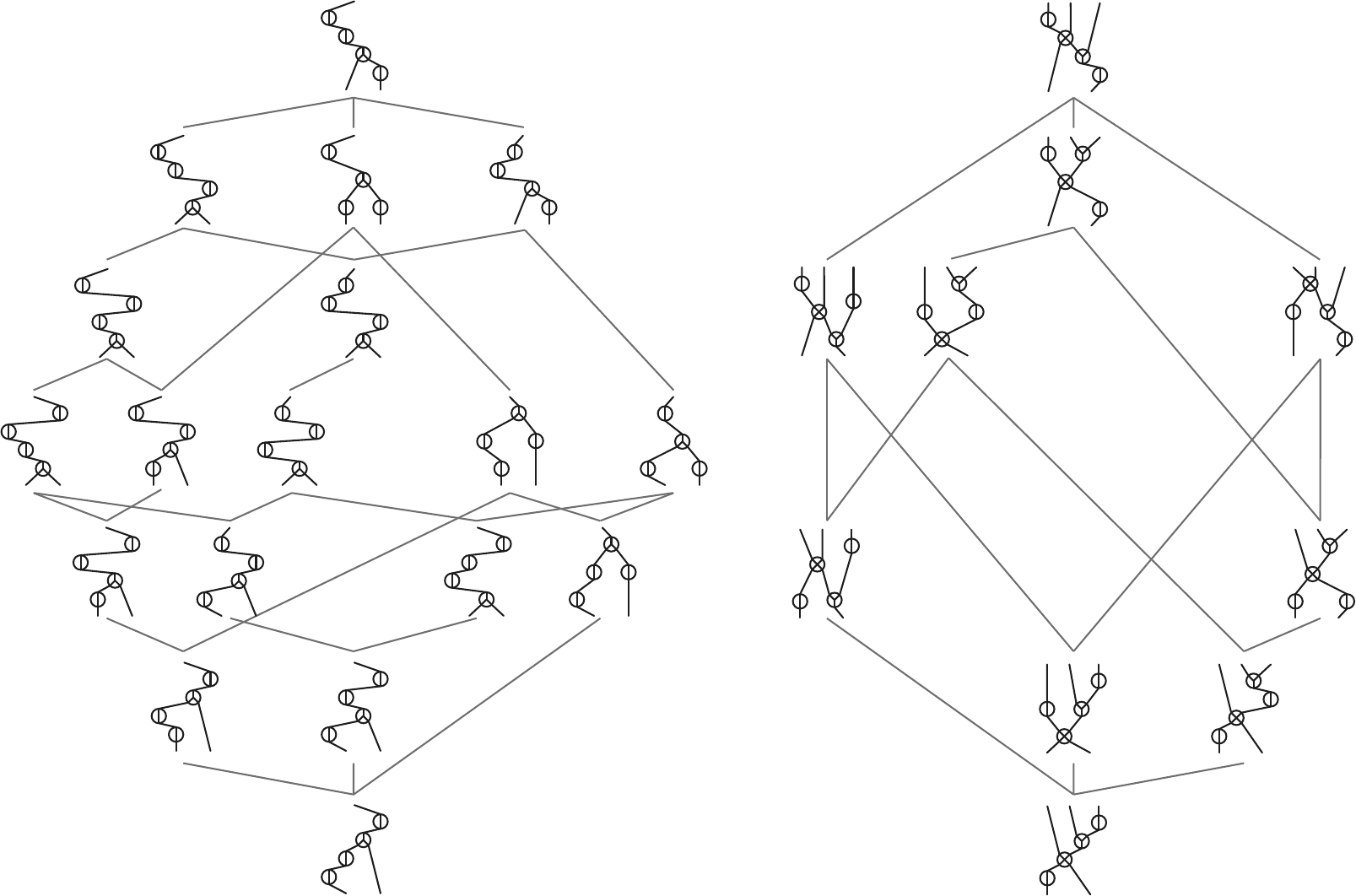}}\end{minipage}}
{The $\decoration$-permutree lattices, for the decorations $\decoration = \noneCirc{}\noneCirc{}\downCirc{}\noneCirc{}$ (left) and $\decoration = \noneCirc{}\upDownCirc{}\upCirc{}\noneCirc{}$~(right).}
{fig:permutreeLattices}

\begin{proposition}
\label{prop:permutreeLattice}
The transitive closure of the increasing rotation graph on~$\Permutrees(\decoration)$ is a lattice, called \defn{$\decoration$-permutree lattice}. The map~$\PSymbol : \fS^\decoration \to \Permutrees(\decoration)$ defines a lattice homomorphism from the weak order on~$\fS^\decoration$ to the $\decoration$-permutree lattice on~$\Permutrees(\decoration)$.
\end{proposition}

\begin{proof}
For two $\decoration$-permutree congruence classes~$X,Y$, define~$X \le Y$ if and only if there are representatives~$\sigma \in X$ and~$\tau \in Y$ such that~$\sigma \le \tau$ in weak order. By Proposition~\ref{prop:weakOrderCongruence}, this order defines a lattice on the $\decoration$-permutree congruence classes.
Moreover, these classes correspond to the $\decoration$-permutrees by Proposition~\ref{prop:permutreeCongruenceClass}. Therefore, we only have to show that the quotient order on these classes coincides with the order given by increasing flips on the $\decoration$-permutrees. For this, it suffices to check that if two permutations~$\sigma$ and~$\tau$ differ by the transposition of two consecutive values~$i<j$, then the $\decoration$-permutrees~$\PSymbol(\sigma)$ and~$\PSymbol(\tau)$ either coincide or differ by the rotation of the edge~$i \to j$. If~$\sigma \equiv_\decoration \tau$, then we have~$\PSymbol(\sigma) = \PSymbol(\tau)$ by Proposition~\ref{prop:permutreeCongruenceClass}. If~$\sigma \not\equiv_\decoration \tau$, then~$i$ and~$j$ are consecutive values comparable in~$\PSymbol(\sigma)$, so that $\PSymbol(\sigma)$ contains the edge~$i \to j$. We then check locally that the effect of switching~$i$ and~$j$ in the insertion process precisely rotates this edge in~$\PSymbol(\sigma)$.
\end{proof}

Note that the minimal (resp.~maximal) $\decoration$-permutree is an oriented path from~$1$ to~$n$ (resp.~from~$n$ to~$1$) with an additional incoming leaf at each down vertex (decorated by~\downCirc{} or~\upDownCirc{}) and an additional outgoing leaf at each up vertex (decorated by~\upCirc{} or~\upDownCirc{}). See \fref{fig:permutreeLattices}.

\begin{example}
Following Example~\ref{exm:permutreesSpecificDecorations},  the $\decoration$-permutree lattice is:
\begin{enumerate}[(i)]
\item the weak order on~$\fS_n$ when~$\decoration = \noneCirc{}^n$,
\item the Tamari lattice defined by right rotations on binary trees when~$\decoration = \downCirc{}^n$,
\item the (type~$A$) Cambrian lattices of N.~Reading~\cite{Reading-CambrianLattices} when~$\decoration \in \{\downCirc{}, \upCirc{}\}^n$,
\item the boolean lattice when~$\decoration = \upDownCirc{}^n$.
\end{enumerate}
\end{example}

\begin{remark}
Note that different decorations generally give rise to different permutree lattices. In fact, although it preserves the number of $\decoration$-permutrees by Corollary~\ref{coro:equienumerated} and the number of Schr\"oder $\decoration$-permutrees by Corollary~\ref{coro:equienumeratedSchroder}, changing a~\downCirc{} to a~\upCirc{} in the decoration~$\decoration$ may change both the $\decoration$-permutree lattice and the rotation graph on $\decoration$-permutrees. For example, the rotation graphs for the decoration~$\noneCirc{}\noneCirc{}\downCirc{}\downCirc{}\noneCirc{}\noneCirc{}$ and~$\noneCirc{}\noneCirc{}\downCirc{}\upCirc{}\noneCirc{}\noneCirc{}$ both have~$248$ vertices but are not isomorphic. Note however that the symmetrees discussed in Remark~\ref{rem:symmetrees} provide lattice (anti)-isomorphisms between the permutree lattices for the decorations~$\decoration$, $\horizontalSymmetry{\decoration}$ and~$\verticalSymmetry{\decoration}$.
\end{remark}


\subsection{Decoration refinements}
\label{subsec:decorationRefinements}

Order the possible decorations by~$\noneCirc{} \less \{\downCirc{}, \upCirc{}\} \less \upDownCirc{}$ (in other words, by increasing number of incident edges). Let~$\decoration, \decoration'$ be two decorations of the same size~$n$. If~$\decoration_i \less \decoration'_i$ for all~$i \in [n]$, then we say that~$\decoration$ \defn{refines}~$\decoration'$, or that~$\decoration'$ \defn{coarsens}~$\decoration$, and we write~$\decoration \less \decoration'$. The set of all decorations of size~$n$ ordered by refinement~$\less$ is a boolean lattice with minimal element~$\noneCirc{}^n$ and maximal element~$\upDownCirc{}^n$.

For any permutation~$\tau \in \fS_n$ and any decoration~$\decoration \in \Decorations^n$, we denote by~$\tau^\decoration \in \fS^\decoration$ the permutation~$\tau$ whose values get decorated by~$\decoration$. Observe now that for any two decorations~$\decoration \less \decoration'$, the $\decoration$-permutree congruence refines the $\decoration'$-permutree congruence: $\sigma^\decoration \equiv_\decoration \tau^\decoration \Longrightarrow \sigma^{\decoration'} \equiv_{\decoration'} \tau^{\decoration'}$. See \fref{fig:fibersPermutreeCongruences} for an illustration of the refinement lattice on the $\decoration$-permutree congruences for all decorations~$\decoration \in \noneCirc{} \cdot \Decorations^2 \cdot \noneCirc{}$.

In other words, all linear extensions of a given $\decoration$-permutree~$\tree$ are linear extensions of the same \mbox{$\decoration'$-permutree~$\tree'$}. This defines a natural surjection map~$\surjection : \Permutrees(\decoration) \to \Permutrees(\decoration')$ from $\decoration$-permutrees to $\decoration'$-permutrees for any two decorations~$\decoration \less \decoration'$. Namely, the image~$\surjection(\tree)$ of any $\decoration$-permutree~$\tree$ is obtained by inserting any linear extension of~$\tree$ seen as a permutation decorated by~$\decoration'$.

This surjection~$\surjection$ can be described visually as follows, see \fref{fig:refinement}. We start from a $\decoration$-permutree, with vertices labeled from left to right as usual. We then redecorate its vertices according to~$\decoration'$ and place the corresponding vertical red walls below the down vertices (decorated by~\downCirc{} or~\upDownCirc{}) and above the up vertices (decorated by~\upCirc{} or~\upDownCirc{}). The result is not a permutree at the moment as some edges of the tree cross some red walls. In order to fix it, we cut the edges crossing red walls and reconnect them with vertical segments as illustrated in \fref{fig:refinement}\,(middle right). Finally, we stretch the picture to see a $\decoration'$-permutree with our usual straight edges.

\begin{figure}[t]
  \centerline{\includegraphics{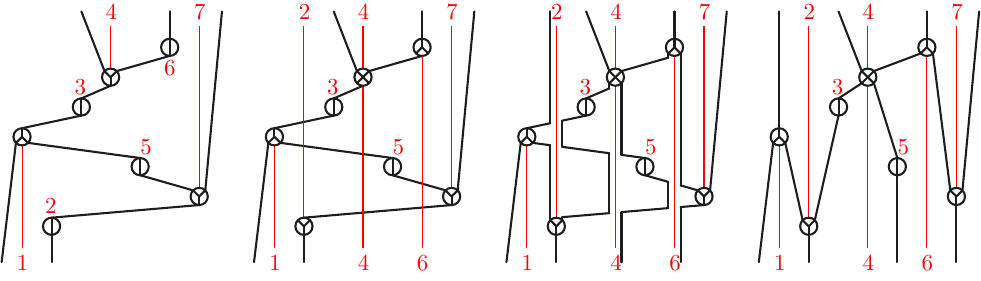}}
  \caption{Refinement by cut and stretch: starting from a permutree decorated with $\downCirc{}\noneCirc{}\noneCirc{}\upCirc{}\noneCirc{}\noneCirc{}\upCirc{}$ (left), we redecorate its vertices with~$\downCirc{}\upCirc{}\noneCirc{}\upDownCirc{}\noneCirc{}\downCirc{}\upCirc{}$ (middle left), cut and reconnect its edges along the resulting red walls (middle right), and stretch the resulting $\downCirc{}\upCirc{}\noneCirc{}\upDownCirc{}\noneCirc{}\downCirc{}\upCirc{}$-permutree (right).}
  \label{fig:refinement}
\end{figure}



Note that if~$\equiv$ and~$\approx$ are two lattice congruences of the same lattice~$L$ and~$\equiv$ refines~$\approx$, then the map sending a class of~$L/{\equiv}$ to its class in~$L/{\approx}$ is a lattice homormorphism. Since the $\decoration$-permutree lattice is isomorphic to the quotient of the weak order by the $\decoration$-permutree congruence by Proposition~\ref{prop:permutreeLattice}, we obtain the following statement.

\begin{proposition}
\label{prop:refinementLatticeHomomorphism}
The surjection~$\surjection$ defines a lattice homomorphism from the $\decoration$-permutree lattice to the $\decoration'$-permutree lattice.
\end{proposition}

\begin{example}
When~$\decoration = \noneCirc{}^n$, the surjection~$\surjection$ is just the $\PSymbol$-symbol described in Section~\ref{subsec:correspondence}. The reader can thus apply the previous description to see pictorially the classical maps described in Example~\ref{exm:insertionSpecificDecorations} (BST insertion, descents). For example, we have illustrated in \fref{fig:bstinsertion} an insertion of a permutation in a binary tree seen with this cut and stretch interpretation. Besides these maps, the surjection~$\surjection$ specializes when~${\decoration = \downCirc{}^n}$ and~${\decoration' = \upDownCirc{}^n}$ to the classical \defn{canopy} map from binary trees to binary sequences, see~\cite{LodayRonco, Viennot}.

\begin{figure}[h]
  \centerline{\includegraphics[scale=1]{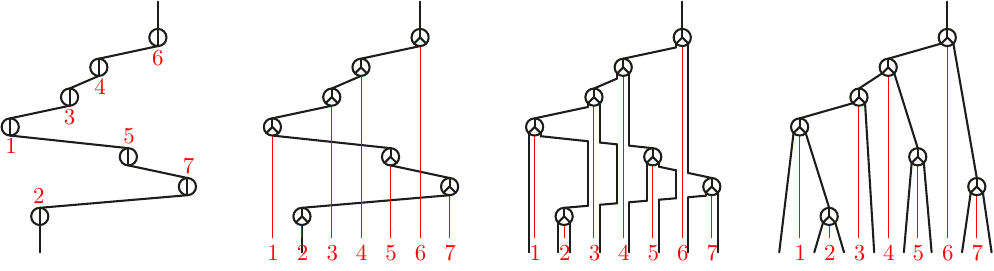}}
  \caption{BST insertion seen as the refinement map from a $\noneCirc{}^n$-permutree (\ie a permutation of~$\fS_n$) to a $\downCirc{}^n$-permutree (\ie a binary tree on~$[n]$).}
  \label{fig:bstinsertion}
\end{figure}
\end{example}

\hvFloat[floatPos=p, capWidth=h, capPos=r, capAngle=90, objectAngle=90, capVPos=c, objectPos=c]{figure}
{\begin{minipage}{25cm}\vspace*{-2.1cm}\centerline{\includegraphics[scale=.39]{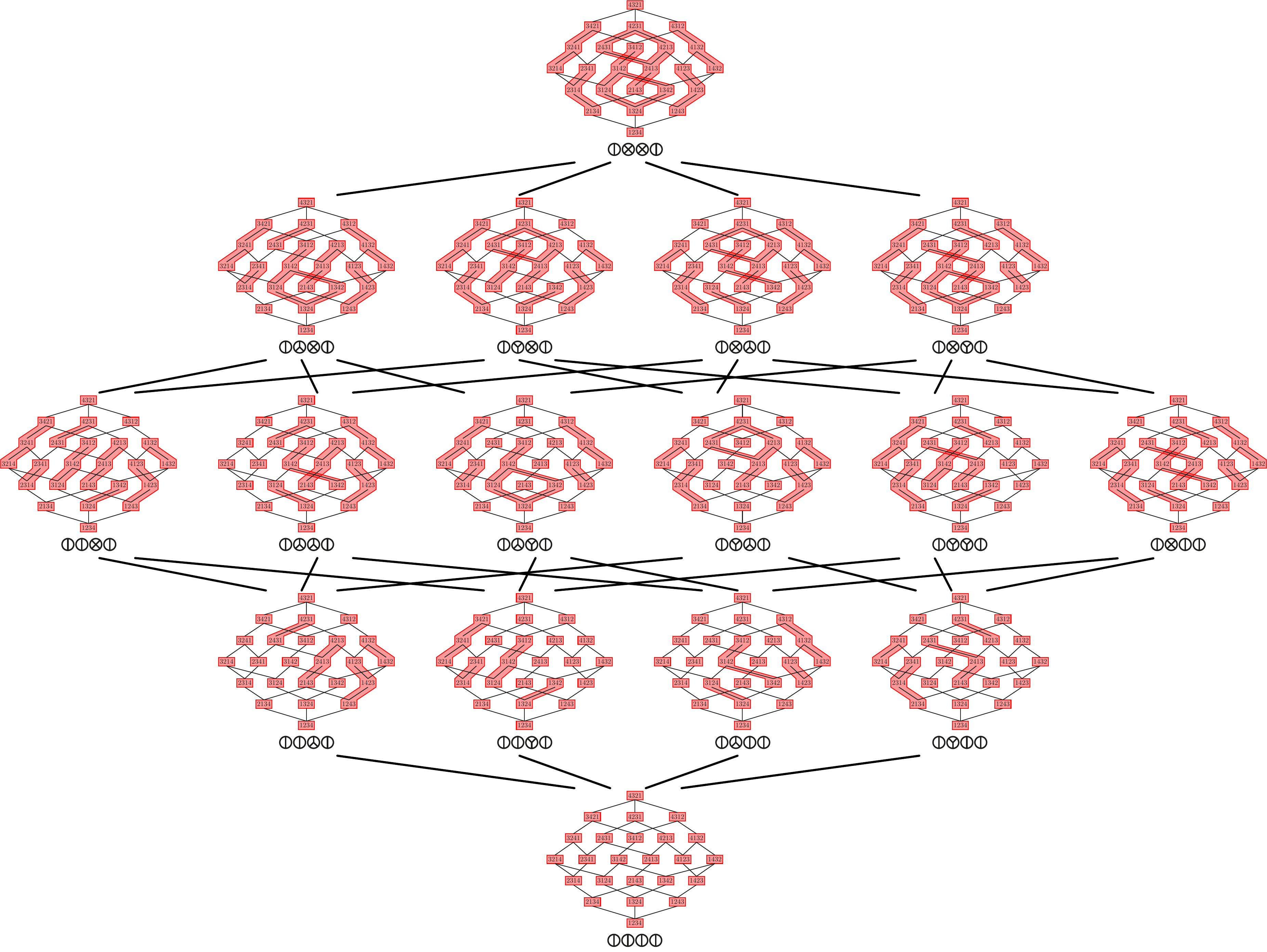}}\end{minipage}}
{The fibers of the $\decoration$-permutree congruence, for all decorations~$\decoration \in \noneCirc{} \cdot \Decorations^2 \cdot \noneCirc{}$.}
{fig:fibersPermutreeCongruences}


\section{Permutreehedra}
\label{sec:permutreehedra}

In this section, we show that the known geometric constructions of the associahedron~\cite{Loday, HohlwegLange, LangePilaud} extend in our setting. We therefore obtain a family of polytopes which interpolate between the permutahedron, the associahedra, and some graphical zonotopes. We call these polytopes \defn{permutreehedra} .

We refer to~\cite{Ziegler} or~\cite[Chapter~5]{Matousek} for background on polytopes and fans. Let us just remind that a \defn{polytope} is a subset~$P$ in~$\R^n$ defined equivalently as the convex hull of finitely many points in~$\R^n$ or as a bounded intersection of finitely many closed half-spaces of~$\R^n$. The faces of~$P$ are the intersections of~$P$ with its supporting hyperplanes. A \defn{polyhedral fan} is a collection of polyhedral cones of~$\R^n$ closed under faces and which intersect pairwise along faces. The (outer) \defn{normal cone} of a face~$F$ of~$P$ is the cone generated by the outer normal vectors of the facets (codimension~$1$ faces) of~$P$ containing~$F$. Finally, the (outer) \defn{normal fan} of~$P$ is the collection of the (outer) normal cones of all its faces.

We denote by~$(\b{e}_i)_{i \in [n]}$ the canonical basis of~$\R^n$ and let~$\one \eqdef \sum_{i \in [n]} \b{e}_i$. All our constructions will lie in the affine subspace
\[
\HH \eqdef \biggset{\b{x} \in \R^n}{\sum_{i \in [n]} x_i = \binom{n+1}{2}}.
\]


\subsection{Permutree fans}
\label{subsec:permutreeFans}

We consider the (type~$A$) Coxeter arrangement defined by the hyperplanes~$\set{\b{x} \in \R^n}{x_i = x_j}$ for all~$1 \le i < j \le n$. It defines a complete simplicial fan, called the \defn{braid fan}. Each maximal cone~$\Cone$ of this fan corresponds to a permutation given by the order of the coordinates of any point of~$\Cone$. More precisely, we can associate to each permutation~$\tau$ the maximal cone~$\Cone\polar(\tau) \eqdef \set{\b{x} \in \R^n}{x_{\tau^{-1}(1)} \le \dots \le x_{\tau^{-1}(n)}}$ of the braid fan.

Define now the \defn{incidence cone}~$\Cone(\tree)$ and the \defn{braid cone}~$\Cone\polar(\tree)$ of a permutree~$\tree$ as the cones
\begin{align*}
\Cone(\tree)
& \eqdef \frac{n+1}{2}\one + \cone \set{\b{e}_i-\b{e}_j}{\forall\, i \to j \text{ in } \tree} = \biggset{\b{x} \in \HH}{\sum_{j \in J} \frac{x_j}{|J|} \le \sum_{i \in I} \frac{x_i}{|I|}, \, \forall \edgecut{I}{J} \in \cuts(\tree)} \\[.1cm]
\Cone\polar(\tree)
& \eqdef \set{\b{x} \in \HH}{x_i \le x_j, \, \forall \, i \to j \text{ in } \tree} = \frac{n+1}{2}\one + \cone \biggset{\sum_{j \in J} \frac{\b{e}_j}{|J|} - \sum_{i \in I} \frac{\b{e}_i}{|I|}}{\forall \edgecut{I}{J} \in \cuts(\tree)},
\end{align*}
where~$\cuts(\tree)$ denotes the set of edge cuts of~$\tree$. Note that these two cones both lie in the space~$\HH$, are simplicial, and are polar to each other. We use these cones to construct the permutree~fan.

\begin{proposition}
\label{prop:permutreeFan}
For any decoration~$\decoration \in \Decorations^n$, the set of cones~$\set{\Cone\polar(\tree)}{\tree \in \Permutrees(\decoration)}$, together with all their faces, forms a complete simplicial fan~$\Fan$ of~$\HH$ called \defn{$\decoration$-permutree fan}.
\end{proposition}

\begin{proof}
The statement is a special case of a general result of N.~Reading~\cite[Theorem~1.1]{Reading-HopfAlgebras}, but we give a direct short proof for the convenience of the reader. A cone~$\Cone\polar(\tree)$ is the union of the cones~$\Cone\polar(\tau)$ over all linear extensions~$\tau$ of~$\tree$. Since the sets of linear extensions of the $\decoration$-permutrees form a partition of~$\fS_n$ ($\decoration$-permutree congruence classes), we already know that the cones~$\Cone\polar(\tree)$ for all~$\tree \in \Permutrees(\decoration)$ are interior disjoint and cover the all space~$\R^n$. Now any $\decoration$-permutree~$\tree$ is adjacent by rotation to~$n$ other $\decoration$-permutrees~$\tree_1, \dots, \tree_n$. If~$\tree_k$ is obtained from~$\tree$ by rotating the edge~$i \to j$ with edge cut~$\edgecut{I}{J}$, it has precisely the same edge cuts as~$\tree$ except the edge cut~$\edgecut{I}{J}$. The cone~$\Cone\polar(\tree_k)$ thus shares all but one ray of~$\Cone\polar(\tree)$. We conclude that each facet of~$\Cone\polar(\tree)$ is properly shared by another cone~$\Cone\polar(\tree')$. Since the cones~$\Cone\polar(\tree)$ for all~$\tree \in \Permutrees(\decoration)$ are interior disjoint, it shows that no such cone can improperly share a (portion of a) facet with~$\Cone\polar(\tree)$. We obtain that all cones intersect properly which concludes the proof that we have a complete simplicial fan.
\end{proof}

\begin{proposition}
\label{prop:refinementPermutreeFan}
Consider two decorations~$\decoration \less \decoration'$ and the surjection~$\surjection : \Permutrees(\decoration) \to \Permutrees(\decoration')$ defined in Section~\ref{subsec:decorationRefinements}. For any $\decoration$-permutree~$\tree$ and its image~$\tree' = \surjection(\tree)$, we have
\[
\Cone(\tree) \supseteq \Cone(\tree')
\qquad\text{and}\qquad
\Cone\polar(\tree) \subseteq \Cone\polar(\tree').
\]
In other words, the $\decoration$-permutree fan~$\Fan$ refines the $\decoration'$-permutree fan~$\Fan[\decoration']$.
\end{proposition}

\begin{proof}
We have seen in the definition of~$\surjection$ that all linear extensions of~$\tree$ are linear extensions of~$\tree' = \surjection(\tree)$. Since the cone~$\Cone\polar(\tree)$ is the union of the cones~$\Cone\polar(\tau)$ over all linear extensions~$\tau$ of~$\tree$, we obtain that~$\Cone\polar(\tree) \subseteq \Cone\polar(\tree')$. The inclusion~$\Cone(\tree) \supseteq \Cone(\tree')$ follows by polarity.
\end{proof}

\begin{example}
Following Example~\ref{exm:permutreesSpecificDecorations}, the $\decoration$-permutree fans specialize to:
\begin{enumerate}[(i)]
\item the braid fan when~$\decoration = \noneCirc{}^n$,
\item the (type~$A$) Cambrian fans of N.~Reading and D.~Speyer~\cite{ReadingSpeyer} when~$\decoration \in \{\downCirc{}, \upCirc{}\}^n$,
\item the fan defined by the hyperplane arrangement~$x_i = x_{i+1}$ for each~$i \in [n-1]$ when~$\decoration = \upDownCirc{}^n$,
\item the fan defined by the hyperplane arrangement~$x_i = x_{j}$ for each~$i < j \in [n-1]$ such that~$\decoration_{|(i,j)} = \noneCirc{}^{j-i-1}$ when~$\decoration \in \{\noneCirc{},\upDownCirc{}\}^n$.
\end{enumerate}
\end{example}


\subsection{Permutreehedra}
\label{subsec:permutreehedra}

We are now ready to construct the $\decoration$-permutreehedron whose normal fan is the $\decoration$-permutree fan. As for J.-L.~Loday's or C.~Hohlweg and C.~Lange's associahedra~\cite{Loday, HohlwegLange}, our permutreehedra are obtained by deleting certain inequalities in the facet description of the classical permutahedron. We thus first recall the vertex and facet descriptions of this polytope. The \defn{permutahedron}~$\Perm$ is the polytope obtained as
\begin{enumerate}[(i)]
\item either the convex hull of the vectors~$\b{p}(\tau) \eqdef [\tau^{-1}(i)]_{i \in [n]} \in \R^n$, for all permutations~$\tau \in \fS_n$,
\item or the intersection of the hyperplane~$\HH = \Hyp([n])$ with the half-spaces~$\HS(J)$ for~$\varnothing \ne J \subseteq \ground$, where
\[
\Hyp(J) \eqdef \biggset{\b{x} \in \R^n}{\sum_{j \in J} x_j = \binom{|J|+1}{2}} \quad \text{and} \quad \HS(J) \eqdef \biggset{\b{x} \in \R^n}{\sum_{j \in J} x_j \ge \binom{|J|+1}{2}}.
\]
\end{enumerate}
Its normal fan is precisely the braid fan described in the previous section. An illustration of~$\Perm[4]$ is given on the bottom of \fref{fig:permutreehedra}.

From this polytope, we construct the \defn{$\decoration$-permutreehedron}~$\Permutreehedron$, for which we give both vertex and facet descriptions:
\begin{enumerate}[(i)]
\item The vertices of~$\Permutreehedron$ correspond to $\decoration$-permutrees. We associate to a $\decoration$-permutree~$\tree$ a point~$\b{a}(\tree) \in \R^n$ whose coordinates are defined by
\[
\b{a}(\tree)_i =
\begin{cases}
1 + d & \text{if } \decoration_i = \noneCirc{}, \\
1 + d + \down{\ell}\down{r} & \text{if } \decoration_i = \downCirc{}, \\
1 + d - \up{\ell}\up{r} & \text{if } \decoration_i = \upCirc{}, \\
1 + d + \down{\ell}\down{r} - \up{\ell}\up{r} & \text{if } \decoration_i = \upDownCirc{},
\end{cases}
\]
where~$d$ denotes the number of the descendants~$i$ in~$\tree$, $\down{\ell}$ and~$\down{r}$ denote the sizes of the left and right descendant subtrees of~$i$ in~$\tree$ when~$\decoration_i \in \{\downCirc{}, \upDownCirc{}\}$, and $\up{\ell}$ and~$\up{r}$ denote the sizes of the left and right ancestor subtrees of~$i$ in~$\tree$ when~$\decoration_i \in \{\upCirc{}, \upDownCirc{}\}$. Note that~$\b{a}(\tree)$ is independent of the decorations of the first and last vertices of~$\tree$.
\item The facets of~$\Permutreehedron$ correspond to the \defn{$\decoration$-building blocks}, that is, to all subsets~$I \subset [n]$ such that there exists a $\decoration$-permutree~$\tree$ which admits~$\edgecut{I}{J}$ as an edge cut. We associate to a $\decoration$-building block~$I$ the hyperplane~$\Hyp(I)$ and the half-space~$\HS(I)$ defined above for the permutahedron. Note that~$\decoration$-building blocks are easy to compute on the dual representation described in Remark~\ref{rem:234angulations}: any internal arc~$\alpha$ in~$\b{P}_\decoration$ corresponds to the building block~$I_\alpha$ of the indices of all points~$\b{p}_k$ below~$\alpha$, including the endpoints of~$\alpha$ in the upper convex hull of~$\b{P}_\decoration$ (and excluding~$0$ and~$n+1$). For example, when~$\decoration \in \{\noneCirc, \downCirc\}^n$, the building blocks are the subsets~$I$ such that~$I \cap \decoration^{-1}(\downCirc)$ is an interval of~$\decoration^{-1}(\downCirc)$.
\end{enumerate}
As an illustration, the vertex corresponding to the permutree of \fref{fig:leveledPermutree} is~$[7, -4, 3, 8, 1, 12, 1]$ and the facet corresponding to the edge~$3 \to 4$ in the permutree of \fref{fig:leveledPermutree} is~$x_1+x_2+x_3 \ge 6$.

\begin{theorem}
\label{theo:permutreehedra}
The permutree fan~$\Fan$ is the normal fan of the \defn{permutreehedron}~$\Permutreehedron$ defined equivalently as
\begin{enumerate}[(i)]
\item either the convex hull of the points~$\b{a}(\tree)$ for all $\decoration$-permutrees~$\tree$,
\item or the intersection of the hyperplane~$\HH$ with the half-spaces~$\HS(B)$ for all $\decoration$-building blocks~$B$.
\end{enumerate}
\end{theorem}

\begin{figure}[t]
  \centerline{\includegraphics[scale=.5]{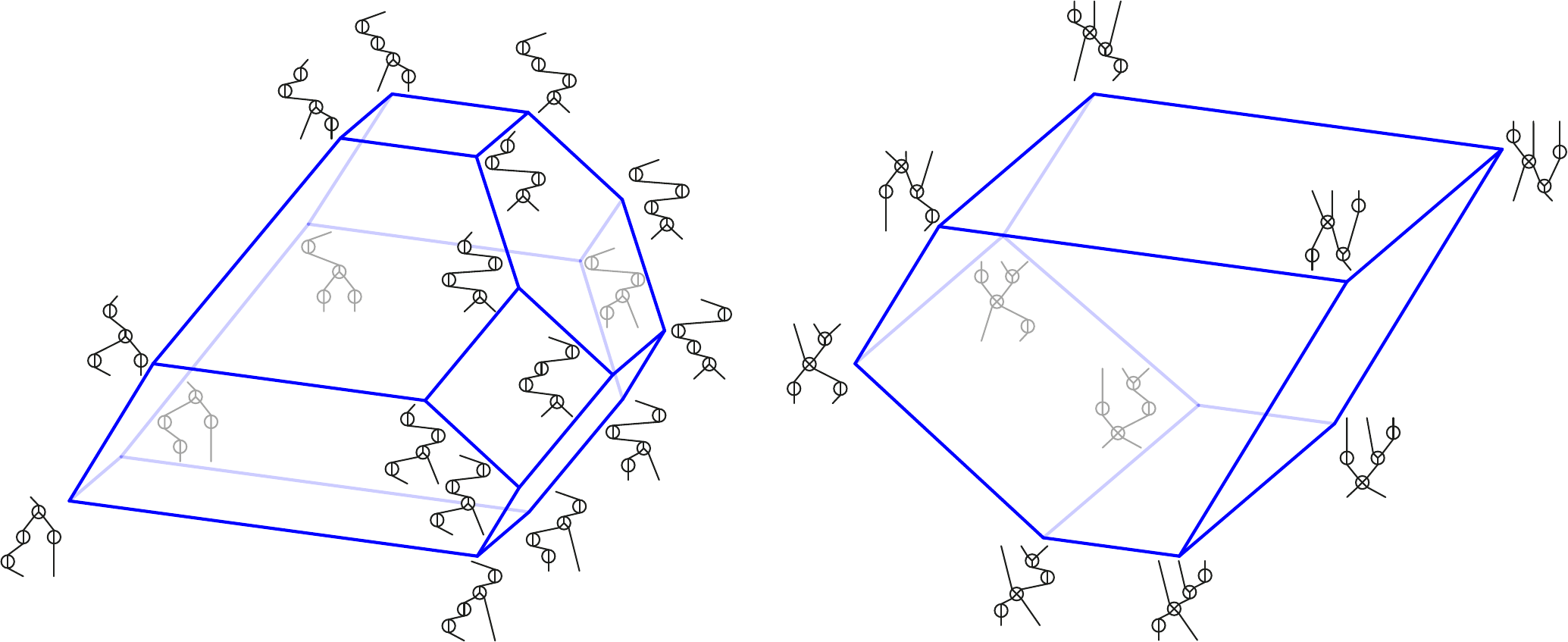}}
  \caption{The permutreehedra~$\Permutreehedron[\noneCirc{}\noneCirc{}\downCirc{}\noneCirc{}]$ (left) and~$\Permutreehedron[\noneCirc{}\upDownCirc{}\upCirc{}\noneCirc{}]$ (right).}
  \label{fig:permutreehedron}
\end{figure}

\fref{fig:permutreehedron} illustrates the permutreehedra~$\Permutreehedron[\noneCirc{}\noneCirc{}\downCirc{}\noneCirc{}]$ and~$\Permutreehedron[\noneCirc{}\upDownCirc{}\upCirc{}\noneCirc{}]$. For more examples, \fref{fig:permutreehedra} shows all $\decoration$-permutreehedra for~$\decoration \in \noneCirc{} \cdot \Decorations^2 \cdot \noneCirc{}$. We restrict to the cases when~${\decoration_1 = \decoration_4 = \noneCirc{}}$ as the first and last symbols of~$\decoration$ do not matter for~$\Permutreehedron$ (see Remark~\ref{rem:bijectionsPermutrees}).

Our proof of Theorem~\ref{theo:permutreehedra} is based on the following characterization of the valid right hand sides to realize a complete simplicial fan as the normal fan of a convex polytope. A proof of this statement can be found \eg in~\cite[Theorem~4.1]{HohlwegLangeThomas}.

\begin{theorem}[\protect{\cite[Theorem~4.1]{HohlwegLangeThomas}}]
\label{theo:HohlwegLangeThomas}
Given a complete simplicial fan~$\fan$ in~$\R^d$, consider for each ray~$\rho$ of~$\fan$ a half-space~$\HS_\rho$ of~$\R^d$ containing the origin and defined by a hyperplane~$\Hyp_\rho$ orthogonal to~$\rho$. For each maximal cone~$\Cone$ of~$\fan$, let~$\b{a}(\Cone) \in \R^d$ be the intersection of the hyperplanes~$\Hyp_\rho$ for~$\rho \in \Cone$. Then the following assertions are equivalent:
\begin{enumerate}[(i)]
\item The vector~$\b{a}(\Cone') - \b{a}(\Cone)$ points from~$\Cone$ to~$\Cone'$ for any two adjacent maximal cones~$\Cone$, $\Cone'$~of~$\fan$.
\item The polytopes
\[
\conv\set{\b{a}(\Cone)}{\Cone \text{ maximal cone of } \fan} \quad\text{ and }\quad
\bigcap_{\rho \text{ ray of } \fan} \HS_\rho
\]
coincide and their normal fan is~$\fan$.
\end{enumerate}
\end{theorem}

To apply this theorem, we start by checking that our point~$\b{a}(\tree)$ is indeed the intersection of the hyperplanes corresponding to the cone~$\Cone(\tree)$.

\begin{lemma}
\label{lem:intersectionPoint}
For any permutree~$\tree$, the point~$\b{a}(\tree)$ is the intersection point of the hyperplanes~$\Hyp(I)$ for all edge cuts~$\edgecut{I}{J}$ of~$\tree$.
\end{lemma}

\begin{proof}
Let~$\b{x}$ denote the intersection point of the hyperplanes~$\Hyp(I)$ for all edge cuts~$\edgecut{I}{J}$ of~$\tree$. Fix~$k \in [n]$. Each cut given by the incoming and outgoing edges of vertex~$k$ provides an equation of the form~$\sum_{i \in I} x_i = \binom{|I|+1}{2}$. Combining these equations (more precisely, adding the outgoing equations and substracting the incoming equations), we obtain the value of~$x_k$ in terms of the sizes~$d, \down{\ell}, \down{r}, \up{\ell}, \up{r}$ defined earlier. We distinguish four cases, depending on the decoration of~$i$:
\[
x_k = \left\{ \begin{array}{@{}l@{\;}l@{\quad\;}l}
\binom{d+2}{2} - \binom{d+1}{2} & = 1 + d & \text{if } \decoration_i = \noneCirc{}, \\
\binom{\down{\ell}+\down{r}+2}{2} - \binom{\down{\ell}+1}{2} - \binom{\down{r}+1}{2} & = 1 + d + \down{\ell}\down{r} & \text{if } \decoration_i = \downCirc{}, \\
\binom{d+\up{r}+2}{2} + \binom{d+\up{\ell}+2}{2} - \binom{d+1}{2} - \binom{d+\up{\ell}+\up{r}+1}{2} & = 1 + d - \up{\ell}\up{r} & \text{if } \decoration_i = \upCirc{}, \\
\binom{\down{\ell}+\down{r}+\up{r}+2}{2} + \binom{\down{\ell}+\down{r}+\up{\ell}+2}{2} - \binom{\down{\ell}+1}{2} - \binom{\down{r}+1}{2} - \binom{\down{\ell}+\down{r}+\up{\ell}+\up{r}+1}{2} & = 1 + d + \down{\ell}\down{r} - \up{\ell}\up{r} & \text{if } \decoration_i = \upDownCirc{}. \!\!\qedhere
\end{array} \right.
\]
\end{proof}

We can now check that Condition~(i) in Theorem~\ref{theo:HohlwegLangeThomas} holds. This requires a short case analysis of the rotation on permutrees (see Definition~\ref{def:rotation} and \fref{fig:rotation}).

\begin{lemma}
\label{lem:differenceIntersectionPoints}
Let~$\tree, \tree'$ be two permutrees connected by the rotation of the edge~$i \to j \in \tree$ to the edge~$i \leftarrow j \in \tree'$. Then the difference~$\b{a}(\tree') - \b{a}(\tree)$ is a positive multiple of~$\b{e}_i - \b{e}_j$.
\end{lemma}

\begin{proof}
Analyzing the 16 possible situations presented in \fref{fig:rotation}, we obtain that
\[
\b{a}(\tree') - \b{a}(\tree) = (\ell+1)(r+1)(\b{e}_i-\b{e}_j),
\]
where~$\ell$ is the sum of the sizes of the left subtrees of~$i$ (if any), and~$r$ is the sum of the sizes of the right subtrees of~$j$ (if any). The result follows since~$\ell \ge 0$ and~$r \ge 0$.
\end{proof}

\begin{proof}[Proof of Theorem~\ref{theo:permutreehedra}]
Direct application of Theorem~\ref{theo:HohlwegLangeThomas}, where~(i) holds by Lemmas~\ref{lem:intersectionPoint} and~\ref{lem:differenceIntersectionPoints}.
\end{proof}

\begin{example}
Following Example~\ref{exm:permutreesSpecificDecorations}, the $\decoration$-permutreehedron~$\Permutreehedron$ specializes to:
\begin{enumerate}[(i)]
\item the permutahedron~$\Perm$ when~$\decoration = \noneCirc{}^n$,
\item the associahedron~$\Asso$ of J.-L.~Loday~\cite{ShniderSternberg, Loday} when~$\decoration = \downCirc{}^n$,
\item the associahedra~$\Asso[\decoration]$ of C.~Hohlweg and C.~Lange~\cite{HohlwegLange, LangePilaud} when~$\decoration \in \{\downCirc{}, \upCirc{}\}^n$,
\item the parallelepiped~$\Para$ with directions~$\b{e}_i - \b{e}_{i+1}$ for each~$i \in [n-1]$ when~$\decoration = \upDownCirc{}^n$,
\item the graphical zonotope~$\Zono$ generated by the vectors~$\b{e}_i - \b{e}_{j}$ for each~$i < j \in [n-1]$ such that~$\decoration_{|(i,j)} = \noneCirc{}^{j-i-1}$ when~$\decoration \in \{\noneCirc{},\upDownCirc{}\}^n$.
\end{enumerate}
See \fref{fig:permutreehedra} for examples when~$n = 4$.
\end{example}


\subsection{Further geometric topics}
\label{subsec:furtherGeometricTopics}

We now explore several miniatures about permutreehedra. All are inspired from similar properties known for the associahedron, see \eg~\cite{LangePilaud} for a survey.


\subsubsection{Linear orientation and permutree lattice}

The $\decoration$-permutree lattice studied in Section~\ref{subsec:rotations} naturally appears in the geometry of the $\decoration$-permutreehedron~$\Permutreehedron$. Denote by $U$ the vector
\[
U \eqdef (n,n-1,\dots,2,1) - (1,2,\dots,n-1,n) = \sum_{i \in [n]} (n+1-2i) \, \b{e}_i.
\]

\begin{proposition}
When oriented in the direction~$U$, the $1$-skeleton of the $\decoration$-permutreehedron~$\Permutreehedron$ is the Hasse diagram of the $\decoration$-permutree lattice.
\end{proposition}

\begin{proof}
By Theorem~\ref{theo:permutreehedra}, the $1$-skeleton of~$\Permutreehedron$ is the rotation graph on $\decoration$-permutrees. It thus only remains to check that increasing rotations are oriented as~$U$. Consider two $\decoration$-permutrees~$\tree, \tree'$ connected by the rotation of the edge~$i \to j \in \tree$ to the edge~$i \leftarrow j \in \tree'$ such that~$i < j$. Then according to Lemma~\ref{lem:differenceIntersectionPoints} (and using the same notations), we have
\[
\dotprod{U}{\b{a}(\tree') - \b{a}(\tree)} = (\ell+1)(r+1) \dotprod{U}{\b{e}_i-\b{e}_j} = 2(\ell+1)(r+1)(j-i) > 0. \qedhere
\]
\end{proof}


\subsubsection{Matriochka permutreehedra}

We have seen in Proposition~\ref{prop:refinementPermutreeFan} that the $\decoration$-permutree fan~$\Fan[\decoration]$ refines the $\decoration'$-permutree fan~$\Fan[\decoration']$ when~$\decoration \less \decoration'$. It implies that all rays of~$\Fan[\decoration']$ are also rays of~$\Fan[\decoration]$, and thus that all inequalities of the $\decoration'$-permutreehedron~$\Permutreehedron[\decoration']$ are also inequalities of the $\decoration$-permutreehedron~$\Permutreehedron[\decoration]$.

\begin{corollary}
For any two decorations~$\decoration \less \decoration'$, we have the inclusion~$\Permutreehedron[\decoration] \subseteq \Permutreehedron[\decoration']$.
\end{corollary}

In other words, the poset of $(n-1)$-dimensional permutreehedra ordered by inclusion is isomorphic to the refinement poset on the decorations of~$\noneCirc{} \cdot \Decorations^{n-2} \cdot \noneCirc{}$ (remember that the decorations on the first and last vertices do not matter by Remark~\ref{rem:bijectionsPermutrees}). Chains along this poset provide \defn{Matriochka permutreehedra}. This poset is illustrated on \fref{fig:permutreehedra} when~$n=4$.

\hvFloat[floatPos=p, capWidth=h, capPos=r, capAngle=90, objectAngle=90, capVPos=c, objectPos=c]{figure}
{\begin{minipage}{25cm}\vspace*{-2.1cm}\centerline{\includegraphics[scale=.39]{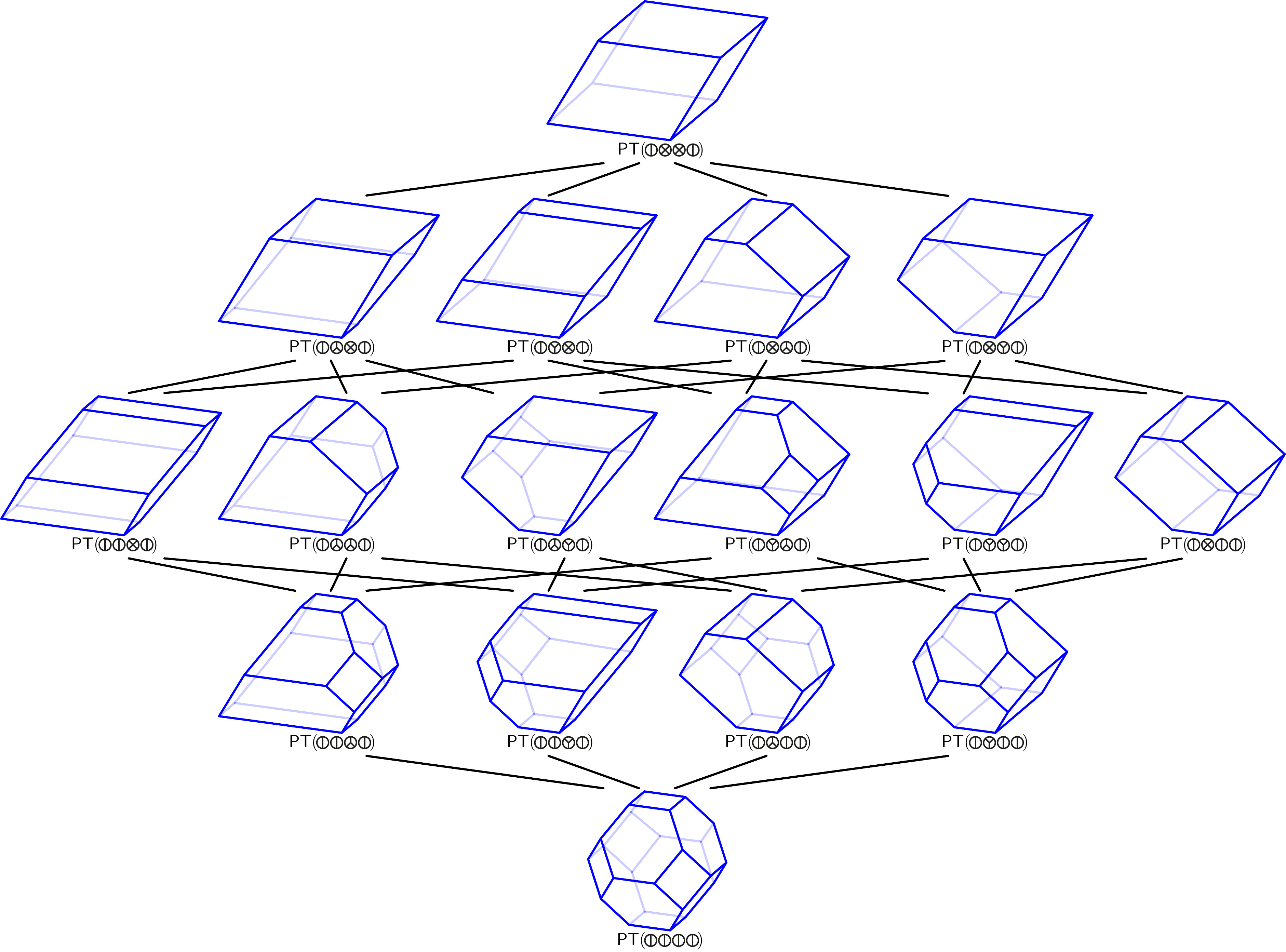}}\end{minipage}}
{The $\decoration$-permutreehedra, for all decorations~$\decoration \in \noneCirc{} \cdot \Decorations^2 \cdot \noneCirc{}$.}
{fig:permutreehedra}


\subsubsection{Parallel facets}

It is known that the permutahedron~$\Perm$ has $2^{n-1}-1$ pairs of parallel facets, while all associahedra~$\Asso[\decoration]$ of C.~Hohlweg and C.~Lange~\cite{HohlwegLange, LangePilaud} as well as the parallelepiped~$\Para$ have~$n-1$ pairs of parallel facets. This property extends to all permutreehedra as follows.

\begin{proposition}
Consider a decoration~$\decoration \in \Decorations^n$, assume without loss of generality that~$\decoration_1 \ne \noneCirc{}$ and~$\decoration_n \ne \noneCirc{}$, let~$1 = u_0, u_1, \dots, u_{v-1}, u_v = n$ be the positions in~$\decoration$ such that~${\decoration_{u_i} \ne \noneCirc{}}$, and let~$n_i = u_i-u_{i-1}-1$ be the sizes of the (possibly empty) blocks of~$\noneCirc{}$ in~$\decoration$. Then the permutreehedron~$\Permutreehedron$ has
\[
\sum_{i \in [v]} (2^{n_i+1}-1)
\]
pairs of parallel facets whose normal vectors are the characteristic vectors of
\begin{itemize}
\item the sets~$[u_{i-1}] \cup X$ for~$i \in [v]$ and~$X \subseteq (u_{i-1}, u_i)$, and
\item the sets~$X$ for~$i \in [v]$ and~$\varnothing \ne X \subsetneq (u_{i-1}, u_i)$.
\end{itemize}
\end{proposition}

\begin{proof}
Pairs of parallel facets correspond to pairs of complementary building blocks. As already mentioned, it is easier to think of building blocks as arcs in~$\b{P}_\decoration$: the building block corresponding to an arc~$\alpha$ is the set~$I_\alpha$ of indices of all points~$\b{p}_k$ below~$\alpha$, including the endpoints of~$\alpha$ in the upper convex hull of~$\b{P}_\decoration$ (and excluding~$0$ and~$n+1$). Let~$\b{p}_i, \b{p}_j$ with~$i < j$ denote the endpoints of~$\alpha$, and define~$J_\alpha \eqdef (i,j) \cap \decoration^{-1}(\downCirc)$ and~$K_\alpha \eqdef (i,j) \ssm \decoration^{-1}(\upCirc)$. Observe that
\begin{itemize}
\item if~$\b{p}_0$ and~$\b{p}_{n+1}$ are both above~$\alpha$, then~$J_\alpha \subseteq I_\alpha \subseteq K_\alpha$,
\item if~$\b{p}_0$ is below~$\alpha$ while~$\b{p}_{n+1}$ is above~$\alpha$, then~$[1,i] \cup J_\alpha \subseteq I_\alpha \subseteq [1,i] \cup K_\alpha$,
\item if~$\b{p}_0$ is above~$\alpha$ while~$\b{p}_{n+1}$ is below~$\alpha$, then~$[j,n] \cup J_\alpha \subseteq I_\alpha \subseteq [j,n] \cup K_\alpha$,
\item if~$\b{p}_0$ and~$\b{p}_{n+1}$ are both below~$\alpha$, then~$[1,i] \cup [j,n] \cup J_\alpha \subseteq I_\alpha \subseteq [1,i] \cup [j,n] \cup K_\alpha$.
\end{itemize}
We conclude that if~$I_\alpha$ and~$I_\beta$ are complementary, then
\begin{itemize}
\item either there is~$i \in [v]$ such that~$I_\alpha = [1,u_{i-1}] \cup X$ and~$I_\beta = [u_i+1,n] \cup (u_{i-1},u_i) \ssm X$ for some~$X \subseteq (u_{i-1},u_i)$ (or the opposite),
\item or there is~$i \in [v]$ such that~$I_\alpha = X$ and~$I_\beta = [n] \ssm X$ for some~$\varnothing \ne X \subsetneq (u_{i-1}, u_i)$ (or the opposite).
\end{itemize}
The enumerative formula is then immediate.
\end{proof}


\subsubsection{Common vertices}

It is combinatorially relevant to characterize which vertices are common to two nested permutreehedra. For example, note that~$[1, 2, \dots, n-1, n]$ and~$[n, n-1, \dots, 2, 1]$ are common vertices of all $(n-1)$-dimensional permutreehedra. The other common vertices are characterized in the following statement.

\begin{proposition}
\label{prop:singletons}
Consider two decorations~$\decoration \less \decoration'$, a $\decoration$-permutree~$\tree$ and a $\decoration'$-permutree~$\tree'$. The following assertions are equivalent:
\begin{enumerate}[(i)]
\item the vertex~$\b{a}(\tree)$ of~$\Permutreehedron[\decoration]$ coincides with the vertex~$\b{a}(\tree')$ of~$\Permutreehedron[\decoration']$,
\item the cone~$\Cone(\tree)$ of~$\Permutreehedron[\decoration]$ coincides with the cone~$\Cone(\tree')$ of~$\Permutreehedron[\decoration']$,
\item the normal cone~$\Cone\polar(\tree)$ of~$\Permutreehedron[\decoration]$ coincides with the normal cone~$\Cone\polar(\tree')$ of~$\Permutreehedron[\decoration']$,
\item the fiber of~$\tree'$ under the surjection~$\surjection$ is the singleton~$(\surjection)^{-1}(\tree') = \{\tree\}$,
\item $\tree$ and~$\tree'$ have precisely the same linear extensions,
\item $\tree$ and~$\tree'$ coincide up to some empty descendant or ancestors,
\item for each~$i \in [n]$ with~$\decoration_i \in \{\noneCirc{}, \upCirc{}\}$ but~$\decoration'_i \in \{\downCirc{}, \upDownCirc{}\}$ (resp.~with~$\decoration_i \in \{\noneCirc{}, \downCirc{}\}$ but~$\decoration'_i \in \{\upCirc{}, \upDownCirc{}\}$), the vertex~$i$ of~$\tree$ is not above (resp.~below) an edge of~$\tree$, and $\tree'$ is obtained from~$\tree$ by adding one empty descendant (resp.~ancestor) subtree at each such vertex~$i$,
\item for each~$i \in [n]$ with~$\decoration_i \in \{\noneCirc{}, \upCirc{}\}$ but~$\decoration'_i \in \{\downCirc{}, \upDownCirc{}\}$ (resp.~with~$\decoration_i \in \{\noneCirc{}, \downCirc{}\}$ but~$\decoration'_i \in \{\upCirc{}, \upDownCirc{}\}$), the vertex~$i$ of~$\tree'$ has at least one empty descendant (resp.~ancestor), and~$\tree$ is obtained from~$\tree'$ by deleting one empty descendant (resp.~ancestor) subtree at each such vertex~$i$.
\end{enumerate}
\end{proposition}

\begin{proof}
Since~$\Permutreehedron[\decoration']$ is obtained from~$\Permutreehedron[\decoration]$ by deleting facet inequalities, and since~$\Permutreehedron[\decoration]$ is simple, a vertex~$\b{a}(\tree)$ is common to~$\Permutreehedron[\decoration]$ and~$\Permutreehedron[\decoration']$ if and only if the facets containing~$\b{a}(\tree)$ are common to~$\Permutreehedron[\decoration]$ and~$\Permutreehedron[\decoration']$. This proves (i)$\iff$(ii). The equivalence (ii)$\iff$(iii) is immediate by polarity. The equivalence (iii)$\iff$(iv) follows by Proposition~\ref{prop:refinementPermutreeFan}. The equivalence (iii)$\iff$(v) holds since~$\Cone\polar(\tree)$ is the union of the cones~$\Cone\polar(\tau)$ over all linear extensions~$\tau$ of~$\tree$. The equivalence (ii)$\iff$(vi) holds since~$\Cone(\tree)$ has a facet for each internal edge of~$\tree$. Finally, the equivalence (vi)$\iff$(vii)$\iff$(viii) directly follow from the description of~$\surjection$ in Section~\ref{subsec:decorationRefinements}.
\end{proof}

In particular, since~$\noneCirc{}^n \less \decoration$ for any decoration~$\decoration$, Proposition~\ref{prop:singletons} characterizes the common points of the permutahedron~$\Perm$ with any permutreehedron~$\Permutreehedron$. Call \defn{$\decoration$-singleton} a permutation~$\tau$ corresponding to such a common vertex, that is, such that~$\PSymbol^{-1}(\PSymbol(\tau)) = \{\tau\}$. In the next section, we will need the following consequences of Proposition~\ref{prop:singletons}.

\begin{lemma}
\label{lem:singleton}
Let~$\tau \in \fS_n$ be a $\decoration$-singleton and let~$i,j \in [n]$ be such that~$\tau(i) = j$. Then
\begin{itemize}
\item if~$\decoration_j \in \{\downCirc{}, \upDownCirc{}\}$, then either~$\tau([1,i-1]) \subseteq [1,j-1]$ or~$\tau([1,i-1]) \subseteq [j+1,n]$;
\item if~$\decoration_j \in \{\upCirc{}, \upDownCirc{}\}$, then either~$\tau([i+1,n]) \subseteq [1,j-1]$ or~$\tau([i+1,n]) \subseteq [j+1,n]$.
\end{itemize}
\end{lemma}

\begin{proof}
Let~$\tree$ be the $\noneCirc{}^n$-permutree obtained by insertion of~$\tau$. By Proposition~\ref{prop:singletons}, if~$\decoration_j \in \{\downCirc{}, \upDownCirc{}\}$, then the vertex~$j$ of~$\tree$ is not above an edge of~$\tree$. Equivalently, all vertices of~$\tree$ below~$j$ are either to the left or to the right of~$j$. In other words, if~$i = \tau^{-1}(j)$, we have either~$\tau([1,i-1]) \subseteq [1,j-1]$ or~$\tau([1,i-1]) \subseteq [j+1,n]$. The second part of the statement is symmetric.
\end{proof}

\begin{lemma}
\label{lem:oppositeSingletons}
Consider a decoration~$\decoration \in \noneCirc{} \cdot \Decorations^{n-2} \cdot \noneCirc{}$, and denote by~$I_1, \dots, I_p$ the blocks of consecutive decorations~\noneCirc{} in~$\decoration$. Then a permutation~$\tau$ and its opposite~$\bar\tau = \tau \cdot [n, \dots, 1]$ are both $\decoration$-singletons if and only if there exists~$(\pi_1, \dots, \pi_p) \in \fS_{I_1} \times \dots \times \fS_{I_p}$ such that~$\tau = \pi_1 \cdot \ldots \cdot \pi_p$ or~$\bar\tau = \pi_1 \cdot \ldots \cdot \pi_p$.
\end{lemma}

\begin{proof}
Let~$i,j \in [n]$ be such that~$\tau(i) = j$ and~$\decoration_j \in \{\downCirc{}, \upDownCirc{}\}$. Since~$\tau$ is a $\decoration$-singleton, $\tau([1,i-1])$ is a subset of either~$[1,j-1]$ or~$[j+1,n]$ by Lemma~\ref{lem:singleton}. Since~$\bar\tau$ is a $\decoration$-singleton and~${\bar\tau(n+1-i) = j}$, $\tau([i+1,n]) = \bar\tau([1,n-i]) $ is a subset of either~$[1,j-1]$ or~$[j+1,n]$ by Lemma~\ref{lem:singleton}. If~$\tau([1,i-1])$ and~$\tau([i+1,n])$ are both subsets of~$[1,j-1]$ (resp.~of~$[j+1,n]$), then~$j=n$ (resp.~$j=1$) which contradicts our assumption that~$\decoration_n = \noneCirc{}$ (resp.~that~$\decoration_1 = \noneCirc{}$). By symmetry, we conclude that for all~$i,j \in [n]$ with~$\tau(i) = j$ and~$\decoration_j \ne \noneCirc{}$, we have either~$\tau([1,i-1]) \subseteq [1,j-1]$ and~${\tau([i+1,n]) = [j+1,n]}$, or~$\tau([1,i-1]) \subseteq [j+1,n]$ and~$\tau([i+1,n]) = [1,j-1]$. The result immediately follows.
\end{proof}


\subsubsection{Isometrees}

We now consider isometries of the permutreehedron~$\Permutreehedron$ and between distinct permutreehedra~$\Permutreehedron[\decoration]$ and~$\Permutreehedron[\decoration']$. As already observed, $\Permutreehedron$ is independent of the first and last decorations~$\decoration_1$ and~$\decoration_n$. In this section, we assume without loss of generality that~${\decoration_1 = \decoration_n = \noneCirc{}}$.

For a permutation~$\tau \in \fS_n$, we denote by~$\rho_\tau: \R^n \to \R^n$ the isometry of~$\R^n$ given by permutation of the coordinates~$\rho_\tau(x_1, \dots, x_n) = (x_{\tau(1)}, \dots, x_{\tau(n)})$. We write~$\rho_i = \rho_{(i \; i+1)}$ for the exchange of the $i$th and~$(i+1)$th coordinates. With these notations, we have the following observation.

\begin{proposition}
\label{prop:isometrees1}
If~$\decoration_i = \decoration_{i+1} = \noneCirc{}$, then~$\rho_i$ is an isometry of the permutreehedron~$\Permutreehedron$. Thus, if~$I_1, \dots, I_p$ are the blocks of consecutive~\noneCirc{} in~$\decoration$, then for any~${\Pi = (\pi_1, \dots, \pi_v) \in \fS_{I_1} \times \dots \times \fS_{I_p}}$, the map~$\rho_\Pi \eqdef \rho_{\pi_1} \circ \dots \circ \rho_{\pi_p}$ is an isometry of the permutreehedron~$\Permutreehedron$.
\end{proposition}

\begin{proof}
Assume that~$\decoration_i = \decoration_{i+1} = \noneCirc{}$. Then there are no local condition around vertices~$i$ and~$i+1$ in the definition of $\decoration$-permutrees. Therefore, exchanging the labels~$i$ and~$i+1$ in any $\decoration$-permutree~$\tree$ results in a new $\decoration$-permutree~$\tree'$. Moreover, the vertices associated to~$\tree$ and~$\tree'$ are related by~$\b{a}(\tree') = \rho_i(\b{a}(\tree))$. We conclude that~$\rho_i$ is indeed an isometry of the permutreehedron~$\Permutreehedron$.
\end{proof}

Remember now the two symmetrees discussed in Remark~\ref{rem:symmetrees}. Denote by~$\horizontalSymmetry{\tree}$ (resp.~$\verticalSymmetry{\tree}$) the permutree obtained from~$\tree$ by a horizontal (resp.~vertical) symmetry, and denote by~$\horizontalSymmetry{\decoration}$ (resp.~$\verticalSymmetry{\decoration}$) the decoration obtained from~$\decoration$ by a mirror image (resp.~by interverting~\downCirc{} and~\upCirc{} decorations), so that~$\decoration(\horizontalSymmetry{\tree}) = \horizontalSymmetry{\decoration(\tree)}$ (resp.~$\decoration(\verticalSymmetry{\tree}) = \verticalSymmetry{\decoration(\tree)}$). Finally, we write~$\horizontalVerticalSymmetry{\decoration} = \horizontalSymmetry{(\verticalSymmetry{\decoration})} = \verticalSymmetry{(\horizontalSymmetry{\decoration})}$.

\begin{proposition}
\label{prop:isometrees2}
For any permutree~$\tree$, we have
\[
\b{a}(\horizontalSymmetry{\tree})_i = \b{a}(\tree)_{n+1-i}
\qquad\text{and}\qquad
\b{a}(\verticalSymmetry{\tree})_i = n+1-\b{a}(\tree)_i.
\]
Consequently, the permutreehedra~$\Permutreehedron[\decoration]$, $\Permutreehedron[\horizontalSymmetry{\decoration}]$, $\Permutreehedron[\verticalSymmetry{\decoration}]$, and~$\Permutreehedron[\horizontalVerticalSymmetry{\decoration}]$ are all isometric.
\end{proposition}

\begin{proof}
The formulas immediately follow by case analysis from the definition of~$\b{a}(\tree)_i$.
\end{proof}

We denote by~$\chi : (x_1, \dots, x_n) \mapsto (x_n, \dots, x_1)$ and~$\theta : (x_1, \dots, x_n) \mapsto (n+1-x_1, \dots, n+1-x_n)$ the two maps of Proposition~\ref{prop:isometrees2}. To follow Remark~\ref{rem:symmetrees}, we call them \defn{isometrees}. The main result of this section claims that the isometries of Propositions~\ref{prop:isometrees1} and~\ref{prop:isometrees2} are essentially the only isometries preserving permutreehedra.

\begin{proposition}
\label{prop:isometrees3}
Let~$\decoration, \decoration' \in \noneCirc{} \cdot \Decorations^{n-2} \cdot \noneCirc{}$ and denote by~$I_1, \dots, I_p$ and~$I'_1, \dots, I'_{q}$ the blocks of consecutive decorations~\noneCirc{} in~$\decoration$ and~$\decoration'$ respectively. If~$\rho$ is an isometry of~$\R^n$ which sends the permutreehedron~$\Permutreehedron[\decoration]$ to the permutreehedron~$\Permutreehedron[\decoration']$, then~$\decoration' \in \{\decoration, \horizontalSymmetry{\decoration}, \verticalSymmetry{\decoration}, \horizontalVerticalSymmetry{\decoration}\}$ and there are~$\Pi \in \fS_{I_1} \times \dots \times \fS_{I_p}$ and~$\Pi' \in \fS_{I'_1} \times \dots \times \fS_{I'_{q}}$ such that~$\rho_{\Pi'}^{-1} \circ \rho \circ \rho_{\Pi}^{-1} \in \{\Id, \chi, \theta, \chi \circ \theta\}$.
\end{proposition}

\begin{proof}
By Proposition~\ref{prop:singletons}, the $\decoration$-singletons correspond to the cones of the permutree fan~$\Fan$ that are single braid cones. Since the isometry~$\rho$ sends the permutreehedron~$\Permutreehedron[\decoration]$ to the permutreehedron~$\Permutreehedron[\decoration']$, it sends the permutree fan~$\Fan[\decoration]$ to the permutree fan~$\Fan[\decoration']$, and thus the $\decoration$-singletons to the $\decoration'$-singletons. Therefore, $\rho$ sends any pair~$(\tau, \bar\tau)$ of opposite $\decoration$-singletons to a pair~$(\tau', \bar\tau')$ of opposite $\decoration'$-singletons. By Lemma~\ref{lem:oppositeSingletons}, there is~$\Pi \eqdef (\pi_1, \dots, \pi_p) \in \fS_{I_1} \times \dots \times \fS_{I_p}$ such that~${\tau = \pi_1 \cdot \ldots \cdot \pi_p}$ or~$\bar\tau = \pi_1 \cdot \ldots \cdot \pi_p$, and there is~$\Pi' \eqdef (\pi'_1, \dots, \pi'_{q}) \in \fS_{I'_1} \times \dots \times \fS_{I'_{q}}$ such that~$\tau' = \pi'_1 \cdot \ldots \cdot \pi'_{q}$ or~$\bar\tau' = \pi'_1 \cdot \ldots \cdot \pi'_{q}$. Therefore, the isometry~$\Gamma \eqdef \rho_{\Pi'}^{-1} \circ \rho \circ \rho_{\Pi}^{-1}$ stabilizes the pair of opposite singletons~$\{[1, \dots, n], [n, \dots, 1]\}$, and thus the center~$\b{O} \eqdef \frac{n+1}{2} \one$ of the permutahedron~$\Perm$. For any subset~$\varnothing \ne U \subseteq [n]$ of cardinality~$u \eqdef |U|$, the distance from~$\b{O}$ to the hyperplane~$\Hyp(U)$~is
\[
d \big( \b{O}, \Hyp(U) \big) = \frac{n \, u\, (n - u)}{\sqrt{u^2 + (n - u)^2}}.
\]
Observe that the function
\[
x \longmapsto \frac{x \, (1-x)}{\sqrt{x^2 + (1-x)^2}}
\]
is bijective on~$\big[ 0, \frac{1}{2} \big]$. It follows that the isometry~$\Gamma$ sends the hyperplane~$\Hyp(U)$, for~${\varnothing \ne U \subseteq \ground}$, to an hyperplane~$\Hyp(U')$ for some~$\varnothing \ne U' \subseteq [n]$ with~$|U'| = |U|$ or~$|U'| = n - |U|$. Observe moreover that the facets of~$\Perm$ defined by the hyperplanes~$\Hyp(\{i\})$ for~$i \in [n]$ are pairwise non-adjacent, but that the facet defined by~$\Hyp(\{i\})$ is adjacent to all facets defined by~${\Hyp(\ground \ssm \{j\})}$ for~$j \in [n] \ssm \{i\}$. Therefore, the hyperplanes~$\Hyp(\{i\})$ for~$i \in [n]$ are either all sent to the hyperplanes~$\Hyp(\{j\})$ for~$j \in [n]$, or all sent to the hyperplanes~$\Hyp([n] \ssm \{j\})$ for~$j \in [n]$.

Assume first that we are in the former situation. Define a map~$\gamma : [n] \to [n]$ such that the hyperplane~$\Hyp(\{i\})$ is sent by~$\Gamma$ to the hyperplane~$\Gamma \big( \Hyp(\{i\}) \big) = \Hyp(\{\gamma(i)\})$ for any~$i \in [n]$. It follows that~$\Gamma(\b{e}_i) = \b{e}_{\gamma(i)}$. Since~$\Gamma$ is a linear map, it sends the characteristic vector of any subset~$\varnothing \ne U \subseteq [n]$ to the characteristic vector of~$\gamma(U)$. Thus, the map~$\gamma$ defines an isomorphism from the $\decoration$-building blocks to the $\decoration'$-building blocks. However, up to an horizontal reflection, the $\decoration$-building blocks determine the decoration~$\decoration$. Indeed, we immediately derive from the description of building blocks in terms of arcs in the point set~$\b{P}_\decoration$ that
\begin{itemize}
\item $\decoration_j \in \{\noneCirc{}, \downCirc{}\}$ if and only if~$\{j\}$ is a $\decoration$-building block,
\item $\decoration_j \in \{\noneCirc{}, \upCirc{}\}$ if and only if~$[n] \ssm \{j\}$ is a $\decoration$-building block, and
\item two labels~$i,j \in [n]$ are only separated by~\noneCirc{} in~$\decoration$ if and only if they belong to complementary $\decoration$-building blocks.
\end{itemize}
We conclude that~$\Gamma \in \{\Id, \chi\}$ in this first situation.

Finally, if we were in the latter situation above, then we just apply a vertical reflection to the decoration~$\decoration$. By Proposition~\ref{prop:isometrees2}, it composes~$\Gamma$ by a symmetry~$\theta$ and thus places us back in the situation treated above. We conclude that~$\Gamma \in \{\theta, \chi \circ \theta\}$ in this second situation.
\end{proof}

\begin{corollary}
Consider a decoration~$\decoration \in \noneCirc{} \cdot \Decorations^{n-2} \cdot \noneCirc{}$, and let~$n_1, \dots, n_p$ denote the sizes of the blocks of consecutive~\noneCirc{} in~$\decoration$. Then the isometry group of the permutreehedron~$\Permutreehedron$ has cardinality~$n_1! \cdots n_p! \, (1 + \one_{\decoration = \horizontalSymmetry{\decoration}} - \one_{\decoration = \noneCirc{}^n}) \, (1 + \one_{\decoration \in \{\verticalSymmetry{\decoration}, \horizontalVerticalSymmetry{\decoration}\}})$.
\end{corollary}

We finally consider the number~$x(n)$ of isometry classes of $n$-dimensional permutreehedra. We have~$x(0) = 1$ (a point), $x(1) = 1$ (a segment), $x(2) = 3$ (an hexagon, a pentagon, a quadrilateral), and \fref{fig:permutreehedra} shows that~$x(3) =  7$. In general, $x(n)$ is given by the following statement.

\begin{corollary}
\enlargethispage{-.6cm}
The number~$x(n)$ isometry classes of $n$-dimensional permutreehedra is the number of orbits of decorations~$\decoration \in \noneCirc{} \cdot \Decorations^{n-1} \cdot \noneCirc{}$ under the symmetrees~$\decoration \mapsto \horizontalSymmetry{\decoration}$ and~$\decoration \mapsto \verticalSymmetry{\decoration}$. It is given by the formula
\[
x(n) = 2^{n-4} \big( 2^n + (-1)^n + 7 \big)
\]
for~$n \ge 1$. Its generating function is given by~$\sum_{n \in \N} x(n)t^n = (1-3t-5t^2+7t^3)/(1-4t-4t^2+16t^3)$. See~\href{https://oeis.org/A225826}{\cite[A225826]{OEIS}}.
\end{corollary}

\begin{proof}
The first sentence is a direct consequence of Proposition~\ref{prop:isometrees3}. For the second part, note that~$x(n) = 2^{n-4} \big( 2^n + (-1)^n + 7 \big)$ satisfies the recursive formulas:
\begin{equation}
\tag{$\star$}
\label{eq:cor-symmetrees-rec}
x(2k) = 4 \cdot x(2k-1) - 4^{k-1}
\qquad\text{and}\qquad
x(2k+1) = 16 \cdot x(2k-1) - 9 \cdot 4^{k-1},
\end{equation}
for~$k \ge 1$ with initial values~$x(0) = x(1) = 1$. There is left to prove that the number of orbits satisfies these recurrences as well.
Let~$X_n$ be the set of orbits of decorations ${\decoration \in \noneCirc{} \cdot \Decorations^{n-1} \cdot \noneCirc{}}$ (of size $n+1$ such that the corresponding permutreehedron is of dimension~$n$) under the symmetrees~${\decoration \mapsto \horizontalSymmetry{\decoration}}$ and~${\decoration \mapsto \verticalSymmetry{\decoration}}$. We write $X_n = A_n \sqcup B_n \sqcup C_n \sqcup D_n \sqcup E_n$ as a disjoint union of~$5$ subsets according to the symetrees of the orbit:
\begin{itemize}
\item $A_n$ contains the orbits of the form~$\{\decoration\}$ where~$\delta = \horizontalSymmetry{\decoration} = \verticalSymmetry{\decoration} = \horizontalVerticalSymmetry{\decoration}$ (like $\noneCirc{} \upDownCirc{} \noneCirc{} \upDownCirc{} \noneCirc{}$),
\item $B_n$ contains the orbits of the form~$\lbrace \decoration, \horizontalSymmetry{\decoration} \rbrace$ where $\decoration = \verticalSymmetry{\decoration} \ne \horizontalSymmetry{\decoration} = \horizontalVerticalSymmetry{\decoration}$ (like $\noneCirc{} \upDownCirc{} \noneCirc{} \noneCirc{} \noneCirc{}$),
\item $C_n$ contains the orbits of the form~$\lbrace \decoration, \verticalSymmetry{\decoration} \rbrace$ where $\decoration = \horizontalSymmetry{\decoration} \ne \verticalSymmetry{\decoration} = \horizontalVerticalSymmetry{\decoration}$ (like $\noneCirc{} \downCirc{} \upCirc{} \downCirc{} \noneCirc{}$),
\item $D_n$ contains the orbits of the form~$\lbrace \decoration, \verticalSymmetry{\decoration} = \horizontalSymmetry{\decoration} \rbrace$ where $\decoration= \horizontalVerticalSymmetry{\decoration} \ne \verticalSymmetry{\decoration} = \horizontalSymmetry{\decoration}$ (like $\noneCirc{} \downCirc{} \upDownCirc{} \upCirc{} \noneCirc{}$),
\item $E_n$ contains the orbits of the form~$\{\delta, \horizontalSymmetry{\decoration}, \verticalSymmetry{\decoration}, \horizontalVerticalSymmetry{\decoration}\}$, where all symmetrees of~$\decoration$ are distinct (like $\noneCirc{} \downCirc{} \upDownCirc{} \noneCirc{} \noneCirc{}$).
\end{itemize}
We write the corresponding numbers with lower case letters such that 
\[
x(n) = a(n) + b(n) + c(n) + d(n) + e(n).
\]

Now we can obtain recursive formulas for these number by a simple case by case combinatorial generation: orbits of decorations of odd sizes $2k+1$ (even dimension $2k$) are obtained by adding a letter of $\Decorations$ in the middle of a decoration of size $2k$ and orbits of decorations of even sizes $2k+2$ (odd dimension $2k+1$) are obtained by adding a word of $\Decorations^2$ in the middle of a decoration of size $2k$. Depending on the symmetree subset of the original decoration of size $2k$, this process will have redundancies (adding two different letters can lead to one single orbit) and will lead to the following recursive formulas:
\begin{align*}
a(2k) &= 2 \cdot a(2k-1) &
a(2k+1) &= 2 \cdot a(2k-1) \\
b(2k) &= 2 \cdot b(2k-1) &
b(2k+1) &= 4 \cdot b(2k-1) + a(2k-1) \\
c(2k) &= a(2k-1) + 4 \cdot c(2k-1) &
c(2k+1) &= 4 \cdot c(2k-1) + a(2k-1) \\
d(2k) &= 2 \cdot d(2k-1) &
d(2k+1) &= 4 \cdot d(2k-1) + a(2k-1) \\
e(2k) &= b(2k-1) + d(2k-1) + 4 \cdot e(2k-1) &
e(2k+1) &= 2 \cdot a(2k-1) + 6 \cdot b(2k-1) + 6 \cdot c(2k-1)\\
& & & + 6 \cdot d(2k-1) + 16 \cdot e(2k-1)
\end{align*}
for~$k \ge 1$ with initial values~$a(1) = 1$ and~$b(1) = c(1) = d(1) = e(1) = 0$. In particular, we obtain
\[
a(2k+1) = 2^k
\qquad\text{and}\qquad
b(2k+1) = c(2k+1) = d(2k+1) = 2 \cdot 4^{k-1} - 2^{k-1}.
\]
By a basic computation, one checks that the number~$x(n)$ of orbits in~$X_n$ satisfies~\eqref{eq:cor-symmetrees-rec}.

For completeness, note that~$x(n)$ can be expressed as
\[
x(n) = 4 \cdot x(n-1) + 4 \cdot x(n-2) - 16 \cdot x(n-3)
\]
for $n \geq 4$ with initial values~$x(0) = 1$, $x(1) = 1$, $x(2) = 3$ and $x(3) = 7$.
Therefore, its generating function is given by
\[
\sum_{n \in \N} x(n)t^n = \frac{1-3t-5t^2+7t^3}{1-4t-4t^2+16t^3}. \qedhere
\]
\end{proof}

\subsubsection{Permutrees versus signed tree associahedra}

To conclude this section on permutreehedra, we want to mention that these polytopes were already constructed very implicitly. Indeed, any permutreehdron is a product of certain faces of the signed tree associahedra studied in~\cite{Pilaud-signedTreeAssociahedra}. More specifically, for any decoration in~$\{\noneCirc{},\downCirc{}\}^n$, the permutreehedron is a graph associahedron as constructed in~\cite{CarrDevadoss, Postnikov, Zelevinsky}. We believe however that the present construction is much more explicit and reflects relevant properties of these polytopes, which are not necessarily apparent in the general construction of~\cite{Pilaud-signedTreeAssociahedra} (in particular the combinatorial and algebraic properties of the permutrees which do not hold in general for arbitrary signed tree associahedra).


\section{The permutree Hopf algebra}
\label{sec:permutreeHopfAlgebra}

This section is devoted to algebraic aspects of permutrees. More precisely, using the same idea as G.~Chatel and V.~Pilaud in~\cite{ChatelPilaud} we construct a Hopf algebra on permutrees as a subalgebra of a decorated version of C.~Malvenuto and C.~Reutenauer's algebra. In turn, our algebra contains subalgebras isomorphic to C.~Malevenuto and C.~Reutenauer's Hopf algebra on permutations~\cite{MalvenutoReutenauer}, J.-L.~Loday and M.~Ronco's Hopf algebra on binary trees~\cite{LodayRonco}, G.~Chatel and V.~Pilaud Hopf algebra on Cambrian trees~\cite{ChatelPilaud}, and I.~Gelfand, D.~Krob, A.~Lascoux, B.~Leclerc, V.~S. Retakh, and J.-Y.~Thibon's algebra on binary sequences~\cite{GelfandKrobLascouxLeclercRetakhThibon}. In other words, we obtain an algebraic structure in which it is natural to multiply permutations with binary trees, or Cambrian trees with binary sequences. To keep our paper short, we omit the proofs of most statements as they are straightforward and similar to that of~\cite[Section~1.2]{ChatelPilaud}.


\subsection{The Hopf algebra on (decorated) permutations}
\label{subsec:decoratedMRAlgebra}

We briefly recall here the definition and some elementary properties of a decorated version of C.~Malvenuto and C.~Reutenauer's Hopf algebra on permutations~\cite{MalvenutoReutenauer}. For~$n,n' \in \N$, let
\[
\fS^{(n,n')} \eqdef \set{\tau \in \fS_{n+n'}}{\tau_1 < \dots < \tau_n \text{ and } \tau_{n+1} < \dots < \tau_{n+n'}}
\]
denote the set of permutations of~$\fS_{n+n'}$ with at most one descent, at position~$n$. 
The \defn{shifted concatenation}~$\tau\bar\tau'$, the \defn{shifted shuffle}~$\tau \shiftedShuffle \tau'$, and the \defn{convolution}~$\tau \convolution \tau'$ of two permutations~$\tau \in \fS_n$ and~$\tau' \in \fS_{n'}$ are classically defined by
\begin{gather*}
\tau\bar\tau' \eqdef [\tau_1, \dots, \tau_n, \tau'_1 + n, \dots, \tau'_{n'} + n] \in \fS_{n+n'}, \\
\tau \shiftedShuffle \tau' \eqdef \bigset{(\tau\bar\tau') \circ \pi^{-1}}{\pi \in \fS^{(n,n')}} 
\qquad\text{and}\qquad
\tau \convolution \tau' \eqdef \bigset{\pi \circ (\tau\bar\tau')}{\pi \in \fS^{(n,n')}}.
\end{gather*}
For example,
\begin{align*}
{\red 12} \shiftedShuffle {\blue 231} & = \{ {\red 12}{\blue 453}, {\red 1}{\blue 4}{\red 2}{\blue 53}, {\red 1}{\blue 45}{\red 2}{\blue 3}, {\red 1}{\blue 453}{\red 2}, {\blue 4}{\red 12}{\blue 53}, {\blue 4}{\red 1}{\blue 5}{\red 2}{\blue 3}, {\blue 4}{\red 1}{\blue 53}{\red 2}, {\blue 45}{\red 12}{\blue 3}, {\blue 45}{\red 1}{\blue 3}{\red 2}, {\blue 453}{\red 12} \}, \\
{\red 12} \convolution {\blue 231} & = \{ {\red 12}{\blue 453}, {\red 13}{\blue 452}, {\red 14}{\blue 352}, {\red 15}{\blue 342}, {\red 23}{\blue 451}, {\red 24}{\blue 351}, {\red 25}{\blue 341}, {\red 34}{\blue 251}, {\red 35}{\blue 241}, {\red 45}{\blue 231} \}.
\end{align*}
We also use the notation~$\underprod{\tau}{\tau'} = \tau\bar\tau'$ and~$\overprod{\tau}{\tau'} = \bar\tau'\tau$.

As shown by J.-C.~Novelli and J.-Y.~Thibon in~\cite{NovelliThibon-decoratedFQSym}, these definitions extend to decorated permutations as follows. The \defn{decorated shifted shuffle}~$\tau \shiftedShuffle \tau'$ is defined as the shifted shuffle of the permutations where decorations travel with their values, while the \defn{decorated convolution}~$\tau \convolution \tau'$ is defined as the convolution of the permutations where decorations stay at their positions. For example,
\begin{align*}
\upr{1}\updownr{2} \shiftedShuffle \downb{2}{\blue 3}\upb{1} & = \{ \upr{1}\updownr{2}\downb{4}{\blue 5}\upb{3}, \upr{1}\downb{4}\updownr{2}{\blue 5}\upb{3}, \upr{1}\downb{4}{\blue 5}\updownr{2}\upb{3}, \upr{1}\downb{4}{\blue 5}\upb{3}\updownr{2}, \downb{4}\upr{1}\updownr{2}{\blue 5}\upb{3}, \downb{4}\upr{1}{\blue 5}\updownr{2}\upb{3}, \downb{4}\upr{1}{\blue 5}\upb{3}\updownr{2}, \downb{4}{\blue 5}\upr{1}\updownr{2}\upb{3}, \downb{4}{\blue 5}\upr{1}\upb{3}\updownr{2}, \downb{4}{\blue 5}\upb{3}\upr{1}\updownr{2} \}, \\
\upr{1}\updownr{2} \convolution \downb{2}{\blue 3}\upb{1} & = \{ \upr{1}\updownr{2}\downb{4}{\blue 5}\upb{3}, \upr{1}\updownr{3}\downb{4}{\blue 5}\upb{2}, \upr{1}\updownr{4}\downb{3}{\blue 5}\upb{2}, \upr{1}\updownr{5}\downb{3}{\blue 4}\upb{2}, \upr{2}\updownr{3}\downb{4}{\blue 5}\upb{1}, \upr{2}\updownr{4}\downb{3}{\blue 5}\upb{1}, \upr{2}\updownr{5}\downb{3}{\blue 4}\upb{1}, \upr{3}\updownr{4}\downb{2}{\blue 5}\upb{1}, \upr{3}\updownr{5}\downb{2}{\blue 4}\upb{1}, \upr{4}\updownr{5}\downb{2}{\blue 3}\upb{1} \}.
\end{align*}

Using these operations, we can define a decorated version of C.~Malvenuto and C.~Reutenauer's Hopf algebra on permutations~\cite{MalvenutoReutenauer}. 

\begin{definition}
We denote by~$\FQSym_{\Decorations}$ the Hopf algebra with basis~$(\F_\tau)_{\tau \in \fS_{\Decorations}}$ and whose product and coproduct are defined by
\[
\F_\tau \product \F_{\tau'} = \sum_{\sigma \in \tau \shiftedShuffle \tau'} \F_\sigma
\qquad\text{and}\qquad
\coproduct \F_\sigma = \sum_{\sigma \in \tau \convolution \tau'} \F_\tau \otimes \F_{\tau'}.
\]
\end{definition}

We now recall well-known properties of C.~Malvenuto and C.~Reutenauer's Hopf algebra on permutations which easily translate to similar properties of~$\FQSym_{\Decorations}$.

\begin{proposition}
\label{prop:productIntervalWeakOrder}
A product of weak order intervals in~$\FQSym_{\Decorations}$ is a weak order interval: for any two weak order intervals~$[\mu, \omega] \subseteq \fS^\decoration$ and~$[\mu', \omega'] \subseteq \fS^{\decoration'}$, we have
\[
\bigg( \sum_{\mu \le \tau \le \omega} \F_\tau \bigg) \product \bigg( \sum_{\mu' \le \tau' \le \omega'} \F_{\tau'} \bigg) = \sum_{\underprod{\mu}{\mu'} \le \sigma \le \overprod{\omega}{\omega'}} \F_\sigma,
\]
where~$\le$ denotes the weak order on~$\fS^{\decoration\decoration'}$.
\end{proposition}

\begin{corollary}
\label{coro:multiplicativeBasesFQSym}
For~$\tau \in \fS^\decoration$, define
\[
\EFQSym^\tau = \sum_{\tau \le \tau'} \F_{\tau'}
\qquad\text{and}\qquad
\HFQSym^\tau = \sum_{\tau' \le \tau} \F_{\tau'}
\]
where~$\le$ is the weak order on~$\fS^\decoration$. Then~$(\EFQSym_\tau)_{\tau \in \fS}$ and~$(\HFQSym_\tau)_{\tau \in \fS}$ are multiplicative bases of~$\FQSym$:
\[
\EFQSym^\tau \product \EFQSym^{\tau'} = \EFQSym^{\underprod{\tau}{\tau'}}
\qquad\text{and}\qquad
\HFQSym^\tau \product \HFQSym^{\tau'} = \HFQSym^{\overprod{\tau}{\tau'}}.
\]
A permutation~$\tau \in \fS^\decoration$ is $\EFQSym$-decomposable (resp.~$\HFQSym$-decomposable) if and only if there exists~${k \in [n-1]}$ such that~$\tau([k]) = [k]$ (resp.~such that~$\tau([k]) = [n] \ssm [k]$). Moreover, $\FQSym_{\Decorations}$ is freely generated by the elements~$\EFQSym^\tau$ (resp.~$\HFQSym^\tau$) for all $\EFQSym$-indecomposable (resp.~$\HFQSym$-indecomposable) decorated permutations~$\tau$.
\end{corollary}

We will also consider the dual Hopf algebra of~$\FQSym_{\Decorations}$, defined as follows.

\begin{definition}
We denote by~$\FQSym_{\Decorations}^*$ the Hopf algebra with basis~$(\G_\tau)_{\tau \in \fS_{\Decorations}}$ and whose product and coproduct are defined by
\[
\G_\tau \product \G_{\tau'} = \sum_{\sigma \in \tau \convolution \tau'} \G_\sigma
\qquad\text{and}\qquad
\coproduct \G_\sigma = \sum_{\sigma \in \tau \shiftedShuffle \tau'} \G_\tau \otimes \G_{\tau'}.
\]
\end{definition}


\subsection{Subalgebra}
\label{subsec:subalgebra}

We now construct a subalgebra of~$\FQSym_{\Decorations}$ whose basis is indexed by permutrees. Namely, we denote by~$\PermutreeAlgebra$ the vector subspace of~$\FQSym_{\Decorations}$ generated by the elements
\[
\PPT_{\tree} \eqdef \sum_{\substack{\tau \in \fS_{\Decorations} \\ \PSymbol(\tau) = \tree}} \F_\tau = \sum_{\tau \in \linearExtensions(\tree)} \F_\tau,
\]
for all permutrees~$\tree$.
For example, for the permutree of \fref{fig:leveledPermutree}\,(left), we have
\[
\PPT_{\!\!\includegraphics{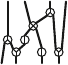}} = \F_{\up{2}\down{1}35\updown{4}\up{7}\down{6}} + \F_{\up{2}\down{1}35\up{7}\updown{4}\down{6}} + \F_{\up{2}\down{1}3\up{7}5\updown{4}\down{6}} + \dots + \F_{\up{7}5\up{2}3\down{1}\updown{4}\down{6}} + \F_{\up{7}5\up{2}3\updown{4}\down{1}\down{6}} + \F_{\up{7}5\up{2}3\updown{4}\down{6}\down{1}} \quad \text{(90 terms)}.
\]

The following statement is similar to~\cite[Theorem~24]{ChatelPilaud}, which was inspired from similar arguments for Hopf algebras arising from lattice quotients of the weak order~\cite{Reading-HopfAlgebras} and from rewriting rules in monoids~\cite{Priez}.

\begin{theorem}
\label{thm:permutreeSubalgebra}
$\PermutreeAlgebra$ is a Hopf subalgebra of~$\FQSym_{\Decorations}$.
\end{theorem}

Once we have observed this property, it is interesting to describe the product and coproduct in the Hopf algebra~$\PermutreeAlgebra$ directly in terms of permutrees. We briefly do it in the next two statements.

\para{Product}
For any permutrees~$\tree, \tree'$, denote by~$\underprod{\tree}{\tree'}$ (resp.~by~$\overprod{\tree}{\tree'}$) the permutree obtained by grafting the rightmost outgoing (resp.~incoming) edge of~$\tree$ to the leftmost incoming (resp.~outgoing) edge of~$\tree'$ while shifting all labels of~$\tree'$. An example is given in \fref{fig:exampleProductCoproduct}\,(left).

\begin{proposition}
\label{prop:product}
For any permutrees~$\tree, \tree'$, the product~$\PPT_{\tree} \product \PPT_{\tree'}$ is given by
\[
\PPT_{\tree} \product \PPT_{\tree'}  = \sum_{\tree[S]} \PPT_{\tree[S]},
\]
where~$\tree[S]$ runs over the interval between~$\underprod{\tree}{\tree'}$ and~$\overprod{\tree}{\tree'}$ in the $\decoration(\tree)\decoration(\tree')$-permutree lattice.
\end{proposition}

\begin{figure}[h]
  \centerline{\includegraphics[scale=.8]{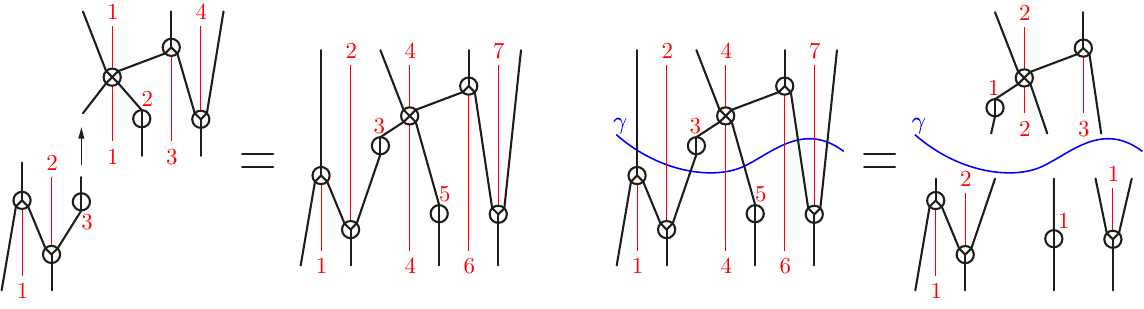}}
  \caption{Grafting two permutrees (left) and cutting a permutree (right).}
  \label{fig:exampleProductCoproduct}
\end{figure}

\para{Coproduct}
Define a \defn{cut} of a permutree~$\tree[S]$ to be a set~$\gamma$ of edges such that any geodesic vertical path in~$\tree[S]$ from a down leaf to an up leaf contains precisely one edge of~$\gamma$. Such a cut separates the permutree~$\tree[S]$ into two forests, one above~$\gamma$ and one below~$\gamma$, denoted~$A(\tree[S], \gamma)$ and~$B(\tree[S],\gamma)$, respectively. An example is given in \fref{fig:exampleProductCoproduct}\,(right).

\begin{proposition}
\label{prop:coproduct}
For any permutree~$\tree[S]$, the coproduct~$\coproduct \PPT_{\tree[S]}$ is given by
\[
\coproduct \PPT_{\tree[S]} = \sum_{\gamma} \bigg( \prod_{\tree \in B(\tree[S],\gamma)} \PPT_{\tree} \bigg) \otimes \bigg( \prod_{\tree' \in A(\tree[S], \gamma)} \PPT_{\tree'} \bigg),
\]
where~$\gamma$ runs over all cuts of~$\tree[S]$ and the products are computed from left to right.
\end{proposition}

\begin{example}
Following Example~\ref{exm:permutreesSpecificDecorations}, let us underline relevant subalgebras of the permutree algebra~$\PermutreeAlgebra$. Namely, for any collection~$\Delta$ of decorations in~$\Decorations^*$ stable by shuffle, the linear subspace of~$\PermutreeAlgebra$ generated by the elements~$\PPT_\decoration$ for~$\decoration \in \Delta$ forms a subalgebra~$\PermutreeAlgebra_\Delta$ of~$\PermutreeAlgebra$. In particular, $\PermutreeAlgebra$ contains the subalgebras:
\begin{enumerate}[(i)]
\item $\PermutreeAlgebra_{\{\noneCirc{}\}^*}$ isomorphic to C.~Malvenuto and C.~Reutenauer's Hopf algebra on permutations~\cite{MalvenutoReutenauer},
\item $\PermutreeAlgebra_{\{\downCirc{}\}^*}$ isomorphic to J.-L.~Loday and M.~Ronco's Hopf algebra of on binary trees~\cite{LodayRonco},
\item $\PermutreeAlgebra_{\{\downCirc{}, \upCirc{}\}^*}$ isomorphic to G.~Chatel and V.~Pilaud's Hopf algebra on Cambrian trees~\cite{ChatelPilaud},
\item $\PermutreeAlgebra_{\{\upDownCirc{}\}^*}$ isomorphic to I.~Gelfand, D.~Krob, A.~Lascoux, B.~Leclerc, V.~S. Retakh, and J.-Y.~Thibon's algebra on binary sequences~\cite{GelfandKrobLascouxLeclercRetakhThibon},
\end{enumerate}
as well as all algebras~$\PermutreeAlgebra_{D^*}$ for any subset~$D$ of the decorations~$\Decorations$. The dimensions of these algebras are given by the number of permutrees with decorations in~$D^*$, gathered in Table~\ref{table:factorialCatalanFamilies}. Interestingly, the rules for the product and the coproduct in the permutree algebra~$\PermutreeAlgebra$ provide uniform product and coproduct rules for all these Hopf algebras.
\end{example}


\subsection{Quotient algebra}
\label{subsec:quotientAlgebra}

The following statement is automatic by duality from Theorem~\ref{thm:permutreeSubalgebra}.

\begin{theorem}
The graded dual~$\PermutreeAlgebra^*$ of the permutree algebra~$\PermutreeAlgebra$ is the quotient of~$\FQSym_{\Decorations}^*$ under the permutree congruence~$\equiv$. The dual basis~$\QPT_{\tree}$ of~$\PPT_{\tree}$ is expressed as~$\QPT_{\tree} = \pi(\G_\tau)$, where~$\pi$ is the quotient map and~$\tau$ is any linear extension of~$\tree$.
\end{theorem}

Similarly as in the previous section, we can describe combinatorially the product and coproduct of $\QPT$-basis elements of~$\PermutreeAlgebra^*$ in terms of operations on permutrees.

\para{Product}
\enlargethispage{-1cm}
Call \defn{gaps} the $n+1$ positions between two consecutive integers of~$[n]$, including the position before~$1$ and the position after~$n$. A gap~$\gamma$ defines a \defn{geodesic vertical path}~$\lambda(\tree,\gamma)$ in a permutree~$\tree$ from the bottom leaf which lies in the same interval of consecutive down labels as~$\gamma$ to the top leaf which lies in the same interval of consecutive up labels as~$\gamma$. See \fref{fig:exampleCoproductDual}. A multiset~$\Gamma$ of gaps therefore defines a \defn{lamination}~$\lambda(\tree,\Gamma)$ of~$\tree$, \ie a multiset of pairwise non-crossing geodesic vertical paths in~$\tree$ from down leaves to up leaves. When cut along the paths of a lamination, the permutree~$\tree$ splits into a forest.

Consider two Cambrian trees~$\tree$ and~$\tree'$ on~$[n]$ and~$[n']$ respectively. For any shuffle~$s$ of their decorations~$\decoration$ and~$\decoration'$, consider the multiset~$\Gamma$ of gaps of~$[n]$ given by the positions of the down labels of~$\decoration'$ in~$s$ and the multiset~$\Gamma'$ of gaps of~$[n']$ given by the positions of the up labels of~$\decoration$ in~$s$. We denote by~$\tree \,{}_s\!\backslash \tree'$ the Cambrian tree obtained by connecting the up leaves of the forest defined by the lamination~$\lambda(\tree,\Gamma)$ to the down leaves of the forest defined by the lamination~$\lambda(\tree',\Gamma')$.

\begin{example}
\enlargethispage{-1cm}
Consider the permutrees~$\tree^\circDecoration$ and~$\tree^\squareDecoration$ of \fref{fig:exampleProductDual}. To distinguish decorations in~$\tree^\circDecoration$ and~$\tree^\squareDecoration$, we circle the symbols in~$\decoration(\tree^\circDecoration) = \upDownCirc{}\downCirc{}\upCirc{}$ and square the symbols in~$\decoration(\tree^\squareDecoration) = \noneSquare{}\upDownSquare{}\downSquare{}\upSquare{}\downSquare{}$. Consider now an arbitrary shuffle~$s = \noneSquare{}\upDownSquare{}\upDownCirc{}\downCirc{}\downSquare{}\upSquare{}\upCirc{}\downSquare{}$ of these two decorations. The resulting laminations of~$\tree^\circDecoration$ and~$\tree^\squareDecoration$, as well as the permutree~$\tree^\circDecoration {}_s\!\backslash \tree^\squareDecoration$ are represented in \fref{fig:exampleProductDual}.

\begin{figure}[h]
  \centerline{\includegraphics{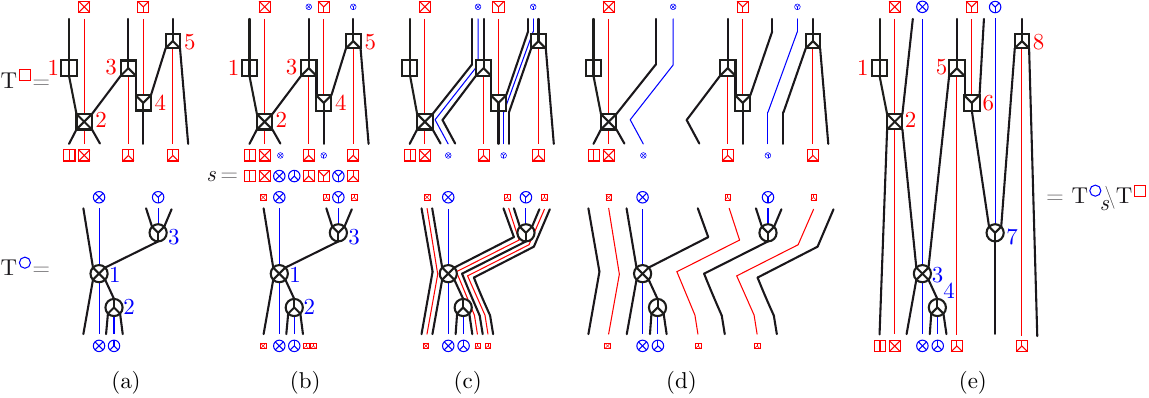}}
  \caption[Combinatorial interpretation of the product in the~$\QPT$-basis of~$\PermutreeAlgebra^*$.]{(a) The two initial permutrees~$\tree^\circDecoration$ and~$\tree^\squareDecoration$. (b) Given the shuffle $s = \noneSquare{}\upDownSquare{}\upDownCirc{}\downCirc{}\downSquare{}\upSquare{}\upCirc{}\downSquare{}$, the positions of the~\downSquare{} and~\upDownSquare{} are reported in~$\tree^\circDecoration$ and the positions of the~\upCirc{} and~\upDownCirc{} are reported in~$\tree^\squareDecoration$. (c) The corresponding laminations. (d) The permutrees are split according to the laminations. (e) The resulting permutree~$\tree^\circDecoration {}_s\!\backslash \tree^\squareDecoration$.}
  \label{fig:exampleProductDual}
\end{figure}
\end{example}

\begin{proposition}
\label{prop:productDual}
For any permutrees~$\tree, \tree'$, the product~$\QPT_{\tree} \product \QPT_{\tree'}$ is given by
\[
\QPT_{\tree} \product \QPT_{\tree'} = \sum_s \QPT_{\tree \,{}_s\!\backslash \tree'},
\]
where~$s$ runs over all shuffles of the decorations of~$\tree$ and~$\tree'$.
\end{proposition}

\para{Coproduct}
For a gap~$\gamma$, we denote by~$L(\tree[S],\gamma)$ and~$R(\tree[S],\gamma)$ the left and right subpermutrees of~$\tree[S]$ when split along the path~$\lambda(\tree[S], \gamma)$. An example is given in \fref{fig:exampleCoproductDual}.

\begin{figure}[t]
  \centerline{\includegraphics{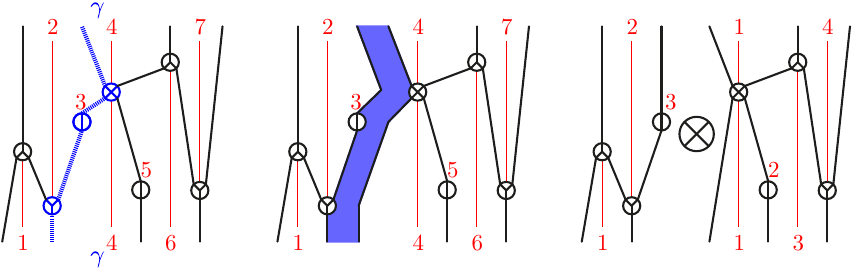}}
  \caption{A gap~$\gamma$ between~$3$ and~$4$ (left) defines a vertical cut (middle) which splits the permutree vertically (right).}
  \label{fig:exampleCoproductDual}
\end{figure}

\begin{proposition}
\label{prop:coproductDual}
For any permutree~$\tree[S]$, the coproduct~$\coproduct\QPT_{\tree[S]}$ is given by
\[
\coproduct\QPT_{\tree[S]} = \sum_{\gamma} \QPT_{L(\tree[S],\gamma)} \otimes \QPT_{R(\tree[S],\gamma)},
\]
where~$\gamma$ runs over all gaps between vertices of~$\tree[S]$.
\end{proposition}


\subsection{Further algebraic topics}
\label{subsec:furtherAlgebraicTopics}

To conclude this section, we explore some more advanced properties of the permutree algebra.


\subsubsection{Multiplicative bases and indecomposable elements}

For a permutree~$\tree$, define
\[
\EPT^{\tree} \eqdef \sum_{\tree \le \tree'} \PPT_{\tree'}
\qquad\text{and}\qquad
\HPT^{\tree} \eqdef \sum_{\tree' \le \tree} \PPT_{\tree'}.
\]
The next statement follows from Corollary~\ref{coro:multiplicativeBasesFQSym} and Proposition~\ref{prop:product}.

\begin{proposition}
$(\EPT^{\tree})_{\tree \in \Permutrees}$ and~$(\HPT^{\tree})_{\tree \in \Permutrees}$ are multiplicative bases of~$\Permutrees$:
\[
\EPT^{\tree} \product \EPT^{\tree'} = \EPT^{\underprod{\tree}{\tree'}}
\qquad\text{and}\qquad
\HPT^{\tree} \product \HPT^{\tree'} = \HPT^{\overprod{\tree}{\tree'}}.
\]
\end{proposition}

We now consider decomposition properties of permutrees. Since the $\EPT$- and $\HPT$-bases have similar properties, we focus on the $\EPT$-basis and invite the reader to translate the following properties to the $\HPT$-basis.

\begin{proposition}
The following properties are equivalent for a permutree~$\tree[S]$:
\begin{enumerate}[(i)]
\item $\EPT^{\tree[S]}$ can be decomposed into a product~$\EPT^{\tree[S]} = \EPT^{\tree} \product \EPT^{\tree'}$ for non-empty permutrees~$\tree, \tree'$; \label{enum:decomposable}
\item $\edgecut{[k]}{[n] \ssm [k]}$ is an edge cut of~$\tree[S]$ for some~$k \in [n-1]$; \label{enum:cut}
\item at least one linear extension~$\tau$ of~$\tree[S]$ is decomposable, \ie $\tau([k]) = [k]$ for some~$k \in [n]$. \label{enum:perm}
\end{enumerate}
The tree~$\tree[S]$ is then called \defn{$\EPT$-decomposable} and the edge cut~$\edgecut{[k]}{[n] \ssm [k]}$ is called \defn{splitting}.
\end{proposition}

We are interested in $\EPT$-indecomposable elements. We first understand the behavior of decomposability under rotations.

\begin{lemma}
\label{lem:rotationIndecomposable}
Let~$\tree$ be a $\decoration$-permutree, let~$i \to j$ be an edge of~$\tree$ with~$i < j$, and let~$\tree'$ be the $\decoration$-permutree obtained by rotating~$i \to j$ in~$\tree$. Then
\begin{enumerate}[(i)]
\item if~$\tree$ is $\EPT$-indecomposable, then so is~$\tree'$;
\item if~$\tree$ is $\EPT$-decomposable while~$\tree'$ is not, then~$\decoration_i \ne \downCirc{}$ or~$i = 1$, and~$\decoration_j \ne \upCirc{}$ or~$j = n$.
\end{enumerate}
\end{lemma}

\begin{proof}
Denote by~$\up{L}$ and~$\down{L}$ (resp.~$\up{R}$ and~$\down{R}$) the index sets of the left (resp.~right) ancestor and descendant subtrees of~$i$ (resp.~of~$j$) in~$\tree$ and~$\tree'$ (using~$\varnothing$ if there is no such subtree), and let~$U$ and~$D$ denote the ancestor and descendant subtrees as in \fref{fig:rotation}. The main observation is that the two permutrees~$\tree$ and~$\tree'$ have the same cuts except the cut corresponding to the edge between~$i$ and~$j$. Namely, the cut~$C \eqdef \edgecut{\{i\} \cup \down{L} \cup \up{L} \cup D}{\{j\} \cup \down{R} \cup \up{R} \cup U}$ in~$\tree$ is replaced by the cut~$C' \eqdef \edgecut{\{j\} \cup \down{R} \cup \up{R} \cup D}{\{i\} \cup \down{L} \cup \up{L} \cup U}$ in~$\tree'$. Since~$i < j$, the cut~$C'$ cannot be splitting, so that~$\tree'$ is automatically $\EPT$-indecomposable when~$\tree$ is $\EPT$-indecomposable. Assume now that~$\tree$ is $\EPT$-decomposable while~$\tree'$ is not and that~$\decoration_i = \downCirc{}$. Then the cut~$C$ is splitting so that ${\{i\} \cup \down{L} \cup D \ll \{j\} \cup \down{R} \cup \up{R} \cup U}$ (where~$X \ll Y$ means that~$x < y$ for all~$x \in X$ and~$y \in Y$). But~$\down{L} \ll \{i\} \ll D$ since~$\decoration_i = \downCirc{}$. Therefore~$\down{L} \ll \{i,j\} \cup \down{R} \cup \up{R} \cup D \cup U$. Since~$\tree'$ is $\EPT$-indecomposable, this implies that~$\down{L} = \varnothing$ since otherwise the cut~$\edgecut{\down{L}}{\{i,j\} \cup \down{R} \cup \up{R} \cup D \cup U}$ of~$\tree'$ would be splitting. Therefore, $i = 1$. We prove similarly that~$\decoration_j = \upCirc{}$ implies~$j = n$.
\end{proof}

\begin{corollary}
The set of $\EPT$-indecomposable $\decoration$-permutrees is an upper ideal of the $\decoration$-permutree lattice.
\end{corollary}

\begin{remark}
Note that contrarily to the Cambrian algebra, this ideal is not primitive in general. For example, the ideal of $\decoration$-permutrees for the decoration~$\decoration = \downCirc{}\upCirc{}\noneCirc{}\upDownCirc{}\noneCirc{}\downCirc{}\upCirc{}$ illustrated in \fref{fig:leveledPermutree} is generated by the~$4$ permutrees illustrated in \fref{fig:indecomposableGenerators}.

\begin{figure}[t]
  \centerline{\includegraphics[scale=.8]{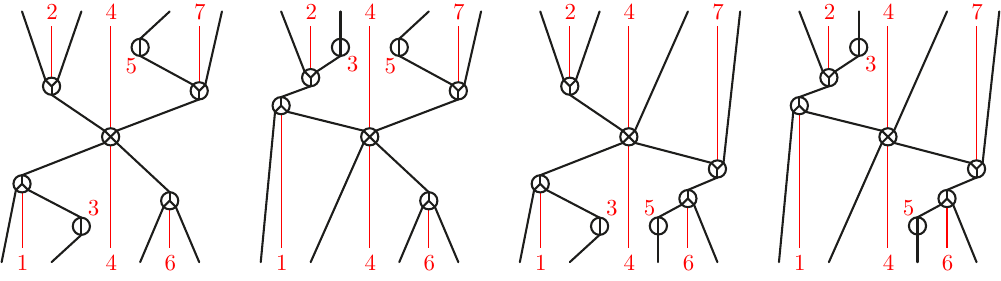}}
  \caption{The four generators of the upper ideal of $\EPT$-indecomposable \mbox{$\decoration$-permutrees} for the decoration~$\decoration = \downCirc{}\upCirc{}\noneCirc{}\upDownCirc{}\noneCirc{}\downCirc{}\upCirc{}$.}
  \label{fig:indecomposableGenerators}
\end{figure}
\end{remark}


\subsubsection{Dendriform structures}

Dendriform algebras were introduced by J.-L.~Loday in~\cite[Chap.~5]{Loday-dialgebras}. In a dendriform algebra, the product~$\product$ is decomposed into two partial products~${\product} = {\prec} + {\succ}$ satisfying:
\begin{align*}
\sfx \prec \big(\sfy \product \sfz \big) & = \big(\sfx \prec \sfy \big) \prec \sfz, \\
\sfx \succ \big(\sfy \prec \sfz \big) & = \big(\sfx \succ \sfy \big) \prec \sfz, \\
\sfx \succ \big(\sfy \succ \sfz \big) & = \big(\sfx \product \sfy \big) \succ \sfz.
\end{align*}
It is well known that the shuffle product on permutations can be decomposed into~$\shiftedShuffle = {\prec} + {\succ}$ where for~$\tau = \sigma\tau_n$ and~$\tau' = \sigma'\tau'_{n'}$ we have
\[
\tau \prec \tau' \eqdef (\sigma \shiftedShuffle \tau') \tau_n
\qquad\text{and}\qquad
\tau \succ \tau' \eqdef (\tau \shiftedShuffle \sigma') \bar\tau'_{n'}.
\]
This endows C.~Malvenuto and C.~Reutenauer's algebra on permutations with a dendriform algebra structure defined on the $\F$-basis by
\[
\F_\tau \prec \F_{\tau'} = \sum_{\sigma \in \tau \prec \tau'} \F_\sigma
\qquad\text{and}\qquad
\F_\tau \succ \F_{\tau'} = \sum_{\sigma \in \tau \succ \tau'} \F_\sigma.
\]
It turns out that some subalgebras of the permutree algebra are stable under these dendriform operations~$\prec$ and~$\succ$.

\begin{proposition}
\label{prop:dendriform}
For any subset~$\Delta$ of~$\{\noneCirc{}, \downCirc{}\}^*$ stable by shuffle, the subalgebra of the permutree algebra~$\PermutreeAlgebra$ generated by~$\set{\PPT_{\tree}}{\tree \in \Permutrees(\decoration), \; \decoration \in \Delta}$ is stable by the dendriform operations~$\prec$ and~$\succ$.
\end{proposition}

\begin{proof}
Let~$\decoration \in \{\noneCirc{}, \downCirc{}\}^n$ and~$\decoration' \in \{\noneCirc{}, \downCirc{}\}^{n'}$, let~$\tree \in \Permutrees(\decoration)$ and~$\tree' \in \Permutrees(\decoration')$, and let~$\tau = \sigma\tau_n \in \fS_\decoration$ and~$\tau' = \sigma'\tau'_{n'} \in \fS_{\decoration'}$ be such that~$\tree = \PSymbol(\tau)$ and~$\tree' = \PSymbol(\tau')$. Then
\[
\PPT_{\tree} \prec \PPT_{\tree'} = \sum_{\PSymbol(\sigma\bar\tau'\tau_n) \le \tree[S] \le \overprod{\tree}{\tree'}} \PPT_{\tree[S]}
\qquad\text{and}\qquad
\PPT_{\tree} \succ \PPT_{\tree'} = \sum_{\underprod{\tree}{\tree'} \le \tree[S] \le \PSymbol(\bar\sigma'\tau\bar\tau'_{n'})} \PPT_{\tree[S]}.
\]
Indeed, $\PPT_{\tree} \prec \PPT_{\tree'}$ is the sum of~$\F_\sigma$ for~$\tau\bar\tau' \le \sigma \le \bar\tau'\tau$ such that~$\sigma_{n+n'} = \tau_n$. These permutations are exactly all linear extensions of the trees~$\tree[S]$ such that~$\underprod{\tree}{\tree'} \le \tree[S] \le \overprod{\tree}{\tree'}$ whose root is~$\tau_n$. Therefore, $\PPT_{\tree} \prec \PPT_{\tree'}$ is the sum of~$\PPT_{\tree[S]}$ for all these trees. Similarly, $\PPT_{\tree} \succ \PPT_{\tree'}$ is the sum of~$\PPT_{\tree[S]}$ for the trees~$\tree[S]$ such that~$\underprod{\tree}{\tree'} \le \tree[S] \le \overprod{\tree}{\tree'}$ whose root is~$\bar\tau'_{n'}$.
\end{proof}

\begin{remark}
Note that the assumption that~$\Delta \subset \{\noneCirc{}, \downCirc{}\}^*$ in Proposition~\ref{prop:dendriform} is necessary. For example, we have
\[
\PPT_{\includegraphics{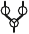}} \prec \PPT_{\!\!\includegraphics{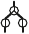}} = (\F_{\up{2}13} + \F_{\up{2}31}) \prec (\F_{13\down{2}} + \F_{31\down{2}})
\]
contains~$\F_{\up{2}346\down{5}1}$ but not~$\F_{\up{2}3461\down{5}}$ although~$\PSymbol(\up{2}346\down{5}1) = \PSymbol(\up{2}3461\down{5})$.
\end{remark}


\subsubsection{Integer point transform}

\enlargethispage{-.3cm}
We now show that the product of two permutrees can be interpreted in terms of their integer point transforms. This leads to relevant equalities for permutrees with decorations in~$\{\noneCirc{}, \downCirc{}\}$.

\begin{definition}
The \defn{integer point transform}~$\integerPointTransform_S$ of a subset~$S$ of~$\R^n$ is the multivariate generating function of the integer points inside~$S$:
\[
\integerPointTransform_S(t_1, \dots, t_n) = \sum_{(i_1, \dots, i_n) \in \Z^n \cap S} t_1^{i_1} \cdots t_n^{i_n}.
\]
\end{definition}

For a permutree~$\tree$, we denote by~$\integerPointTransform_{\tree}$ the integer point transform of the cone
\[
\Cone\blackPolar(\tree) \eqdef \set{\b{x} \in \R_+^n}{\begin{array}{c} x_i \le x_j \text{ for any edge } i \to j \text{ of } \tree \text{ with } i < j \\  x_i < x_j \text{ for any edge } i \to j \text{ of } \tree \text{ with } i > j \end{array}}.
\]
Note that this cone differs from the cone~$\Cone\polar(\tree)$ defined in Section~\ref{subsec:permutreeFans} in two ways: first it leaves in~$\R_+^n$ and not in~$\HH$, second it excludes the facets of~$\Cone\polar(\tree)$ corresponding to the decreasing edges of~$\tree$ (\ie the edges~$i \to j$ with~$i > j$). We denote by~$\integerPointTransform_\tau$ the integer point transform of the chain~$\tau_1 \to \cdots \to \tau_n$ for a permutation~$\tau \in \fS_n$. The following statements are classical.

\begin{proposition}
\label{prop:integerPointTransform}
\begin{enumerate}[(i)]
\item For any permutation~$\tau \in \fS_n$, the integer point transform~$\integerPointTransform_\tau$ is given by
\[
\integerPointTransform_\tau(t_1, \dots, t_n) = \bigg( \prod\limits_{i \in [n]} \big( 1 - t_{\tau_i} \cdots t_{\tau_n} \big)^{-1} \bigg) \bigg( \prod_{\substack{i \in [n-1] \\ \tau_i > \tau_{i+1}}} t_{\tau_i+1} \cdots t_{\tau_n} \bigg).
\]
\label{item:integerPointTransformPermutation}

\item The integer point transform of an arbitrary permutree~$\tree$ is given by~$\integerPointTransform_{\tree} = \sum_{\tau \in \linearExtensions(\tree)} \integerPointTransform_\tau$, where the sum runs over the set~$\linearExtensions(\tree)$ of linear extensions of~$\tree$.
\label{item:integerPointTransformLinearExtensions}

\item The product of the integer point transforms~$\integerPointTransform_\tau$ and~$\integerPointTransform_{\tau'}$ of two permutations~${\tau \in \fS_n}$ and~${\tau' \in \fS_{n'}}$ is given by the shifted shuffle
\[
\integerPointTransform_\tau(t_1, \dots, t_n) \product \integerPointTransform_{\tau'}(t_{n+1}, \dots, t_{n+n'}) = \sum_{\sigma \in \tau \shiftedShuffle \tau'} \integerPointTransform_\sigma(t_1, \dots, t_{n+n'}).
\]
In other words, the linear map from~$\FQSym$ to the rational functions defined by~$\Psi : \F_\tau \mapsto \integerPointTransform_\tau$ is an algebra morphism.
\label{item:integerPointTransformProduct}
\end{enumerate}
\end{proposition}

\begin{proof}
For Point~\eqref{item:integerPointTransformPermutation}, we just observe that the cone~$\set{\b{x} \in \R_+^n}{x_{\tau_i} \le x_{\tau_{i+1}} \text{ for all } i \in [n-1]}$ is generated by the vectors~$\b{e}_{\tau_i} + \cdots + \b{e}_{\tau_n}$, for~$i \in [n]$, which form a (unimodular) basis of the lattice~$\Z^n$. A straightforward inductive argument shows that the integer point transform of the cone $\set{\b{x} \in \R_+^n}{x_{\tau_i} \le x_{\tau_{i+1}} \text{ for all } i \in [n-1]}$ is thus given by~$\prod_{i \in [n]} \big( 1- t_{\tau_i} \cdots t_{\tau_n} \big)^{-1}$. The second product of~$\integerPointTransform_\tau$ is then given by the facets which are excluded from the cone~$\Cone\blackPolar(\tau)$.

Point~\eqref{item:integerPointTransformLinearExtensions} follows from the fact that the cone~$\Cone\blackPolar(\tree)$ is partitioned by the cones~$\Cone\blackPolar(\tau)$ for the linear extensions~$\tau$ of~$\tree$.

Finally, the product~$\integerPointTransform_\tau(t_1, \dots, t_n) \product \integerPointTransform_{\tau'}(t_{n+1}, \dots, t_{n+n'})$ is the integer point transform of the poset formed by the two disjoint chains~$\tau$ and~$\bar\tau'$, whose linear extensions are precisely the permutations which appear in the shifted shuffle of~$\tau$ and~$\tau'$. This shows Point~\eqref{item:integerPointTransformProduct}.
\end{proof}

It follows from Proposition~\ref{prop:integerPointTransform} that the product of the integer point transforms of two permutrees behaves as the product in the permutree algebra~$\PermutreeAlgebra$.

\begin{corollary}
For any two permutrees~$\tree \in \Permutrees(n)$ and~$\tree' \in \Permutrees(n')$, we have
\[
\integerPointTransform_{\tree}(t_1, \dots, t_n) \product \integerPointTransform_{\tree'}(t_{n+1}, \dots, t_{n+n'}) = \sum_{\underprod{\tree}{\tree'} \, \le \, \tree[S] \, \le \, \overprod{\tree}{\tree'}} \integerPointTransform_{\tree[S]}(t_1, \dots, t_{n+n'}).
\]
\end{corollary}

\begin{proof}
Omitting the variables~$(t_1, \dots, t_{n+n'})$ for concision, we have
\[
\integerPointTransform_{\tree} \product \integerPointTransform_{\tree'} = \Psi(\PPT_{\tree}) \product \Psi(\PPT_{\tree'}) = \Psi(\PPT_{\tree} \cdot \PPT_{\tree'}) = \Psi \bigg( \sum_{\tree[S]} \PPT_{\tree[S]} \bigg) = \sum_{\tree[S]} \Psi(\PPT_{\tree[S]}) = \sum_{\tree[S]} \integerPointTransform_{\tree[S]},
\]
where the sums run over the permutrees~$\tree[S]$ of the increasing flip lattice interval~$[\underprod{\tree}{\tree'}, \overprod{\tree}{\tree'}]$.
\end{proof}

Finally, we specialize this result to permutrees with decorations in~$\{\noneCirc{}, \upCirc{}\}^*$. The main observation is the following.

\begin{proposition}
The integer point transform of any permutree~$\tree$ with decoration in~$\{\noneCirc{}, \upCirc{}\}^*$ is given by
\[
\integerPointTransform_{\tree}(t_1, \dots, t_n) = \bigg( \prod_{\edgecut{I}{J} \in \cuts(\tree)} \big( 1-\prod_{j \in J} t_j \big)^{-1} \bigg) \bigg( \prod_{\edgecut{I}{J} \in \decreasingCuts(\tree)} \; \prod_{j \in J} t_j \bigg),
\]
where~$\cuts(\tree)$ denotes the set of all edge cuts of~$\tree$ (including the artificial edge cut~$\edgecut{\varnothing}{[n]}$) and $\decreasingCuts(\tree)$ denotes the decreasing edge cuts of~$\tree$, \ie those corresponding to edges~$i \to j$ with~$i > j$.
\end{proposition}

\begin{proof}
Consider the cone~$C$ defined by the inequalities~$0 \le x_i$ for all~$i \in [n]$ and~$x_i \le x_j$ for all edges~$i \to j$ in~$\tree$. As the decoration of~$\tree$ is in~$\{\noneCirc{}, \upCirc{}\}^*$, the tree~$\tree$ is naturally rooted at its bottommost vertex~$r$. Since there is a path from~$r$ to any other vertex~$v$, the inequalities~$0 \le x_r$ and~$x_i \le x_j$ for all edges~$i \to j$ in~$\tree$ already imply all other inequalities~$0 \le x_i$ for~${i \in [n] \ssm \{r\}}$. We therefore obtain that the cone~$C$ is defined by~$n$ inequalities and thus is simplicial. Moreover, we easily check that the rays of~$C$ are given by the characteristic vectors~$\sum_{j \in J} \b{e}_j$ for all edge cuts~$\edgecut{I}{J} \in \cuts(\tree)$, including the vector~$\one$ for the artificial edge cut~$\edgecut{\varnothing}{[n]}$. We obtain that the integer point transform of~$C$ is~$\prod_{\edgecut{I}{J} \in \cuts(\tree)} \big( 1-\prod_{j \in J} t_j \big)^{-1}$. Finally, the second product of~$\integerPointTransform_{\tree}$ is given by the facets which are excluded from~$C$ to obtain~$\Cone\blackPolar(\tree)$.
\end{proof}

\begin{corollary}
For any permutrees~$\tree \in \Permutrees(n)$ and~$\tree' \in \Permutrees(n')$ with decorations in~$\{\noneCirc{}, \upCirc{}\}^*$, we have
\[
\frac{\prod_{\edgecut{I}{J} \in \decreasingCuts(\tree)} \; \prod_{j \in J} t_j}{\prod_{\edgecut{I}{J} \in \cuts(\tree)} \big( 1-\prod_{j \in J} t_j \big)} \cdot \frac{\prod_{\edgecut{I}{J} \in \decreasingCuts(\tree')} \; \prod_{j \in J} t_{n+j}}{\prod_{\edgecut{I}{J} \in \cuts(\tree')} \big( 1-\prod_{j \in J} t_{n+j} \big)} = \sum_{\tree[S]}  \frac{\prod_{\edgecut{I}{J} \in \decreasingCuts(\tree[S])} \; \prod_{j \in J} t_j}{\prod_{\edgecut{I}{J} \in \cuts(\tree[S])} \big( 1-\prod_{j \in J} t_j \big)},
\]
where~$\tree[S]$ ranges over the permutree lattice interval~$[\underprod{\tree}{\tree'}, \overprod{\tree}{\tree'}]$.
\end{corollary}

\begin{example}
For example, we have the following equality of rational functions:
\begin{align*}
\integerPointTransform_{\includegraphics{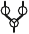}} \cdot \integerPointTransform_{\includegraphics{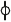}}
= & \; (1-x_1)^{-1}x_1 (1-x_3)^{-1} (1-x_1x_2x_3)^{-1} \cdot (1-x_4)^{-1} \\[-.2cm]
= & \quad\; (1-x_1)^{-1}x_1 (1-x_3x_4)^{-1} (1-x_4)^{-1} (1-x_1x_2x_3x_4)^{-1} \\
  & + (1-x_1)^{-1}x_1 (1-x_3x_4)^{-1} (1-x_3)^{-1}x_3 (1-x_1x_2x_3x_4)^{-1} \\
  & + (1-x_1x_2x_3)^{-1}x_1x_2x_3 (1-x_1)^{-1}x_1 (1-x_3)^{-1} (1-x_1x_2x_3x_4)^{-1} \\[.1cm]
= & \; \integerPointTransform_{\includegraphics{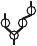}} + \integerPointTransform_{\includegraphics{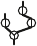}} + \integerPointTransform_{\includegraphics{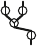}}.
\end{align*}
\end{example}

\begin{remark}
\enlargethispage{-1cm}
Note that the simple product formula for the integer point transform~$\integerPointTransform_{\tree}$ does not hold for an arbitrary permutree~$\tree$. Indeed, the cone~$\Cone\blackPolar(\tree)$ is not always simplicial. \vspace{-.05cm} For example, the cone~$\Cone\blackPolar(\raisebox{-.2cm}{\includegraphics{permutreeWedge}}\vspace{-.1cm})$ is generated by the vectors~$[0,1,0], [0,1,1], [1,1,0], [1,1,1]$, where the first is not the characteristic vector~$\sum_{j \in J} \b{e}_j$ for an edge cut~$\edgecut{I}{J}$ of~$\raisebox{-.2cm}{\includegraphics{permutreeWedge}}$.
\end{remark}


\section{Schr\"oder permutrees}
\label{sec:SchroderPermutrees}

This section is devoted to Schr\"oder permutrees which correspond to the faces of the permutreehedra. It is largely inspired from the presentation of~\cite[Part 3]{ChatelPilaud}.


\subsection{Schr\"oder permutrees}
\label{subsec:SchroderPermutrees}

In this section, we focus on the following family of trees.

\begin{definition}
\label{def:SchroderPermutrees}
For~$\decoration \in \Decorations^n$ and~$X \subseteq [n]$, we define~$X^\vee \eqdef \set{x \in X}{\decoration_x \in \{\upCirc{}, \upDownCirc{}\}}$ and~$X_\wedge \eqdef \set{x \in X}{\decoration_x \in \{\downCirc{}, \upDownCirc{}\}}$. A \defn{Schr\"oder $\decoration$-permutree} is a directed tree~$\tree[S]$ with vertex set~$\ground$ endowed with a vertex labeling~${p : \ground \to 2^{[n]} \ssm \varnothing}$ such that
\begin{enumerate}[(i)]
\item the labels of~$\tree[S]$ partition~$[n]$, \ie $v \ne w \in \ground \implies p(v) \cap p(w) = \varnothing$ and~$\bigcup_{v \in \ground} p(v) = [n]$;
\item each vertex~$v \in \ground$ has one incoming (resp.~outgoing) subtree~$\tree[S]^v_I$ (resp.~$\tree[S]_v^I$) for each interval~$I$ of~${[n] \ssm p(v)_\wedge}$ (resp.~of~$[n] \ssm p(v)^\vee$) and all labels of~$\tree[S]^v_I$ (resp.~of~$\tree[S]_v^I$) are subsets of~$I$.
\end{enumerate}
For~$\decoration \in \Decorations^n$, we denote by~$\SchrPermutrees(\decoration)$ the set of Schr\"oder $\decoration$-permutrees, and we define~$\SchrPermutrees(n) \eqdef \bigsqcup_{\decoration \in \Decorations^n} \SchrPermutrees(\decoration)$ and $\SchrPermutrees \eqdef \bigsqcup_{n \in \N} \SchrPermutrees(n)$.
\end{definition}

\begin{definition}
A \defn{$k$-leveled Schr\"oder $\decoration$-permutree} is a directed tree with vertex set~$\ground$ endowed with two labelings~$p : \ground \to 2^{[n]} \ssm \varnothing$ and~$q : \ground \to [k]$ which respectively define a Schr\"oder $\decoration$-permutree and an increasing tree (meaning that $q$ is surjective and~$v \to w$ in~$\tree[S]$ implies that~$q(v) < q(w)$).
\end{definition}

\fref{fig:leveledSchroderPermutree} illustrates a Schr\"oder permutree and a $3$-leveled Schr\"oder permutree. For example, for the node~$v$ labeled by~$p(v) = \{4,6\}$, we have~$p(v)_\wedge = \{4,6\}$ so that~$[7] \ssm p(v)_\wedge = \{1,2,3\} \sqcup \{5\} \sqcup \{7\}$ and~$v$ has $3$ incoming subtrees and~$p(v)_\vee = \{4\}$ so that~$[7] \ssm p(v)_\vee = \{1,2,3\} \sqcup \{5,6,7\}$ and~$v$ has~$2$ outgoing subtrees. Note that each level of a $k$-leveled Schr\"oder permutree may contain more than one node.

\begin{figure}[t]
  \centerline{\includegraphics{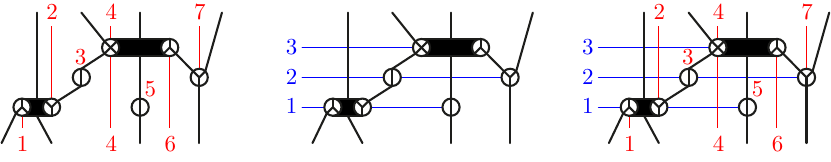}}
  \caption{A Schr\"oder permutree (left), an increasing tree (middle), and a $3$-leveled Schr\"oder permutree (right).}
  \label{fig:leveledSchroderPermutree}
\end{figure}

\begin{example}
Following Example~\ref{exm:permutreesSpecificDecorations}, observe that Schr\"oder $\decoration$-permutrees specialize to classical families of combinatorial objects:
\begin{enumerate}[(i)]
\item Schr\"oder $\noneCirc{}^n$-permutrees are in bijection with \defn{ordered partitions} of~$[n]$, \ie with sequences $\lambda \eqdef \lambda_1 \sep \lambda_2 \sep \dots \sep \lambda_{k-1} \sep \lambda_k$ where~$\lambda_i \subseteq [n]$ are such that~$\bigcup_{i \in [k]} \lambda_i = [n]$ and~$\lambda_i \cap \lambda_j = \varnothing$ for~$i \ne j$.
\item Schr\"oder $\downCirc{}^n$-permutrees are precisely \defn{Schr\"oder trees}, \ie planar rooted trees where each node has at least two children.
\item Schr\"oder $\decoration$-permutrees with~$\decoration \in \{\upCirc{}, \downCirc{}\}$ are Schr\"oder Cambrian trees~\cite[Section 3.1]{ChatelPilaud}.
\item Schr\"oder $\upDownCirc{}^n$-permutrees are in bijection with ternary sequences of length~$n-1$.
\end{enumerate}
\fref{fig:orderedPartitionsSchroderTreesSchroderCambrianTreesTernarySequences} illustrates these families represented as Schr\"oder permutrees. In this section, we provide a uniform treatment of these families.

\begin{figure}[b]
  \centerline{\includegraphics[scale=.8]{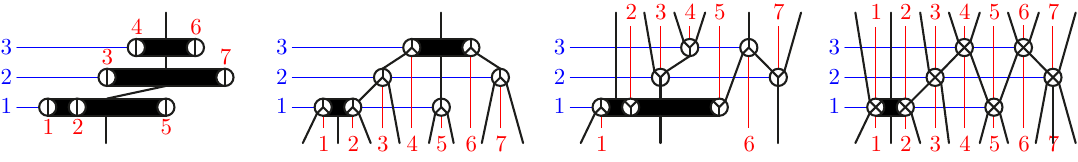}}
  \caption{Leveled Schr\"oder permutrees corresponding to an ordered partition (left), a leveled Schr\"oder tree (middle left), a leveled Schr\"oder Cambrian tree (middle right), and a leveled ternary sequence (right).}
  \label{fig:orderedPartitionsSchroderTreesSchroderCambrianTreesTernarySequences}
\end{figure}
\end{example}

\begin{remark}
Similar to Remark~\ref{rem:234angulations}, Schr\"oder $\decoration$-permutrees are dual trees of dissections of~$\b{P}_\decoration$, that is, non-crossing sets of arcs in~$\b{P}_\decoration$ (as defined in Remark~\ref{rem:234angulations}). See \fref{fig:234angulationSchroder} for an illustration. Note that contracting an edge in a Schr\"oder permutree corresponds to deleting an arc in its dual dissection.

\begin{figure}[t]
  \centerline{\includegraphics[scale=1]{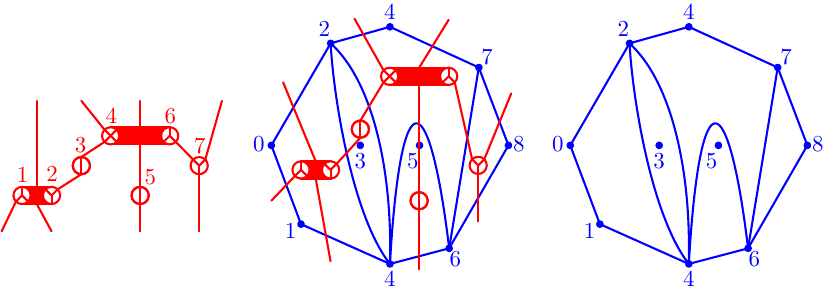}}
  \caption{Schr\"oder permutrees (left) and dissections (right) are dual to each other.}
  \label{fig:234angulationSchroder}
\end{figure}
\end{remark}

We will need the following statement in the next section.

\begin{lemma}
Schr\"oder $\decoration$-permutrees are stable by edge contraction.
\end{lemma}

\begin{proof}
Let~$e = v \to w$ be an edge in a Schr\"oder $\decoration$-permutree~$\tree[S]$, and let~$\tree[S]/e$ be the tree obtained by contraction of~$e$, where the contracted vertex $vw$ gets the label~$p(vw) = p(v) \cup p(w)$. Condition~(i) of Definition~\ref{def:SchroderPermutrees} is clearly satisfied. For Condition~(ii), let~$I_1, \dots, I_p$ be the intervals of~$[n] \ssm p(w)_\wedge$, where~$I_i$ is the interval which contains~$p(v)$, and let~$J_1, \dots, J_q$ be the intervals of~$[n] \ssm p(v)_\wedge$. Observe that for all~$j \in [q]$, all labels of~$\tree[S]^v_{J_j}$ are all included in~$I_i$ (because~$\tree[S]^v_{J_j}$ belongs to~$\tree[S]^w_{I_i}$) and in~$J_j$. Therefore, the descendant subtrees~$\tree[S]^w_{I_1}, \dots, \tree[S]^w_{I_{i-1}}, \tree[S]^v_{J_1}, \dots, \tree[S]^v_{J_q}, \tree[S]^w_{I_{i+1}}, \dots, \tree[S]^w_{I_p}$ of~$vw$ in~$\tree[S]/e$ indeed belong to the intervals~$I_1, \dots, I_{i-1}, I_i \cap J_1, \dots, I_i \cap J_q, I_{i+1}, \dots, I_p$ of~$[n] \ssm p(vw)$. The proof is similar for ancestor subtrees of~$vw$ in~$\tree[S]/e$. The other vertices of~$\tree[S]/e$ have not changed locally.
\end{proof}

Let~$\tree[S], \tree[S]'$ be two Schr\"oder $\decoration$-permutrees. We say that~$\tree[S]$ \defn{refines}~$\tree[S]'$ if~$\tree[S]'$ can be obtained from~$\tree[S]$ by contraction of some edges.


\subsection{Faces}
\label{subsec:faces}

Consider a Schr\"oder permutree~$\tree[S]$. For~$i,j \in [n]$, we write~$i \to j$ if there are vertices~$v,w$ of~$\tree[S]$ such that~$i \in p(v)$, $j \in p(w)$ and~$v = w$ or~$v \to w$ is an edge in~$\tree[S]$. We say that~$\edgecut{I}{J}$ is an edge cut of~$\tree[S]$ if~$I$ (resp.~$J$) is the union of all labels of the vertices in the source (resp.~sink) set of an edge of~$\tree[S]$. Define the \defn{braid cone}~$\Cone\polar(\tree[S])$ of~$\tree[S]$ as the cone
\[
\Cone\polar(\tree[S]) \eqdef \set{\b{x} \in \HH}{x_i \le x_j \text{ for any } i \to j \text{ in } \tree[S]} = \cone \biggset{\sum_{j \in J} \b{e}_j}{\text{for all edge cuts } \edgecut{I}{J} \text{ of } \tree[S]}.
\]

\begin{proposition}
\label{prop:fan}
The map~$\tree[S] \mapsto \Cone\polar(\tree[S])$ is a lattice homomorphism from the refinement lattice on Schr\"oder $\decoration$-permutrees to the inclusion lattice on the cones of the $\decoration$-permutree fan~$\Fan$. In particular,
\[
\Fan = \set{\Cone\polar(\tree[S])}{\tree[S] \text{ Schr\"oder $\decoration$-permutree}}.
\]
\end{proposition}

\begin{proof}
Contracting an edge~$v \to w$ in~$\tree[S]$ corresponds to forcing the inequalities~$x_i \le x_j$ to become equalities~$x_i = x_j$ for~$i \in p(v)$ and~$j \in p(w)$. Reciprocally, since the cone~$\Cone\polar(\tree[S])$ is simplicial, its faces are obtained by forcing some of its inequalities to become equalities, which corresponds to contracting some edges in~$\tree[S]$. The result immediately follows.
\end{proof}

\begin{proposition}
For any Schr\"oder $\decoration$-permutree~$\tree[S]$, the set
\[
\face(\tree[S]) \eqdef \conv\set{\b{a}(\tree)}{\tree \; \decoration\text{-permutree refining } \tree[S]}= \bigcap_{\edgecut{I}{J} \text{ cut of } \tree[S]} \HS(I)
\]
is a face of the $\decoration$-permutreehedron~$\Permutreehedron$. Moreover, the map~$\tree[S] \mapsto \face(\tree[S])$ is a lattice homomorphism from the refinement lattice on Schr\"oder $\decoration$-permutrees to the face lattice of the $\decoration$-permutreehedron~$\Permutreehedron$.
\end{proposition}

\begin{proof}
Since the normal fan of the $\decoration$-permutreehedron~$\Permutreehedron$ is the $\decoration$-permutree fan~$\Fan$ by Theorem~\ref{theo:permutreehedra}, there is a face~$\face(\tree[S])$ whose normal cone is~$\Cone\polar(\tree[S])$. This face is given by the inequalities of~$\Permutreehedron$ corresponding to the rays of~$\Cone\polar(\tree[S])$, that is, by the inequalities~$\HS(I)$ for the edge cuts~$\edgecut{I}{J}$ of~$\tree[S]$. Moreover, a $\decoration$-permutree satisfies these inequalities if and only if it refines~$\tree[S]$, which proves the second equality. Finally, the lattice homomorphism property is a direct consequence of Proposition~\ref{prop:fan}.
\end{proof}


\subsection{Schr\"oder permutree correspondence}
\label{subsec:SchroderPermutreeCorrespondence}

We now define an analogue of the permutree correspondence and $\PSymbol$-symbol, which will map decorated ordered partitions of~$[n]$ to Schr\"oder permutrees. We represent graphically an ordered partition~$\lambda \eqdef \lambda_1 \sep \cdots \sep \lambda_k$ of~$[n]$ into~$k$ parts by the~$(k \times n)$-table with a dot at row~$i$ and column~$j$ for each~$j \in \lambda_i$. See \fref{fig:insertionAlgorithmSchroder}\,(left). We denote by~$\fP_n$ the set of ordered partitions of~$[n]$ and we set~$\fP \eqdef \bigsqcup_{n \in \N} \fP_n$.

A \defn{decorated ordered partition} is an ordered partition table where each dot receives a decoration of~$\Decorations$. For a decoration~$\decoration \in \Decorations^n$, we denote by~$\fP_\decoration$ the set of ordered partitions of~$[n]$ decorated by~$\decoration$, and we set~$\fP_{\Decorations} \eqdef \bigsqcup_{n \in \N, \decoration \in \Decorations^n} \fP_\decoration.$

Given such a decorated ordered partition~$\lambda \eqdef \lambda_1 \sep \cdots \sep \lambda_k$, we construct a leveled Schr\"oder-Cambrian tree~$\SchroderPermutreeCorresp(\lambda)$ as follows. As a preprocessing, we represent the table of~$\lambda$ (with a dot at row~$i$ and column~$j$ for each~$j \in \lambda_i$), we draw a vertical red wall below the down dots (decorated by~\downCirc{} or~\upDownCirc{}) and above the up dots (decorated by~\upCirc{} or~\upDownCirc{}), and we connect into nodes the dots at the same level which are not separated by a wall. Note that we might obtain several nodes per level. We then sweep the table from bottom to top as follows. The procedure starts with an incoming strand in between any two consecutive down values. At each level, each node~$v$ (connected set of dots) gathers all strands in the region below and visible from~$v$ (\ie not hidden by a vertical wall) and produces one strand in each region above and visible from~$v$. The procedure finished with an outgoing strand in between any two consecutive up values. See \fref{fig:insertionAlgorithmSchroder}.

\begin{figure}[t]
  \centerline{\includegraphics[width=\textwidth]{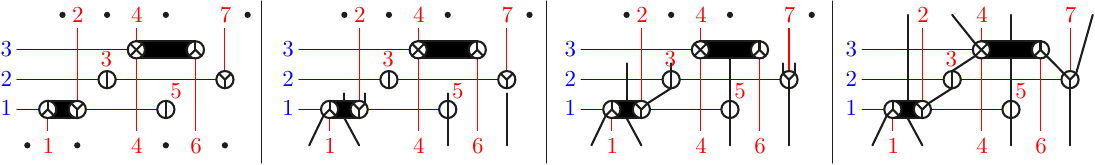}}
  \caption{The insertion algorithm on the decorated ordered partition~$\down{1}\up{2}5 \sep 3\up{7} \sep \updown{4}\down{6}$.}
  \label{fig:insertionAlgorithmSchroder}
\end{figure}

\begin{example}
As illustrations of the Schr\"oder permutree correspondence, the leveled Schr\"oder permutrees of \fref{fig:orderedPartitionsSchroderTreesSchroderCambrianTreesTernarySequences} were all obtained by inserting the ordered partition~$125 \sep 37 \sep 46$ with different decorations.
\end{example}

\begin{proposition}
The map~$\SchroderPermutreeCorresp$ is a bijection from decorated ordered partitions to leveled Schr\"oder-Cambrian trees.
\end{proposition}

\begin{proof}
The proof is similar to that of Proposition~\ref{prop:permutreeCorrespondence}.
\end{proof}

For a decorated ordered partition~$\lambda$, we denote by~$\SchroderPSymbol(\tau)$ the permutree obtained by forgetting the increasing labeling in~$\SchroderPermutreeCorresp(\tau)$. Note that a decorated ordered partition of~$[n]$ into~$k$ parts is sent to a Schr\"oder permutree with at least~$k$ internal nodes, since some levels can be split into several~nodes.

Similar to Proposition~\ref{prop:fibersPSymbol}, the following characterization of the fibers of the map~$\SchroderPSymbol$ is immediate from the description of the Schr\"oder permutree correspondence. For an ordered partition~$\lambda \eqdef \lambda_1 \sep \dots \sep \lambda_k$, we write~$\lambda^{-1}(i)$ the index of the part such that~$i \in \lambda_{\lambda^{-1}(i)}$. For a Schr\"oder permutree~$\tree[S]$, we write~$i \to j$ in~$\tree[S]$ if the node of~$\tree[S]$ containing~$i$ is below the node of~$\tree[S]$ containing~$j$, and~$i \sim j$ in~$\tree[S]$ if~$i$ and~$j$ belong to the same node of~$\tree[S]$. We say that~$i$ and~$j$ are \defn{incomparable} in~$\tree[S]$ when~$i \not\to j$, $j \not\to i$, and~$i \not\sim j$.

\begin{proposition}
\label{prop:mergeLinearExtensions}
For any Schr\"oder $\decoration$-permutree~$\tree[S]$ and decorated ordered partition~$\lambda \in \fP_\decoration$, we have~$\SchroderPSymbol(\lambda) = \tree[S]$ if and only if~$i \sim j$ in~$\tree[S]$ implies $\lambda^{-1}(i) = \lambda^{-1}(j)$ and $i \to j$ in~$\tree[S]$ implies $\lambda^{-1}(i) < \lambda^{-1}(j)$.
In other words, $\lambda$ is obtained from a linear extension of~$\tree[S]$ by merging parts which label incomparable vertices of~$\tree[S]$.
\end{proposition}


\subsection{Schr\"oder permutree congruence}
\label{subsec:SchroderPermutreeCongruence}

Similar to the permutree congruence, we now characterize the fibers of $\SchroderPSymbol$ by a congruence defined as a rewriting rule. Remember that we write~$X \ll Y$ when~$x < y$ for all~$x \in X$ and~$y \in Y$, that is, when~$\max(X) < \min(Y)$.

\begin{definition}
For a decoration~$\decoration \in \Decorations^n$, the Schr\"oder $\decoration$-permutree congruence is the equivalence relation on~$\fP_\decoration$ defined as the transitive closure of the rewriting rules
\[
U \sep \b{a} \sep \b{c} \sep V \equiv^\star_\decoration U \sep \b{a}\b{c} \sep V \equiv^\star_\decoration U \sep \b{c} \sep \b{a} \sep V,
\]
where $\b{a}, \b{c}$ are parts while~$U, V$ are sequences of parts of~$[n]$, and there exists~$\b{a} \ll b \ll \b{c}$ such that~$\decoration_b \in \{\upCirc{}, \upDownCirc{}\}$ and~$b \in \bigcup U$, or~$\decoration_b \in \{\downCirc{}, \upDownCirc{}\}$ and~$b \in \bigcup V$. The Schr\"oder permutree congruence is the equivalence relation~$\equiv^\star$ on~$\fP_{\Decorations}$ obtained as the union of all Schr\"oder $\decoration$-permutree congruences.
\end{definition}

For example, $\down{1}\up{2} \sep 5 \sep 3\up{7} \sep \updown{4}\down{6} \equiv^\star \down{1}\up{2}5 \sep 3\up{7} \sep \updown{4}\down{6} \equiv^\star \down{1}\up{2}5 \sep \up{7} \sep 3 \sep \updown{4}\down{6} \not\equiv^\star \down{1}\up{2}5 \sep \up{7} \sep \updown{4}\down{6} \sep 3$.

\begin{proposition}
Two decorated ordered partitions~$\lambda, \lambda' \in \fP_{\Decorations}$ are Schr\"oder permutree congruent if and only if they have the same $\PSymbol^\star$-symbol:
\[
\lambda \equiv^\star \lambda' \iff \SchroderPSymbol(\lambda) = \SchroderPSymbol(\lambda').
\]
\end{proposition}

\begin{proof}
It boils down to observe that two consecutive parts~$\b{a}$ and~$\b{c}$ of an ordered partition~$U \sep \b{a} \sep \b{c} \sep V$ in a fiber~$(\SchroderPSymbol)^{-1}(\tree[S])$ can be merged to~$U \sep \b{a}\b{c} \sep V$ and even exchanged to~$U \sep \b{c} \sep \b{a} \sep V$ while staying in~$(\SchroderPSymbol)^{-1}(\tree[S])$ precisely when they belong to distinct subtrees of a node of~$\tree[S]$. They are therefore separated by the vertical wall above (resp.~below) a value~$b$ with~$\b{a} \ll b \ll \b{c}$ and such that~${\decoration_b \in \{\upCirc, \upDownCirc\}}$ and~$b \in \bigcup U$ (resp.~$\decoration_b  \in \{\downCirc, \upDownCirc\}$ and~$b \in \bigcup V$).
\end{proof}

We now see the Schr\"oder permutree congruence as a lattice congruence. We first need to remember the facial weak order on ordered partitions as defined by D.~Krob, M.~Latapy, \mbox{J.-C.~Novelli}, H.~D.~Phan and~S.~Schwer in~\cite{KrobLatapyNovelliPhanSchwer}. This order extends the classical weak order on permutations of~$[n]$. This order was also extended to faces of permutahedra of arbitrary finite Coxeter groups by P.~Palacios and M.~Ronco~\cite{PalaciosRonco} and studied in detail by A.~Dermenjian, C.~Hohlweg and V.~Pilaud~\cite{DermenjianHohlwegPilaud}.

\begin{definition}
The \defn{coinversion map}~${\coinv(\lambda) : \binom{[n]}{2} \to \{-1,0,1\}}$ of an ordered partition~$\lambda \in \fP_n$ is the map defined for~$i < j$ by~$\coinv(\lambda)(i,j) = \sign(\lambda^{-1}(i) - \lambda^{-1}(j))$. The \defn{facial weak order}~$\le$ is the poset on ordered partitions defined by~$\lambda \le \lambda'$ if~$\coinv(\lambda)(i,j) \le \coinv(\lambda')(i,j)$ for all~$i < j$.
\end{definition}

The following properties of the facial weak order were proved in~\cite{KrobLatapyNovelliPhanSchwer} and extended for arbitrary finite Coxeter groups in~\cite{DermenjianHohlwegPilaud}.

\begin{proposition}[\cite{KrobLatapyNovelliPhanSchwer}]
The facial weak order on the set of ordered partitions is a lattice.
\end{proposition}

\begin{proposition}[\cite{KrobLatapyNovelliPhanSchwer}]
The cover relations of the weak order~$<$ on~$\fP_n$ are given by
\begin{gather*}
\lambda_1 \sep \cdots \sep \lambda_i \sep \lambda_{i+1} \sep \cdots \sep \lambda_k \;\; < \;\; \lambda_1 \sep \cdots \sep \lambda_i\lambda_{i+1} \sep \cdots \sep \lambda_k \qquad\text{if } \lambda_i \ll \lambda_{i+1}, \\
\lambda_1 \sep \cdots \sep \lambda_i\lambda_{i+1} \sep \cdots \sep \lambda_k \;\; < \;\; \lambda_1 \sep \cdots \sep \lambda_i \sep \lambda_{i+1} \sep \cdots \sep \lambda_k \qquad\text{if } \lambda_{i+1} \ll \lambda_i.
\end{gather*}
\end{proposition}

We now use the facial weak order to understand better the Schr\"oder permutree congruence.

\begin{proposition}
\label{prop:facialWeakOrderCongruence}
For any decoration~$\decoration \in \Decorations^n$, the Schr\"oder $\decoration$-permutree congruence~$\equiv^\star_\decoration$ is a lattice congruence of the facial weak order on~$\fP^\decoration$.
\end{proposition}

\begin{proof}
We could write a direct proof that the classes of the Schr\"oder $\decoration$-permutree congruence are intervals and that the up and down projection maps are order preserving. This was done for example in~\cite[Proposition~105]{ChatelPilaud}. To avoid this tedious proof, we prefer to refer the reader to~\cite{DermenjianHohlwegPilaud} which states that a lattice congruence of the weak order automatically transposes to a lattice congruence of the facial weak order.
\end{proof}

\begin{corollary}
\label{coro:patternAvoidingPartitions}
The Schr\"oder $\decoration$-permutree congruence classes are intervals of the facial weak order on~$\fP_n$. In particular, the following sets are in bijection:
\begin{enumerate}[(i)]
\item Schr\"oder permutrees with decoration~$\decoration$,
\item Schr\"oder $\decoration$-permutree congruence classes,
\item partitions of~$\fP_n$ avoiding the patterns~$\b{a} \sep \b{c} \dash b$ and~$\b{ac} \dash b$ for~$\decoration_b \in \{\downCirc{}, \upDownCirc{}\}$ and~$b \dash \b{a} \sep \b{c}$ and~$b \dash \b{ac}$ for~$\decoration_b \in \{\upCirc{}, \upDownCirc{}\}$, where~$\b{a} \ll b \ll \b{c}$,
\item partitions of~$\fP_n$ avoiding the patterns~$\b{c} \sep \b{a} \dash b$ and~$\b{ac} \dash b$ for~$\decoration_b \in \{\downCirc{}, \upDownCirc{}\}$ and~$b \dash \b{c} \sep \b{a}$ and~$b \dash \b{a}\b{c}$ for~$\decoration_b \in \{\upCirc{}, \upDownCirc{}\}$, where~$\b{a} \ll b \ll \b{c}$.
\end{enumerate}
\end{corollary}

It follows that the Schr\"oder permutrees inherit a lattice structure defined by~$\tree[S] < \tree[S]'$ if and only if there exist ordered partitions~$\lambda, \lambda'$ such that~$\SchroderPSymbol(\lambda) = \tree[S]$, $\SchroderPSymbol(\lambda') = \tree[S]'$ and~$\lambda < \lambda'$ in facial weak order. We now provide another interpretation of this lattice.

\begin{definition}
We say that the contraction of an edge~$e = v \to w$ in a Schr\"oder permutree~$\tree[S]$ is \defn{increasing} if~$p(v) \ll p(w)$ and \defn{decreasing} if~$p(w) \ll p(v)$. The \defn{Schr\"oder $\decoration$-permutree lattice} is the transitive closure of the relations~$\tree[S] < \tree[S]/e$ (resp.~$\tree[S]/e < \tree[S]$) for any Schr\"oder $\decoration$-permutree~$\tree[S]$ and any edge~$e \in \tree[S]$ defining an increasing (resp.~decreasing) contraction.
\end{definition}

\begin{proposition}
The map~$\SchroderPSymbol$ defines a lattice homomorphism from the facial weak order on~$\fP_\decoration$ to the Schr\"oder $\decoration$-permutree lattice. In other words, the Schr\"oder $\decoration$-permutree lattice is isomorphic to the lattice quotient of the facial weak order by the Schr\"oder $\decoration$-permutree congruence.
\end{proposition}

\begin{proof}
Let~$\lambda < \lambda'$ be a cover relation in the weak order on~$\fP_\decoration$. Assume that~$\lambda'$ is obtained by merging the parts~$\lambda_i \ll \lambda_{i+1}$ of~$\lambda$ (the other case being symmetric). Let~$u$ denote the rightmost node of~$\SchroderPSymbol(\lambda)$ at level~$i$, and~$v$ the leftmost node of~$\SchroderPSymbol(\lambda)$ at level~$i+1$. If~$u$ and~$v$ are not comparable, then~$\SchroderPSymbol(\lambda) = \SchroderPSymbol(\lambda')$. Otherwise, there is an edge~$u \to v$ in~$\SchroderPSymbol(\lambda)$ and~$\SchroderPSymbol(\lambda')$ is obtained by the increasing contraction of~$u \to v$ in~$\SchroderPSymbol(\lambda)$.
\end{proof}

\enlargethispage{.6cm}
Examples of Schr\"oder $\decoration$-permutree lattices for $\decoration = \noneCirc{}^3$, $\downCirc{}^3$ and~$\upDownCirc{}^3$ are illustrated in \fref{fig:SchroderPermutreeLattices}.

\begin{figure}[h]
  \centerline{\includegraphics[scale=1]{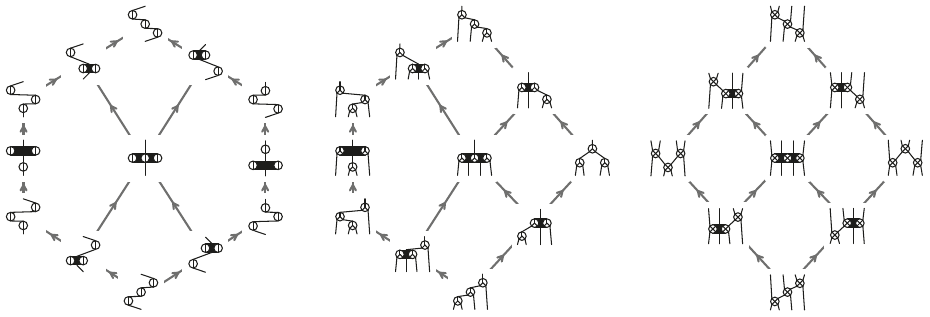}}
  \caption{The Schr\"oder $\decoration$-permutree lattices for $\decoration = \noneCirc{}^3$ (left), $\downCirc{}^3$ (middle) and~$\upDownCirc{}^3$~(right).}
  \label{fig:SchroderPermutreeLattices}
\end{figure}


\subsection{Numerology}
\label{subsec:numerologySchroder}

\enlargethispage{-.4cm}
According to Corollary~\ref{coro:patternAvoidingPartitions}, Schr\"oder $\decoration$-permutrees are in bijection with ordered partitions of~$\fP_n$ avoiding the patterns~$\b{a}\sep\b{c} \dash b$ and~$\b{a}\b{c} \dash b$ for~$\decoration_b \in \{\downCirc{}, \upDownCirc{}\}$ and~$b \dash \b{a}\sep\b{c}$ and~$b \dash \b{a}\b{c}$ for~$\decoration_b \in \{\upCirc{}, \upDownCirc{}\}$. Similar to Section~\ref{subsec:numerology}, we construct a generating tree~$\generatingTreeSchroder_{\decoration}$ for these ordered partitions. This tree has~$n$ levels, and the nodes at level~$\level$ are labeled by the ordered partitions of~$[\level]$ whose values are decorated by the restriction of~$\decoration$ to~$[\level]$ and avoiding the four forbidden patterns. The parent of an ordered partition in~$\generatingTreeSchroder_\decoration$ is obtained by deleting its maximal value. See~\fref{fig:generatingTreeSchroder} for examples of such trees. 

\begin{figure}[t]
  \centerline{\input{generatingTreeSchroder}}
  \caption{The generating trees~$\generatingTreeSchroder_\decoration$ for the decorations~$\decoration = \downCirc{}\upCirc{}\upDownCirc{}$ (top) and ${\decoration = \downCirc{}\noneCirc{}\upCirc{}}$ (bottom). Free gaps are marked with blue dots.}
  \label{fig:generatingTreeSchroder}
\end{figure}

As in Section~\ref{subsec:numerology}, we consider the possible positions of~$\level+1$ in the children of an ordered partition~$\lambda$ at level~$\level$ in~$\generatingTreeSchroder_\decoration$. We call \defn{free gaps} the positions where placing~$\level+1$ does not create a forbidden pattern. They are marked with a blue point~$\freeGap$ in \fref{fig:generatingTreeSchroder}. Note that free gaps can appear at the end of a part of~$\lambda$, on a separator~$\sep$ in between two parts of~$\lambda$, or at the beginning or end of~$\lambda$. We therefore include two fake separators at the beginning and at the end of~$\lambda$. Except the first free gap at the beginning of~$\lambda$, all free gaps come by pairs of the form~$\freeGap\freeSep$: if~$\level+1$ can be inserted at the end of a part, it can as well be inserted on the next separator. Therefore, any ordered partition has an odd number of free gaps. Our main tool is the following lemma. Its proof, similar to that of Lemma~\ref{lem:GeneratingTree}, is left to the reader.

\begin{lemma}
\label{lem:GeneratingTreeSchroder}
Any ordered partition at level~$\level$ with~$2g+1$ free gaps and $s$ internal separators has
\begin{itemize}
\item $g+1$ children with~$2g+3$ free gaps and~$s+1$ internal separators, and~$g$ children with~$2g+1$ free gaps and~$s$ separators when~$\decoration_{\level+1} = \noneCirc{}$,
\item one child with~$2g'+1$ free gaps and $s$ internal separator for each~$g' \in [g]$, and one child with~$2g'+1$ free gaps and~$s+1$ internal separators for each~$g' \in [g+1]$ when~$\decoration_{\level+1} \in \{\downCirc{},\upCirc{}\}$,
\item $g+1$ children with~$3$ free gaps and~$s+1$ internal separators, and~$g$ children with~$3$ free gaps and~$s$ separators when~$\decoration_{\level+1} = \upDownCirc{}$.
\end{itemize}
\end{lemma}

Ordering the children of a node of~$\generatingTreeSchroder_\decoration$ in increasing number of free gaps and increasing number of separators, we obtain the following statement similar to Proposition~\ref{prop:generatingTree}.

\begin{proposition}
\label{prop:generatingTreeSchroder}
For any decorations~$\decoration, \decoration' \in \Decorations^n$ such that~$\decoration^{-1}(\noneCirc{}) = \decoration'\,\!^{-1}(\noneCirc{})$ and ${\decoration^{-1}(\upDownCirc{}) = \decoration'\,\!^{-1}(\upDownCirc{})}$, the generating trees~$\generatingTreeSchroder_\decoration$ and~$\generatingTreeSchroder_{\decoration'}$ are isomorphic.
\end{proposition}

Finally, similar to Corollary~\ref{coro:inductionFactCatalan}, we obtain the following inductive formulas for the number of Schr\"oder permutrees.

\begin{corollary}
Let~$\decoration \in \Decorations^n$ and~$\decoration'$ be obtained by deleting the last letter~$\decoration_n$ of~$\decoration$. The number~$\factSchroder{\decoration,g,s}$ of ordered partitions avoiding~$\b{a} \sep \b{c} \dash b$ and~$\b{a}\b{c} \dash b$ for~$\decoration_b \in \{\downCirc{}, \upDownCirc{}\}$ and~$b \dash \b{a} \sep \b{c}$ and~$b \dash \b{a}\b{c}$ for~$\decoration_b \in \{\upCirc{}, \upDownCirc{}\}$ and with $2g+1$ free gaps and $s$ internal separators satisfies the following recurrence relations:
\[
\factSchroder{\decoration,g,s} = 
\left\{
\begin{array}{l@{\cdot \bigg(}c@{\;+\;}c@{\bigg)\quad}l}
\one_{\substack{g \ge 1 \quad \\ s \ge g-1}} & g \cdot \factSchroder{\decoration',g,s} & g \cdot \factSchroder{\decoration',g-1,s-1} & \text{if } \decoration_n = \noneCirc{}, \\[.3cm]
\displaystyle \one_{\substack{g \ge 1 \quad \\ s \ge g-1}} & \displaystyle \sum_{g' \ge g} \factSchroder{\decoration',g',s} & \displaystyle \sum_{g' \ge g-1} \factSchroder{\decoration',g',s-1} & \text{if } \decoration_n = \upCirc{} \text{ or } \downCirc{}, \\[.3cm]
\displaystyle \one_{g = 1} & \displaystyle \sum_{g' \ge 1} g' \cdot \factSchroder{\decoration',g',s} & \displaystyle \sum_{g' \ge 1} (g'+1) \cdot \factSchroder{\decoration',g',s-1} & \text{if } \decoration_n = \upDownCirc{},
\end{array}
\right.
\]
where~$\one_X$ is $1$ if~$X$ is satisfied and~$0$ otherwise.
\end{corollary}

Note that for any~$\decoration \in \Decorations^n$, the $f$-vector of the permutreehedron~$\Permutreehedron$ is given by
\[
f_k(\Permutreehedron) = \sum_{g \in \N} \factSchroder{\decoration,g,n-1-k}.
\]
For example the $f$-vector of~$\Permutreehedron[\downCirc{}\upCirc{}\noneCirc{}\upDownCirc{}\noneCirc{}\downCirc{}\upCirc{}]$ is~$[1, 324, 972, 1125, 630, 175, 22, 1]$.

\begin{corollary}
\label{coro:equienumeratedSchroder}
The $f$-vector of the permutreehedron~$\Permutreehedron$ only depends on the positions of the \noneCirc{} and \upDownCirc{} in~$\decoration$.
\end{corollary}

\begin{example}
Following Example~\ref{exm:permutreesSpecificDecorations}, observe that the $f$-vector of the permutreehedron~$\Permutreehedron$ specializes to the following well-known sequences of numbers:
\begin{enumerate}[(i)]
\item when~$\decoration = \noneCirc{}^n$,
\[
f_k(\Permutreehedron[\noneCirc^n]) = \sum_{j = 0}^{n-k} (-1)^j \binom{n-k}{j} (n-k-j)^n
\]
is the number of ordered partitions of~$[n]$ into~$n-k$ parts~\href{https://oeis.org/A019538}{\cite[A019538]{OEIS}} and~$\sum_k f_k(\Permutreehedron[\noneCirc^n])$ is a Fubini number \href{https://oeis.org/A000670}{\cite[A000670]{OEIS}},
\item when~$\decoration \in \{\upCirc{}, \downCirc{}\}^n$, 
\[
\displaystyle f(\Permutreehedron) = \frac{1}{n-k} \binom{n-1}{n-k-1}\binom{2n-k}{n-k-1}
\]
is the number of dissections of a convex $(n+2)$-gon with $n-k-1$ non-crossing diagonals \href{https://oeis.org/A033282}{\cite[A033282]{OEIS}} and~$\sum_k f_k(\Permutreehedron)$ is a Schr\"oder number \href{https://oeis.org/A001003}{\cite[A001003]{OEIS}},
\item when~$\decoration = \upDownCirc{}^n$,
\[
f_k(\Permutreehedron[\noneCirc^n]) = 2^{n-1-k}\binom{n-1}{k}
\]
is the number of words of length~$n-1$ on~$\{-1,0,1\}$ with~$k$ occurrences of~$0$ \href{https://oeis.org/A013609}{\cite[A000244]{OEIS}} and~$\sum_k f_k(\Permutreehedron[\upDownCirc{}^n]) = 3^{n-1}$ \href{https://oeis.org/A000244}{\cite[A000670]{OEIS}}.
\end{enumerate}
\end{example}

Note that for any~$\decoration \in \Decorations^n$, the $k$th entry~$h_k(\Permutreehedron)$ of the $h$-vector of the permutreehedron~$\Permutreehedron$ is the number of $\decoration$-permutrees with $k$ increasing edges (\ie edges~$i \to j$ with~$i < j$). Since the $f$-vector only depends on the positions of the \noneCirc{} and \upDownCirc{} in~$\decoration$, so does the $h$-vector. We therefore obtain the following statement.

\begin{corollary}
The number of $\decoration$-permutrees with $k$ increasing edges (\ie edges~$i \to j$ with~$i < j$) only depends on the positions of the \noneCirc{} and \upDownCirc{} in~$\decoration$.
\end{corollary}

\begin{example}
Following Example~\ref{exm:permutreesSpecificDecorations}, observe that the $h$-vector of the permutreehedron~$\Permutreehedron$ specializes to the following well-known sequences of numbers:
\begin{enumerate}[(i)]
\item the Eulerian numbers \href{https://oeis.org/A008292}{\cite[A008292]{OEIS}} when~$\decoration = \noneCirc{}^n$,
\item the Narayana numbers \href{https://oeis.org/A001263}{\cite[A001263]{OEIS}} when~$\decoration \in \{\upCirc{}, \downCirc{}\}^n$,
\item the binomial coefficients \href{https://oeis.org/A007318}{\cite[A007318]{OEIS}} when~$\decoration = \upDownCirc^n$.
\end{enumerate}
\end{example}



\subsection{Refinement}
\label{subsec:refinementSchroder}

Similar to Section~\ref{subsec:decorationRefinements}, consider two decorations~$\decoration, \decoration' \in \Decorations^n$ such that~$\decoration \less \decoration'$ (that is, such that~$\decoration_i \less \decoration'_i$ for all~$i \in [n]$ where the order on~$\Decorations$ is given by~$\noneCirc{} \less \{\downCirc{}, \upCirc{}\} \less \upDownCirc{}$).

The Schr\"oder $\delta$-permutree congruence then refines the Schr\"oder $\delta'$-permutree congruence. Therefore, we obtain a natural surjection~$\surjectionSchroder : \SchrPermutrees(\decoration) \to \SchrPermutrees(\decoration')$: for any ordered partition~$\lambda$ of~$[n]$, $\surjectionSchroder$ sends the Schr\"oder permutree obtained by insertion of~$\lambda$ decorated by~$\delta$ to the Schr\"oder permutree obtained by insertion of~$\lambda$ decorated by~$\delta'$. This surjection can as well be interpreted directly on our representation of the Schr\"oder permutrees as in \fref{fig:refinement}, where now both the blocks and the edges can be refined.

Finally, similar to Proposition~\ref{prop:refinementLatticeHomomorphism}, we obtain the following statement.

\begin{proposition}
\label{prop:refinementLatticeHomomorphismSchroder}
The surjection~$\surjectionSchroder$ defines a lattice homomorphism from the Schr\"oder $\decoration$-permutree lattice to the Schr\"oder $\decoration'$-permutree lattice.
\end{proposition}


\subsection{Schr\"oder permutree algebra}
\label{subsec:permutreeAlgebraSchroder}

To conclude this section on Schr\"oder permutrees, we briefly mention their Hopf algebra structure. Following~\cite{ChatelPilaud}, we first reformulate in terms of ordered partitions the Hopf algebra of F.~Chapoton~\cite{Chapoton} indexed by the faces of the permutahedra.

We define two restrictions on ordered partitions. Consider an ordered partition~$\mu$ of~$[n]$ into~$p$ parts. For~${I \subseteq [p]}$, we let~$n_I \eqdef |\set{j \in [n]}{\exists \, i \in I, \; j \in \mu_i}|$ and we denote by~$\mu_{|I}$ the ordered partition of~$[n_I]$ into $|I|$ parts obtained from~$\mu$ by deletion of the parts indexed by~$[p] \ssm I$ and standardization. Similarly, for~$J \subseteq [n]$, we let~$p_J \eqdef |\set{i \in [p]}{\exists \, j \in J, \; j \in \mu_i}|$ and we denote by~$\mu^{|J}$ the ordered partition of~$[|J|]$ into~$p_J$ parts obtained from~$\mu$ by deletion of the entries in~$[n] \ssm J$ and standardization. For example, for the ordered partition~$\mu = 16 \sep 27 \sep 4 \sep 35$ we have~$\mu_{|\{2,3\}} = 13 \sep 2$ and~$\mu^{|\{1,3,5\}} = 1 \sep 23$.

The \defn{shifted shuffle}~$\lambda \shiftedShuffle \lambda'$ and the \defn{convolution}~$\lambda \convolution \lambda'$ of two ordered partitions~${\lambda \in \fP_n}$ and~$\lambda' \in \fP_{n'}$ are defined by:
\begin{align*}
\lambda \shiftedShuffle \lambda' & \eqdef \bigset{\mu \in \fP_{n+n'}}{\mu^{|\{1, \dots, n\}} = \lambda \text{ and } \mu^{|\{n+1, \dots, n+n'\}} = \lambda'}, \\
\text{and}\qquad \lambda \convolution \lambda' & \eqdef \bigset{\mu \in \fP_{n+n'}}{\mu_{|\{1, \dots, k\}} = \lambda \text{ and } \mu_{|\{k+1, \dots, k+k'\}} = \lambda'}.
\end{align*}

For example,
\begin{align*}
{\red 1} \sep {\red 2} \shiftedShuffle {\blue 2} \sep {\blue 13} = \{ &
	{\red 1} \sep {\red 2} \sep {\blue 4} \sep {\blue 35}, \;
	{\red 1} \sep {\red 2}{\blue 4} \sep {\blue 35}, \;
	{\red 1} \sep {\blue 4} \sep {\red 2} \sep {\blue 35}, \;
	{\red 1} \sep {\blue 4} \sep {\red 2}{\blue 35}, \;
	{\red 1} \sep {\blue 4} \sep {\blue 35} \sep {\red 2}, \;
	{\red 1}{\blue 4} \sep {\red 2} \sep {\blue 35}, \;
	{\red 1}{\blue 4} \sep {\red 2}{\blue 35}, \\[-.1cm]
&	{\red 1}{\blue 4} \sep {\blue 35} \sep {\red 2}, \;
	{\blue 4} \sep {\red 1} \sep {\red 2} \sep {\blue 35}, \;
	{\blue 4} \sep {\red 1} \sep {\red 2}{\blue 35}, \;
	{\blue 4} \sep {\red 1} \sep {\blue 35} \sep {\red 2}, \;
	{\blue 4} \sep {\red 1}{\blue 35} \sep {\red 2}, \;
	{\blue 4} \sep {\blue 35} \sep {\red 1} \sep {\red 2} \}, \\[.2cm]
{\red 1} \sep {\red 2} \convolution {\blue 2} \sep {\blue 13} = \{ &
	{\red 1} \sep {\red 2} \sep {\blue 4} \sep {\blue 35}, \;
	{\red 1} \sep {\red 3} \sep {\blue 4} \sep {\blue 25}, \;
	{\red 1} \sep {\red 4} \sep {\blue 3} \sep {\blue 25}, \;
	{\red 1} \sep {\red 5} \sep {\blue 3} \sep {\blue 24}, \;
	{\red 2} \sep {\red 3} \sep {\blue 4} \sep {\blue 15}, \\[-.1cm]
&	{\red 2} \sep {\red 4} \sep {\blue 3} \sep {\blue 15}, \;
	{\red 2} \sep {\red 5} \sep {\blue 3} \sep {\blue 14}, \;
	{\red 3} \sep {\red 4} \sep {\blue 2} \sep {\blue 15}, \;
	{\red 3} \sep {\red 5} \sep {\blue 2} \sep {\blue 14}, \;
	{\red 4} \sep {\red 5} \sep {\blue 2} \sep {\blue 13} \}.
\end{align*}

These definitions extend to decorated ordered partitions: decorations travel with their values in the shifted shuffle product, and stay at their positions in the convolution product.

We denote by~$\OrdPart_{\Decorations}$ the Hopf algebra with basis~$(\F_\lambda)_{\lambda \in \fP_{\Decorations}}$ and whose product and coproduct are defined by
\[
\F_\lambda \product \F_{\lambda'} = \sum_{\mu \in \lambda \shiftedShuffle \lambda'} \F_\mu
\qquad\text{and}\qquad
\coproduct \F_\mu = \sum_{\mu \in \lambda \convolution \lambda'} \F_\lambda \otimes \F_{\lambda'}.
\]
This indeed defines a Hopf algebra, which is just a decorated version of that~\cite{Chapoton}.

Finally, we denote by~$\SchrPermutreeAlgebra$ the vector subspace of~$\OrdPart_{\Decorations}$ generated by
\[
\PPT_{\tree[S]} \eqdef \sum_{\substack{\lambda \in \fP_{\Decorations} \\ \SchroderPSymbol(\lambda) = \tree[S]}} \F_\lambda,
\]
for all Schr\"oder permutree~$\tree[S]$. We skip the proof of the following statement.

\begin{theorem}
$\SchrPermutreeAlgebra$ is a Hopf subalgebra of~$\OrdPart_{\Decorations}$.
\end{theorem}

\enlargethispage{.4cm}
Relevant subalgebras of the Schr\"oder permutree algebra are the Hopf algebras of F.~Chapoton~\cite{Chapoton} on the faces of the permutahedra, associahedra and cubes, as well as the Schr\"oder Cambrian Hopf algebra of G.~Chatel and V.~Pilaud~\cite[Part~3]{ChatelPilaud}. As in Section~\ref{sec:permutreeHopfAlgebra}, we invite the reader to work out direct combinatorial rules for the product and coproduct in the Schr\"oder permutree algebra. Examples can be found in~\cite[Part~3]{ChatelPilaud}.


%


\section*{Acknowledgments}

We thank two anonymous referees for helpful comments and suggestions on this paper.
We thank Nathan Reading for relevant comments on this paper.
The computation and tests needed along the research were done using the open-source mathematical software \texttt{Sage}~\cite{Sage} and its combinatorics features developed by the \texttt{Sage-combinat} community~\cite{SageCombinat}.


\bibliographystyle{alpha}
\bibliography{permutrees}
\label{sec:biblio}

\end{document}

%% file: figures/generatingTree.tex
\begin{tikzpicture}
  \node(T1) at (0,0) {
    \begin{tikzpicture}[xscale=1.05, yscale=1.1]
      \node(P2341) at (0,0) {$\freeGap\up{2}\bannedGap{3}\bannedGap\updown{4}\freeGap\down{1}\bannedGap$};
      \node(P2431) at (1,0) {$\freeGap\up{2}\bannedGap\updown{4}\freeGap{3}\bannedGap\down{1}\bannedGap$};
      \node(P4231) at (2,0) {$\freeGap\updown{4}\freeGap\up{2}\bannedGap{3}\bannedGap\down{1}\bannedGap$};
      \node(P3241) at (3,0) {$\freeGap{3}\bannedGap\up{2}\bannedGap\updown{4}\freeGap\down{1}\bannedGap$};
      \node(P3421) at (4,0) {$\freeGap{3}\bannedGap\updown{4}\freeGap\up{2}\bannedGap\down{1}\bannedGap$};
      \node(P4321) at (5,0) {$\freeGap\updown{4}\freeGap{3}\bannedGap\up{2}\bannedGap\down{1}\bannedGap$};
      \node(P1234) at (6,0) {$\freeGap\down{1}\bannedGap\up{2}\bannedGap{3}\bannedGap\updown{4}\freeGap$};
      \node(P1243) at (7,0) {$\freeGap\down{1}\bannedGap\up{2}\bannedGap\updown{4}\freeGap{3}\bannedGap$};
      \node(P1423) at (8,0) {$\freeGap\down{1}\bannedGap\updown{4}\freeGap\up{2}\bannedGap{3}\bannedGap$};
      \node(P4123) at (9,0) {$\freeGap\updown{4}\freeGap\down{1}\bannedGap\up{2}\bannedGap{3}\bannedGap$};
      \node(P1324) at (10,0) {$\freeGap\down{1}\bannedGap{3}\bannedGap\up{2}\bannedGap\updown{4}\freeGap$};
      \node(P1342) at (11,0) {$\freeGap\down{1}\bannedGap{3}\bannedGap\updown{4}\freeGap\up{2}\bannedGap$};
      \node(P1432) at (12,0) {$\freeGap\down{1}\bannedGap\updown{4}\freeGap{3}\bannedGap\up{2}\bannedGap$};
      \node(P4132) at (13,0) {$\freeGap\updown{4}\freeGap\down{1}\bannedGap{3}\bannedGap\up{2}\bannedGap$};
      \node(P3124) at (14,0) {$\freeGap{3}\bannedGap\down{1}\bannedGap\up{2}\bannedGap\updown{4}\freeGap$};
      \node(P3142) at (15,0) {$\freeGap{3}\bannedGap\down{1}\bannedGap\updown{4}\freeGap\up{2}\bannedGap$};
      \node(P3412) at (16,0) {$\freeGap{3}\bannedGap\updown{4}\freeGap\down{1}\bannedGap\up{2}\bannedGap$};
      \node(P4312) at (17,0) {$\freeGap\updown{4}\freeGap{3}\bannedGap\down{1}\bannedGap\up{2}\bannedGap$};
      \node(P231) at (1,1) {$\freeGap\up{2}\freeGap{3}\freeGap\down{1}\bannedGap$};
      \node(P321) at (4,1) {$\freeGap{3}\freeGap\up{2}\freeGap\down{1}\bannedGap$};
      \node(P123) at (7.5,1) {$\freeGap\down{1}\freeGap\up{2}\freeGap{3}\freeGap$};
      \node(P132) at (11.5,1) {$\freeGap\down{1}\freeGap{3}\freeGap\up{2}\freeGap$};
      \node(P312) at (15.5,1) {$\freeGap{3}\freeGap\down{1}\freeGap\up{2}\freeGap$};
      \node(P21) at (2.5,2) {$\freeGap\up{2}\freeGap\down{1}\bannedGap$};
      \node(P12) at (11.5,2) {$\freeGap\down{1}\freeGap\up{2}\freeGap$};
      \node(P1) at (7,3) {$\freeGap\down{1}\freeGap$};

      \draw (P21) -- (P1);
      \draw (P12) -- (P1);

      \draw (P123) -- (P12);
      \draw (P132) -- (P12);
      \draw (P312) -- (P12);

      \draw (P231) -- (P21);
      \draw (P321) -- (P21);

      \draw (P3124) -- (P312);
      \draw (P3142) -- (P312);
      \draw (P3412) -- (P312);
      \draw (P4312) -- (P312);

      \draw (P1324) -- (P132);
      \draw (P1342) -- (P132);
      \draw (P1432) -- (P132);
      \draw (P4132) -- (P132);

      \draw (P1234) -- (P123);
      \draw (P1243) -- (P123);
      \draw (P1423) -- (P123);
      \draw (P4123) -- (P123);

      \draw (P3241) -- (P321);
      \draw (P3421) -- (P321);
      \draw (P4321) -- (P321);

      \draw (P2341) -- (P231);
      \draw (P2431) -- (P231);
      \draw (P4231) -- (P231);
    \end{tikzpicture}
  };

  \node(T2) at (0, -4.5) {
    \begin{tikzpicture}[xscale=1.5, yscale=1.1]
      \node(P4123) at (0,0) {$\freeGap\up{4}\freeGap\down{1}\bannedGap{2}\bannedGap\updown{3}\bannedGap$};
      \node(P1234) at (1,0) {$\freeGap\down{1}\bannedGap{2}\bannedGap\updown{3}\freeGap\up{4}\freeGap$};
      \node(P4132) at (2,0) {$\freeGap\up{4}\freeGap\down{1}\bannedGap\updown{3}\bannedGap{2}\bannedGap$};
      \node(P1342) at (3,0) {$\freeGap\down{1}\bannedGap\updown{3}\freeGap\up{4}\freeGap{2}\bannedGap$};
      \node(P4312) at (4,0) {$\freeGap\up{4}\freeGap\updown{3}\bannedGap\down{1}\bannedGap{2}\bannedGap$};
      \node(P3412) at (5,0) {$\freeGap\updown{3}\freeGap\up{4}\freeGap\down{1}\bannedGap{2}\bannedGap$};
      \node(P4213) at (6,0) {$\freeGap\up{4}\freeGap{2}\bannedGap\down{1}\bannedGap\updown{3}\bannedGap$};
      \node(P2134) at (7,0) {$\freeGap{2}\bannedGap\down{1}\bannedGap\updown{3}\freeGap\up{4}\freeGap$};
      \node(P4231) at (8,0) {$\freeGap\up{4}\freeGap{2}\bannedGap\updown{3}\bannedGap\down{1}\bannedGap$};
      \node(P2341) at (9,0) {$\freeGap{2}\bannedGap\updown{3}\freeGap\up{4}\freeGap\down{1}\bannedGap$};
      \node(P4321) at (10,0) {$\freeGap\up{4}\freeGap\updown{3}\bannedGap{2}\bannedGap\down{1}\bannedGap$};
      \node(P3421) at (11,0) {$\freeGap\updown{3}\freeGap\up{4}\freeGap{2}\bannedGap\down{1}\bannedGap$};
      \node(P123) at (0.5,1) {$\freeGap\down{1}\bannedGap{2}\bannedGap\updown{3}\freeGap$};
      \node(P132) at (2.5,1) {$\freeGap\down{1}\bannedGap\updown{3}\freeGap{2}\bannedGap$};
      \node(P312) at (4.5,1) {$\freeGap\updown{3}\freeGap\down{1}\bannedGap{2}\bannedGap$};
      \node(P213) at (6.5,1) {$\freeGap{2}\bannedGap\down{1}\bannedGap\updown{3}\freeGap$};
      \node(P231) at (8.5,1) {$\freeGap{2}\bannedGap\updown{3}\freeGap\down{1}\bannedGap$};
      \node(P321) at (10.5,1) {$\freeGap\updown{3}\freeGap{2}\bannedGap\down{1}\bannedGap$};
      \node(P12) at (2.5,2) {$\freeGap\down{1}\freeGap{2}\freeGap$};
      \node(P21) at (8.5,2) {$\freeGap{2}\freeGap\down{1}\freeGap$};
      \node(P1) at (5.5,3) {$\freeGap\down{1}\freeGap$};

      \draw (P4321) -- (P321);
      \draw (P3421) -- (P321);

      \draw (P4231) -- (P231);
      \draw (P2341) -- (P231);

      \draw (P4213) -- (P213);
      \draw (P2134) -- (P213);

      \draw (P4312) -- (P312);
      \draw (P3412) -- (P312);

      \draw (P4132) -- (P132);
      \draw (P1342) -- (P132);

      \draw (P4123) -- (P123);
      \draw (P1234) -- (P123);

      \draw (P213) -- (P21);
      \draw (P231) -- (P21);
      \draw (P321) -- (P21);

      \draw (P123) -- (P12);
      \draw (P132) -- (P12);
      \draw (P312) -- (P12);

      \draw (P12) -- (P1);
      \draw (P21) -- (P1);
    \end{tikzpicture}
  };
\end{tikzpicture}

%% file: figures/generatingTreeSchroder.tex
\begin{tikzpicture}
  \node(T1) at (0,0) {
    \begin{tikzpicture}[xscale=1.05, yscale=1.1]
      \node(P30201) at (1,1) {$\freeSep\updown{3}\freeGap\freeSep\up{2}\sep\down{1}\sep$};
      \node(P2301) at (2.5,1) {$\freeSep\up{2}\updown{3}\freeGap\freeSep\down{1}\sep$};
      \node(P20301) at (4,1) {$\freeSep\up{2}\sep\updown{3}\freeGap\freeSep\down{1}\sep$};
      \node(P3012) at (5.5,1) {$\freeSep\updown{3}\freeGap\freeSep\down{1}\up{2}\sep$};
      \node(P123) at (7,1) {$\freeSep\down{1}\up{2}\updown{3}\freeGap\freeSep$};
      \node(P1203) at (8.5,1) {$\freeSep\down{1}\up{2}\sep\updown{3}\freeGap\freeSep$};
      \node(P30102) at (10,1) {$\freeSep\updown{3}\freeGap\freeSep\down{1}\sep\up{2}\sep$};
      \node(P1302) at (11.5,1) {$\freeSep\down{1}\updown{3}\freeGap\freeSep\up{2}\sep$};
      \node(P10302) at (13,1) {$\freeSep\down{1}\sep\updown{3}\freeGap\freeSep\up{2}\sep$};
      \node(P1023) at (14.5,1) {$\freeSep\down{1}\sep\up{2}\updown{3}\freeGap\freeSep$};
      \node(P10203) at (16,1) {$\freeSep\down{1}\sep\up{2}\sep\updown{3}\freeGap\freeSep$};
      \node(P201) at (2.5,2) {$\freeSep\up{2}\freeGap\freeSep\down{1}\sep$};
      \node(P12) at (7,2) {$\freeSep\down{1}\up{2}\freeGap\freeSep$};
      \node(P102) at (13,2) {$\freeSep\down{1}\freeGap\freeSep\up{2}\freeGap\freeSep$};
      \node(P1) at (7,3) {$\freeSep\down{1}\freeGap\freeSep$};

      \draw (P201) -- (P1);
      \draw (P12) -- (P1);
      \draw (P102) -- (P1);

      \draw (P30201) -- (P201);
      \draw (P2301) -- (P201);
      \draw (P20301) -- (P201);

      \draw (P3012) -- (P12);
      \draw (P123) -- (P12);
      \draw (P1203) -- (P12);
      
      \draw (P30102) -- (P102);
      \draw (P1302) -- (P102);
      \draw (P10302) -- (P102);
      \draw (P1023) -- (P102);
      \draw (P10203) -- (P102);
    \end{tikzpicture}
  };

  \node(T2) at (0, -3.5) {
    \begin{tikzpicture}[xscale=.9, yscale=1.1]
      \node(P30201) at (1,1) {$\freeSep\up{3}\freeGap\freeSep{2}\sep\down{1}\sep$};
      \node(P2301) at (2.5,1) {$\freeSep{2}\up{3}\freeGap\freeSep\down{1}\sep$};
      \node(P20301) at (4,1) {$\freeSep{2}\freeGap\freeSep\up{3}\freeGap\freeSep\down{1}\sep$};
      \node(P2013) at (5.5,1) {$\freeSep{2}\freeGap\freeSep\down{1}\up{3}\freeGap\freeSep$};
      \node(P20103) at (7,1) {$\freeSep{2}\freeGap\freeSep\down{1}\freeGap\freeSep\up{3}\freeGap\freeSep$};
      \node(P3012) at (8.5,1) {$\freeSep\up{3}\freeGap\freeSep\down{1}{2}\sep$};
      \node(P123) at (10,1) {$\freeSep\down{1}{2}\up{3}\freeGap\freeSep$};
      \node(P1203) at (11.5,1) {$\freeSep\down{1}{2}\freeGap\freeSep\up{3}\freeGap\freeSep$};
      \node(P30102) at (13,1) {$\freeSep\up{3}\freeGap\freeSep\down{1}\sep{2}\sep$};
      \node(P1302) at (14.5,1) {$\freeSep\down{1}\up{3}\freeGap\freeSep{2}\sep$};
      \node(P10302) at (16,1) {$\freeSep\down{1}\freeGap\freeSep\up{3}\freeGap\freeSep{2}\sep$};
      \node(P1023) at (17.5,1) {$\freeSep\down{1}\freeGap\freeSep{2}\up{3}\freeGap\freeSep$};
      \node(P10203) at (19,1) {$\freeSep\down{1}\freeGap\freeSep{2}\freeGap\freeSep\up{3}\freeGap\freeSep$};
      \node(P201) at (4,2) {$\freeSep{2}\freeGap\freeSep\down{1}\freeGap\freeSep$};
      \node(P12) at (10,2) {$\freeSep\down{1}{2}\freeGap\freeSep$};
      \node(P102) at (16,2) {$\freeSep\down{1}\freeGap\freeSep{2}\freeGap\freeSep$};
      \node(P1) at (10,3) {$\freeSep\down{1}\freeGap\freeSep$};

      \draw (P201) -- (P1);
      \draw (P12) -- (P1);
      \draw (P102) -- (P1);

      \draw (P30201) -- (P201);
      \draw (P2301) -- (P201);
      \draw (P20301) -- (P201);
      \draw (P2013) -- (P201);
      \draw (P20103) -- (P201);

      \draw (P3012) -- (P12);
      \draw (P123) -- (P12);
      \draw (P1203) -- (P12);
      
      \draw (P30102) -- (P102);
      \draw (P1302) -- (P102);
      \draw (P10302) -- (P102);
      \draw (P1023) -- (P102);
      \draw (P10203) -- (P102);
    \end{tikzpicture}
  };
\end{tikzpicture}

%% file: permutrees.bbl
\newcommand{\etalchar}[1]{$^{#1}$}
\begin{thebibliography}{GKL{\etalchar{+}}95}

\bibitem[BW91]{BjornerWachs}
Anders Bj{\"o}rner and Michelle~L. Wachs.
\newblock Permutation statistics and linear extensions of posets.
\newblock {\em J. Combin. Theory Ser. A}, 58(1):85--114, 1991.

\bibitem[CD06]{CarrDevadoss}
Michael~P. Carr and Satyan~L. Devadoss.
\newblock Coxeter complexes and graph-associahedra.
\newblock {\em Topology Appl.}, 153(12):2155--2168, 2006.

\bibitem[Cha00]{Chapoton}
Fr{\'e}d{\'e}ric Chapoton.
\newblock Alg\`ebres de {H}opf des permutah\`edres, associah\`edres et
  hypercubes.
\newblock {\em Adv. Math.}, 150(2):264--275, 2000.

\bibitem[CP17]{ChatelPilaud}
Gr\'egory Chatel and Vincent Pilaud.
\newblock {C}ambrian {H}opf {Algebras}.
\newblock {\em Adv. Math.}, 311:598--633, 2017.

\bibitem[DHP18]{DermenjianHohlwegPilaud}
Aram Dermenjian, Christophe Hohlweg, and Vincent Pilaud.
\newblock The facial weak order and its lattice quotients.
\newblock {\em Trans. Amer. Math. Soc.}, 370(2):1469--1507, 2018.

\bibitem[DHT02]{DuchampHivertThibon}
G.~Duchamp, F.~Hivert, and J.-Y. Thibon.
\newblock Noncommutative symmetric functions. {VI}. {F}ree quasi-symmetric
  functions and related algebras.
\newblock {\em Internat. J. Algebra Comput.}, 12(5):671--717, 2002.

\bibitem[GKL{\etalchar{+}}95]{GelfandKrobLascouxLeclercRetakhThibon}
Israel~M. Gelfand, Daniel Krob, Alain Lascoux, Bernard Leclerc, Vladimir~S.
  Retakh, and Jean-Yves Thibon.
\newblock Noncommutative symmetric functions.
\newblock {\em Adv. Math.}, 112(2):218--348, 1995.

\bibitem[HL07]{HohlwegLange}
Christophe Hohlweg and Carsten Lange.
\newblock Realizations of the associahedron and cyclohedron.
\newblock {\em Discrete Comput.~Geom.}, 37(4):517--543, 2007.

\bibitem[HLT11]{HohlwegLangeThomas}
Christophe Hohlweg, Carsten Lange, and Hugh Thomas.
\newblock Permutahedra and generalized associahedra.
\newblock {\em Adv. Math.}, 226(1):608--640, 2011.

\bibitem[HNT05]{HivertNovelliThibon-algebraBinarySearchTrees}
Florent Hivert, Jean-Christophe Novelli, and Jean-Yves Thibon.
\newblock The algebra of binary search trees.
\newblock {\em Theoret. Comput. Sci.}, 339(1):129--165, 2005.

\bibitem[KLN{\etalchar{+}}01]{KrobLatapyNovelliPhanSchwer}
Daniel Krob, Matthieu Latapy, Jean-Christophe Novelli, Ha-Duong Phan, and
  Sylviane Schwer.
\newblock Pseudo-{P}ermutations {I}: {F}irst {C}ombinatorial and {L}attice
  {P}roperties.
\newblock 13th International Conference on Formal Power Series and Algebraic
  Combinatorics (FPSAC 2001), 2001.

\bibitem[KT97]{KrobThibon-NCSF4}
Daniel Krob and Jean-Yves Thibon.
\newblock Noncommutative symmetric functions. {IV}. {Q}uantum linear groups and
  {H}ecke algebras at {$q=0$}.
\newblock {\em J. Algebraic Combin.}, 6(4):339--376, 1997.

\bibitem[Lod01]{Loday-dialgebras}
Jean-Louis Loday.
\newblock Dialgebras.
\newblock In {\em Dialgebras and related operads}, volume 1763 of {\em Lecture
  Notes in Math.}, pages 7--66. Springer, Berlin, 2001.

\bibitem[Lod04]{Loday}
Jean-Louis Loday.
\newblock Realization of the {S}tasheff polytope.
\newblock {\em Arch.~Math.~(Basel)}, 83(3):267--278, 2004.

\bibitem[LP17]{LangePilaud}
Carsten Lange and Vincent Pilaud.
\newblock Associahedra via spines.
\newblock Preprint,
  \href{http://arxiv.org/abs/1307.4391}{\texttt{arXiv:1307.4391}}, to appear in
  \emph{Combinatorica}, 2017.

\bibitem[LR98]{LodayRonco}
Jean-Louis Loday and Mar{\'{\i}}a~O. Ronco.
\newblock Hopf algebra of the planar binary trees.
\newblock {\em Adv. Math.}, 139(2):293--309, 1998.

\bibitem[LS96]{LascouxSchutzenberger}
Alain Lascoux and Marcel-Paul Sch{\"u}tzenberger.
\newblock Treillis et bases des groupes de {C}oxeter.
\newblock {\em Electron. J. Combin.}, 3(2):Research paper 27, approx. 35 pp.
  (electronic), 1996.

\bibitem[Mat02]{Matousek}
Ji{\v{r}}{\'{\i}} Matou{\v{s}}ek.
\newblock {\em Lectures on discrete geometry}, volume 212 of {\em Graduate
  Texts in Mathematics}.
\newblock Springer-Verlag, New York, 2002.

\bibitem[MHPS12]{TamariFestschrift}
Folkert M{\"u}ller-Hoissen, Jean~Marcel Pallo, and Jim Stasheff, editors.
\newblock {\em Associahedra, {T}amari Lattices and Related Structures. Tamari
  Memorial Festschrift}, volume 299 of {\em Progress in Mathematics}.
\newblock Springer, New York, 2012.

\bibitem[MR95]{MalvenutoReutenauer}
Claudia Malvenuto and Christophe Reutenauer.
\newblock Duality between quasi-symmetric functions and the {S}olomon descent
  algebra.
\newblock {\em J. Algebra}, 177(3):967--982, 1995.

\bibitem[Nov00]{Novelli-hypoplactic}
Jean-Christophe Novelli.
\newblock On the hypoplactic monoid.
\newblock {\em Discrete Math.}, 217(1-3):315--336, 2000.
\newblock Formal power series and algebraic combinatorics (Vienna, 1997).

\bibitem[NT10]{NovelliThibon-decoratedFQSym}
Jean-Christophe Novelli and Jean-Yves Thibon.
\newblock Free quasi-symmetric functions and descent algebras for wreath
  products, and noncommutative multi-symmetric functions.
\newblock {\em Discrete Math.}, 310(24):3584--3606, 2010.

\bibitem[OEIS]{OEIS}
The {O}n-{L}ine {E}ncyclopedia of {I}nteger {S}equences.
\newblock Published electronically at \url{http://oeis.org}, 2010.

\bibitem[Pil13]{Pilaud-signedTreeAssociahedra}
Vincent Pilaud.
\newblock Signed tree associahedra.
\newblock Preprint,
  \href{http://arxiv.org/abs/1309.5222}{\texttt{arXiv:1309.5222}}, 2013.

\bibitem[Pos09]{Postnikov}
Alexander Postnikov.
\newblock Permutohedra, associahedra, and beyond.
\newblock {\em Int. Math. Res. Not. IMRN}, (6):1026--1106, 2009.

\bibitem[PR06]{PalaciosRonco}
Patricia Palacios and Mar{\'{\i}}a~O. Ronco.
\newblock Weak {B}ruhat order on the set of faces of the permutohedron and the
  associahedron.
\newblock {\em J. Algebra}, 299(2):648--678, 2006.

\bibitem[Pri13]{Priez}
Jean-Baptiste Priez.
\newblock A lattice of combinatorial {H}opf algebras, {A}pplication to binary
  trees with multiplicities.
\newblock Preprint
  \href{http://arxiv.org/abs/1303.5538}{\texttt{arXiv:1303.5538}}. Extended
  abstract in \emph{25th International Conference on Formal Power Series and
  Algebraic Combinatorics} (FPSAC'13, Paris), 2013.

\bibitem[Rea04]{Reading-latticeCongruences}
Nathan Reading.
\newblock Lattice congruences of the weak order.
\newblock {\em Order}, 21(4):315--344, 2004.

\bibitem[Rea05]{Reading-HopfAlgebras}
Nathan Reading.
\newblock Lattice congruences, fans and {H}opf algebras.
\newblock {\em J. Combin. Theory Ser. A}, 110(2):237--273, 2005.

\bibitem[Rea06]{Reading-CambrianLattices}
Nathan Reading.
\newblock Cambrian lattices.
\newblock {\em Adv.~Math.}, 205(2):313--353, 2006.

\bibitem[Rea15]{Reading-arcDiagrams}
Nathan Reading.
\newblock Noncrossing arc diagrams and canonical join representations.
\newblock {\em SIAM J. Discrete Math.}, 29(2):736--750, 2015.

\bibitem[RS09]{ReadingSpeyer}
Nathan Reading and David~E. Speyer.
\newblock Cambrian fans.
\newblock {\em J.~Eur.~Math.~Soc.}, 11(2):407--447, 2009.

\bibitem[SCc16]{SageCombinat}
The {S}age-{C}ombinat community.
\newblock {\em {S}age-{C}ombinat: enhancing {S}age as a toolbox for computer
  exploration in algebraic combinatorics}, 2016.
\newblock \url{http://wiki.sagemath.org/combinat}.

\bibitem[Sch61]{Schensted}
Craige Schensted.
\newblock Longest increasing and decreasing subsequences.
\newblock {\em Canad. J. Math.}, 13:179--191, 1961.

\bibitem[Sd16]{Sage}
The {S}age developers.
\newblock {\em {S}age {M}athematics {S}oftware}, 2016.
\newblock \url{http://www.sagemath.org}.

\bibitem[SS93]{ShniderSternberg}
Steve Shnider and Shlomo Sternberg.
\newblock {\em Quantum groups: From coalgebras to {D}rinfeld algebras}.
\newblock Series in Mathematical Physics. International Press, Cambridge, MA,
  1993.

\bibitem[Vie07]{Viennot}
Xavier Viennot.
\newblock Catalan tableaux and the asymmetric exclusion process.
\newblock In {\em 19th {I}nternational {C}onference on {F}ormal {P}ower
  {S}eries and {A}lgebraic {C}ombinatorics ({FPSAC} 2007)}. 2007.

\bibitem[Zel06]{Zelevinsky}
Andrei Zelevinsky.
\newblock Nested complexes and their polyhedral realizations.
\newblock {\em Pure Appl. Math. Q.}, 2(3):655--671, 2006.

\bibitem[Zie95]{Ziegler}
G{\"u}nter~M. Ziegler.
\newblock {\em Lectures on polytopes}, volume 152 of {\em Graduate Texts in
  Mathematics}.
\newblock Springer-Verlag, New York, 1995.

\end{thebibliography}
